\documentclass[letterpaper,11pt,reqno]{article}

\usepackage[margin=2.5cm]{geometry}

\usepackage{amsthm}
\usepackage{mathtools,amsfonts}
\usepackage{amssymb,amsmath,mathrsfs,xspace,multicol,stmaryrd}
\usepackage[mathscr]{eucal}
\usepackage{amscd}

\usepackage[all]{xy}
\usepackage[usenames,dvipsnames]{xcolor}

\usepackage[textsize=scriptsize,textwidth=2cm]{todonotes}

\usepackage{currfile}
\usepackage[ddmmyyyy]{datetime}

\usepackage[breaklinks=true,colorlinks=true,citecolor=black,urlcolor=black,linkcolor=black]{hyperref}

\usepackage[shortlabels]{enumitem}
\usepackage{xspace}
\usepackage{tikz}
\usepgflibrary{arrows}
\usetikzlibrary{matrix}
\usetikzlibrary{calc}
\usetikzlibrary{positioning}

\tikzset{%
  operad style/.style={fill=white,inner sep=2.5pt},
  baseline=(current  bounding  box.center),
  ampersand replacement=\&, row sep=2em,column sep=2em,
  std/.style={->,font=\scriptsize}
}

\tikzset{%
  clr style/.style={fill=white,inner sep=2.5pt},
  baseline=(current  bounding  box.center),
  ampersand replacement=\&, row sep=2em,column sep=2em,
  std/.style={->,font=\scriptsize}
}

\newenvironment{tikzdiag}[2]
{
\begin{tikzpicture}[clr style]
\matrix (m) [matrix of math nodes, row sep=#1em, column sep=#2em]
}
{
\end{tikzpicture}
}

\newcommand{\pbdiag}[1][.4]{\begin{scope}[shift=($(m-1-1)!{#1}!(m-2-2)$)]\draw +(-0.25,0) -- +(0,0)  -- +(0,0.25); \end{scope}}

\def\release#1{\let#1\undefined}

\newcommand{\myindent}{\hspace{.5cm}}

\newcommand{\itemqed}{\vspace{-2.0em}}

\newlist{pfitems}{itemize}{1}
\setlist[pfitems]{label=--}

\newcommand{\addtoc}[1]{\addcontentsline{toc}{subsection}{\textcolor{NavyBlue}{\und{#1}}}}

\makeatletter
\providecommand{\leftsquigarrow}{%
  \mathrel{\mathpalette\reflect@squig\relax}%
}
\newcommand{\reflect@squig}[2]{%
  \reflectbox{$\m@th#1\rightsquigarrow$}%
}
\makeatother

\newcommand{\Z}{\mathbb{Z}}
\newcommand{\kk}{\Bbbk}
\newcommand{\F}{\mathbb{F}}

\newcommand{\brac}{[ \cdot, \cdot ]}

\newcommand{\curv}{{\mathsf{curv}}}

\DeclareMathOperator{\MC}{\mathrm{MC}}
\DeclareMathOperator{\sMC}{\mathcal{MC}}

\DeclareMathOperator{\Hom}{\mathrm{Hom}}

\DeclareMathOperator{\id}{\mathrm{id}}

\newcommand{\op}{\mathrm{op}}

\newcommand{\cat}[1]{\mathsf{#1}}
\newcommand{\catC}{\cat{C}}

\newcommand{\Linf}{\mathsf{Lie}_\infty}

\newcommand{\Ch}{\mathsf{Ch}}

\newcommand{\Set}{\mathsf{Set}}
\newcommand{\sSet}{s\mathsf{Set}}
\newcommand{\Kan}{\mathsf{KanCplx}}

\DeclareMathOperator{\cdga}{\mathsf{cdga}}
\DeclareMathOperator{\Vect}{\mathsf{Vect}}

\newcommand{\maps}{\colon}

\newcommand{\xto}[1]{\xrightarrow{#1}}

\newcommand{\fib}{\rightarrow \mathrel{\mkern-14mu}\rightarrow}
\newcommand{\weq}{\xrightarrow{\sim}}
\newcommand{\trivfib}{\overset{\sim}{\twoheadrightarrow}}
\newcommand{\afib}{\trivfib}
\newcommand{\xfib}[1]{\xrightarrow{#1}\mathrel{\mkern-14mu}\rightarrow}
\newcommand{\xlfib}[1]{\leftarrow \mathrel{\mkern -14mu} \xleftarrow{#1}}
\newcommand{\heq}{\simeq}

\newcommand{\emb}{\hookrightarrow}

\newcommand{\adjunct}{\rightleftarrows}

\newcommand{\factor}[3]{ #1 \xrightarrow[{}^{\sim}]{i_{#2}} P(#2) \xfib{p_{#2}} #3}

\newcommand{\al}{\alpha}
\newcommand{\be}{\beta}
\newcommand{\Ga}{\Gamma}
\newcommand{\ga}{\gamma}
\newcommand{\Del}{\Delta}
\newcommand{\del}{\delta}
\newcommand{\ka}{\kappa}
\newcommand{\lam}{\lambda}
\newcommand{\ro}{\rho}
\newcommand{\Om}{\Omega}

\newcommand{\tha}{\theta}
\newcommand{\Tha}{\Theta}
\newcommand{\Ph}{\Phi}

\newcommand{\vph}{\varphi}
\newcommand{\vphi}{\varphi}

\newcommand{\si}{\sigma}

\newcommand{\ta}{\tau}

\newcommand{\jm}{\jmath}

\newcommand{\bev}[1]{\breve{#1}}

\newcommand{\bb}[1]{\vec{#1}}
\newcommand{\ba}[1]{\bar{#1}}
\newcommand{\ti}[1]{\tilde{#1}}
\newcommand{\ha}[1]{\hat{#1}}
\newcommand{\wh}[1]{\widehat{#1}}
\newcommand{\wti}[1]{\widetilde{#1}}
\newcommand{\und}[1]{{\underline{#1}}}
\newcommand{\ov}[1]{{\overline{#1}}}
\newcommand{\bul}{\bullet}

\newcommand{\cF}{\mathcal{F}}

\renewcommand{\iff}{if and only if\xspace}
\newcommand{\cc}{\circ}

\newcommand{\sse}{\subseteq}

\newcommand{\sss}{\supseteq}

\newcommand{\mtimes}[2]{\underbrace{#1  \cdots #1}_{\text{$#2$ {\rm times}}}}

\newcommand{\tensor}{\otimes}

\newcommand{\dsum}{\oplus}

\newcommand{\ev}{\mathrm{ev}}
\newcommand{\pr}{\mathrm{pr}}

\DeclareMathOperator{\chark}{\mathrm{char}}

\DeclareMathOperator{\ppr}{Pr}
\DeclareMathOperator{\im}{\mathrm{im}}

\renewcommand{\deg}[1]{\left \lvert #1 \right \rvert}

\newcommand{\Sh}{\mathrm{Sh}}

\renewcommand{\S}{\bar{S}}

\newcommand{\sQ}{Q^{1}}

\newcommand{\sPhi}{\Phi^{1}}
\newcommand{\pa}{\partial}
                 
\newcommand{\bs}{\mathbf{s}}

\DeclareMathOperator{\Gr}{\mathrm{gr}}

\DeclareMathOperator{\Path}{\mathrm{Path}}

\newcommand*{\Simp}[1]{\Delta^{#1}}

\newcommand{\horn}[2]{\Lambda^{#1}_{#2}}

\DeclareMathOperator{\tow}{\mathsf{tow}}
\DeclareMathOperator{\ttow}{\mathrm{tow}}

\newcommand{\proj}{\mathrm{proj}}
\newcommand{\plim}{\varprojlim}

\newcommand{\dR}{\mathrm{dR}}

\newlist{pfcases}{itemize}{1}
\setlist[pfcases]{label=--}
\newlist{pfsteps}{itemize}{1}
\setlist[pfsteps]{label={}}

\theoremstyle{plain}
\newtheorem{theorem}{Theorem}[section]
\newtheorem*{theorem*}{Theorem}
\newtheorem*{theorem1}{Theorem 1}
\newtheorem*{theorem2}{Theorem 2}
\newtheorem*{theorem3}{Theorem 3}

\newtheorem{proposition}[theorem]{Proposition}
\newtheorem{lemma}[theorem]{Lemma}
\newtheorem*{lemma*}{Lemma}
\newtheorem*{corollary*}{Corollary}

\newtheorem{corollary}[theorem]{Corollary}

\theoremstyle{definition}
\newtheorem{definition}[theorem]{Definition}

\newtheorem*{ass*}{Assumption}

\newtheorem{notation}[theorem]{Notation}

\newtheorem{liftprob}{Lifting Problem}
\newtheorem*{liftprob*}{Lifting Problem (prelude)}
\newtheorem*{lftprob*}{Lifting Problem}
\newtheorem*{notation*}{Notation}

\makeatletter
\newcommand{\neutralize}[1]{\expandafter\let\csname c@#1\endcsname\count@}
\makeatother

\newenvironment{liftprobbis}[1]
  {\renewcommand{\theliftprob}{\ref{#1}$'$}%
   \neutralize{liftprob}\phantomsection
   \begin{liftprob}}
  {\end{liftprob}}

\theoremstyle{remark}
\newtheorem{remark}[theorem]{Remark}

\numberwithin{equation}{section}

\renewcommand{\brac}[1]{ \{ \cdots  \}_{#1}}

\newcommand{\plimV}{\plim_{k} V/\cF_kV}

\newcommand{\qf}[2]{#1/\cF_{#2}#1}

\newcommand{\bp}{\ov{p}}
\newcommand{\hp}{\hat{p}}

\newcommand{\hV}{\widehat{V}}

\newcommand{\ctensor}{\widehat{\otimes}}
\newcommand{\fF}{\mathfrak{F}}

\newcommand{\cd}{\wh{d}}
\newcommand{\cQ}{\wh{Q}}

\renewcommand{\Ch}{\mathsf{Ch}^{\ast}}

\newcommand{\LdQ}[1]{(L^{#1},d^{#1},Q^{#1})}

\newcommand{\pp}[1]{(#1)}
\newcommand{\qq}[2]{#1_{(#2)}}
\newcommand{\qqa}[2]{#1_{(#2) \ast}}
\newcommand{\qpp}[1]{{\bp_{(#1)}}}
\newcommand{\qpr}[1]{{\bp'_{(#1)}}}

\newcommand{\bx}{\ba{x}}

\newcommand{\bcurv}[1]{{{\curv}_{(#1)}}}
\newcommand{\bcurvp}[1]{{{\curv'}_{(#1)}}}

\newcommand{\spQ}{Q^{\prime 1}}

\renewcommand{\catC}{\cat{C}}

\newcommand{\kkz}{\kk[z,dz]}
\newcommand{\pri}[1]{#1^{\prime}}

\newcommand{\sThe}{\Theta^1}
\newcommand{\sPsi}{\Psi^1}
\newcommand{\sPh}{\Ph^{1}_1}

\newcommand{\sMCq}[1]{\sMC (L / \cF_{#1}L)}
\newcommand{\sMCqq}[1]{\sMC (L' / \cF_{#1}L')}
\newcommand{\sMCr}[1]{\qq{\ro}{#1}}
\newcommand{\sMCrr}[1]{\qq{\ro'}{#1}}
\newcommand{\sMCp}[1]{\qq{\phi}{#1}}

\DeclareMathOperator{\tVect}{\tow(\Vect)}

\DeclareMathOperator{\cVect}{\widehat{\mathsf{Vect}}}

\DeclareMathOperator{\cCh}{\widehat{\mathsf{Ch}^\ast}}
\DeclareMathOperator{\fVect}{\mathsf{FiltVect}}
\DeclareMathOperator{\FLie}{\mathsf{FiltLie}\lbrack 1\rbrack_\infty}
\DeclareMathOperator{\Lie}{\mathsf{Lie}\lbrack 1\rbrack_\infty}

\DeclareMathOperator{\SLinf}{\mathsf{Lie}\lbrack 1\rbrack_\infty}
\DeclareMathOperator{\SLie}{\SLinf}

\DeclareMathOperator{\cLie}{\widehat{\mathsf{Lie}}\lbrack 1\rbrack_\infty}
\DeclareMathOperator{\cLinf}{\widehat{\mathsf{Lie}}_\infty}

\DeclareMathOperator{\strNLie}{{\mathsf{Lie}}\lbrack 1\rbrack^{\mathrm{nil/str}}_\infty}
\DeclareMathOperator{\nilstLie}{\strNLie}

\DeclareMathOperator{\strcLinf}{\widehat{\mathsf{Lie}}^{\mathrm{str}}_\infty}

\DeclareMathOperator{\bndfiltLie}{\widehat{\mathsf{Lie}}\lbrack 1\rbrack^{\mathrm{bd}}_\infty}
\DeclareMathOperator{\bdfltLie}{\bndfiltLie}
\DeclareMathOperator{\bdfltstLie}{\widehat{\mathsf{Lie}}\lbrack 1\rbrack^{\mathrm{bd/str}}_\infty}
\DeclareMathOperator{\bndfltstLie}{\bdfltstLie}

\DeclareMathOperator{\tbndfiltLie}{\tow(\bndfiltLie)}
\DeclareMathOperator{\tbdfltLie}{\tbndfiltLie}

\DeclareMathOperator{\tcurv}{\wti{\curv}}
\DeclareMathOperator{\str}{\mathrm{str}}
\DeclareMathOperator{\Cone}{\mathrm{Cone}}

 \title{
{\bf Complete $L_\infty$-algebras and their homotopy theory}
}

\renewcommand{\addtoc}[1]{\addcontentsline{toc}{section}{\textcolor{black}{\large \textbf{#1}}}}

\renewcommand\footnotemark{}
\renewcommand\footnotemark{}

\newcommand{\FF}{\mathbb{F}}

\newcommand{\J}{\mathfrak{I}}
\newcommand{\ctan}{\tan}
\newcommand{\cupp}{\smallsmile}

\newcommand{\sinf}{L[1]_{\infty}}
\newcommand{\cinf}{\widehat{L}[1]_{\infty}}
\newcommand{\df}[1]{\textbf{\textit{#1}}}

\author{Christopher L.\  Rogers 
\thanks{\hspace{-4.25ex}Department of Mathematics \& Statistics, University of Nevada,
  Reno.\ 1664 N.\ Virginia Street Reno, NV 89557-0084 USA. {\it Email:} chrisrogers@unr.edu, chris.rogers.math@gmail.com.}
}
\date{}
\begin{document}
\maketitle

\begin{abstract}
We analyze a model for the homotopy theory of complete filtered $L_\infty$-algebras 
intended for applications in algebraic and algebro-geometric deformation theory. We provide an explicit proof of an unpublished result of E.\ Getzler which states that the category $\cLinf$ of such $L_\infty$-algebras and filtration-preserving $\infty$-morphisms admits the structure of a category of fibrant objects (CFO) for a homotopy theory.
Novel applications of our approach include explicit models for homotopy pullbacks, and an analog of Whitehead's Theorem: under some mild conditions, every filtered $L_\infty$-quasi-isomorphism  in $\cLinf$ 
has a filtration preserving homotopy inverse. Also, we show that the simplicial Maurer--Cartan functor, which assigns a Kan simplicial set to each $L_\infty$-algebra in $\cLinf$, is an exact functor between the respective CFOs. Finally, we provide an obstruction theory for the general problem of lifting a Maurer-Cartan element through an $\infty$-morphism. The obstruction classes reside in the associated graded mapping cone of the corresponding tangent map.
\end{abstract}

\setcounter{tocdepth}{1}
\tableofcontents
\newpage

\section{Introduction} \label{sec:intro}
Complete filtered $L_\infty$-algebras, i.e.\ $\Z$-graded $L_\infty$-algebras $\bigl(L,d,\brac{k \geq 2} \bigr)$ equipped with a compatible complete descending filtration $L=\cF_1 L \sss \cF_2L \sss \cdots$ provide useful models for deformation problems over a field of characteristic zero in algebra and geometry. For example, see \cite{Pridham}, and the expository treatment \cite{Markl-book}. Two different categories of such complete $L_\infty$-algebras frequently appear in the literature. The first, which we denote as $\cLinf$, is the category whose morphisms are filtration-compatible weak $L_\infty$-morphisms (a.k.a.\  ``$\infty$-morphisms''). The second, $\strcLinf \sse \cLinf$, 
is the wide subcategory of the former whose morphisms are filtration-compatible strict $L_\infty$-morphisms. The weak morphisms in the larger first category are arguably more useful in applications, e.g.\ formal deformation quantization, and they are the ones that we focus on in this paper. 

The simplicial Maurer-Cartan functor $\sMC \maps \cLinf \to \Kan$ is a construction \cite{Getzler,Hinich:1997} that produces from any complete $L_\infty$-algebra a Kan simplicial set, or
$\infty$-groupoid. 
We say that a morphism between complete $L_\infty$-algebras $L \to L'$ in $\cLinf$ is a {\it weak equivalence}\footnote{See Sec.\ \ref{sec:hmpty-Linf} for the precise definition.} if the restriction of its linear term or ``tangent map'' to each piece of the filtration is a quasi--isomorphism of subcochain complexes $\cF_nL \weq \cF_nL'$.
In particular, every weak equivalence is an $L_\infty$ quasi-isomorphism.
In previous joint work \cite{GM_Theorem} with V.\ Dolgushev, we generalized a result of E.\ Getzler \cite{Getzler} and established a connection between the homotopy theory of complete $L_\infty$-algebras and Kan complexes:
\begin{theorem*}[Thm.\ 1.1 \cite{GM_Theorem}]
If $\Phi \maps L \to L'$ is a weak equivalence in $\cLinf$, then $\sMC(\Phi) \maps \sMC(L) \to \sMC(L')$
is a homotopy equivalence between Kan complexes.
\end{theorem*}
This is a generalization of the classical theorem of W.\ Goldman and J.\ Millson \cite[Thm.\ 2.4]{GM} from deformation theory, which was first brought into the $L_\infty$ context by Getzler for the special case of nilpotent $L_\infty$-algebras and {strict} $L_\infty$ quasi-isomorphisms \cite[Thm.\ 4.8]{Getzler}. The above theorem has turned out to be useful in a variety of applications beyond deformation theory, including the rational homotopy theory of mapping spaces \cite{Berglund,FTW}, and operadic homotopical algebra \cite{HAforms,RN:2020}.

The present work further develops and strengthens the above result.  We first give a finer-grained description of the homotopy theory within the category $\cLinf$. We then extend the aforementioned ``$L_\infty$ Goldman-Millson Theorem'' by establishing the compatibility of this additional homotopical structure with the Kan-Quillen model structure for simplicial sets. Finally, we use this machinery to build an obstruction theory for the general problem of lifting Maurer-Cartan elements through an arbitrary $\infty$-morphism.

\subsection{Main results}
Here we state our main results in terms of $L_\infty$-algebras, even though 
we work with  ``shifted'' $L_\infty$-algebras throughout the rest of the paper. The former are better known, but using the latter makes certain calculations more straightforward. The conceptual distinction is negligible, since we are  in the $\Z$-graded context.

In Section \ref{sec:hmpty-Linf}, we verify that the category $\cLinf$ of
complete $L_\infty$-algebras over a field $\kk$ with $\chark \kk=0$, and filtration-compatible $\infty$-morphisms form ``one-half of a model category'', i.e.\ a category of fibrant objects for a homotopy theory (Def.\ \ref{def:cfo}) in the sense of K.\ Brown \cite{Brown:1973}. 
We attribute the following statement to unpublished work\footnote{We learned of this statement during a conversation with E.\ Getzler at the ``Algebraic Analysis \& Geometry Workshop'' 
hosted by the University of Padua in the fall of 2013.}  of E.\ Getzler. Our proof, however, is novel, to the best of our knowledge:

\begin{theorem1}[Thm.\ \ref{thm:Linf_cfo}]\label{thm:A} 
The category $\cLinf$ admits the structure of a category of fibrant objects in which 
\begin{pfitems}[leftmargin=20pt]
\item a morphism $\Phi \maps \bigl(L,d,\brac{} \bigr) \to \bigl(\pri{L},d',\brac{}^{\prime}\bigr)$ is a weak equivalence if and only if
its tangent map $\ctan(\Ph) \maps (L, d) \to (L',d')$ induces for each $n \geq 1$ a quasi--isomorphism of cochain complexes
\[
\ctan(\Ph)\vert_{\cF_nL} \maps (\cF_{n}L, d) \weq (\cF_nL', d').
\]

\item a morphism $\Phi \maps \bigl(L,d,\brac{} \bigr) \to \bigl(\pri{L},d',\brac{}^{\prime}\bigr)$ is a fibration if and only if its tangent map $\tan(\Phi) \maps (L, d) \to (L',d')$ induces for each $n \geq 1$ an epimorphism of cochain complexes
\[
\tan(\Phi) \vert_{\cF_nL} \maps (\cF_{n}L, d) \fib (\cF_nL', d').
\]
\end{pfitems}
\end{theorem1}
The proof of the theorem is given in Sec.\ \ref{sec:Linf_cfo}. It builds on the 
following fact which we establish beforehand in Sec.\ \ref{sec:Ch_cfo}:
the CFO structure \cite{Hovey, Strick:2020} on the category $\Ch$ of $\Z$-graded cochain complexes over a field $\FF$ of any characteristic lifts to the category of complete filtered $\Z$-graded complexes $\cCh$. Our proof of Theorem 1 also provides explicit descriptions of pullbacks, and a ``strictification'' result for fibrations. The tractability of this approach relies on technical lemmas from B.\ Vallette's work \cite{Vallette:2014} on the homotopy theory of homotopy algebras. Furthermore, the category $\cLinf$ has a functorial path object, given by the completed tensor product $-\ctensor \Om_1$ with the cdga $\Om_1$ of polynomial de Rham forms on $\Delta^1$. 

Then, in Sec.\ \ref{sec:exact} we give an optimal homotopy-theoretic generalization of the classical Goldman--Millson theorem:
\begin{theorem2}[Thm.\ \ref{thm:sMC-exact}] \label{thm:C}
The simplicial Maurer-Cartan functor $\sMC \maps \cLinf \to \Kan$ is an exact functor between categories of fibrant objects (Def.\ \ref{def:exact_functor}). In particular:
\begin{enumerate}[leftmargin=20pt]
\item $\sMC$ sends weak equivalences in $\cLinf$ to weak homotopy equivalences between Kan simplicial sets.

\item $\sMC$ sends (acyclic) fibrations in $\cLinf$ to (acyclic) Kan fibrations. 

\item $\sMC$ sends pullbacks of (acyclic) fibrations in $\cLinf$ to pullbacks in $\Kan$.

\end{enumerate}
\end{theorem2}
Our proof of Theorem 2 follows the sketch we provided in the announcement \cite{Rogers:2018}.

\subsection{Applications} 
Our first application of the above results is an analogue of Whitehead's Theorem for weak homotopy equivalences. It follows from the abstract homotopical algebra developed in Sec.\ \ref{sec:cfo}.
\begin{theorem3}[Thm.\ \ref{thm:Linf-qinv}] \label{thm:B}
Every weak equivalence $\Phi \maps \bigl(L,d,\brac{} \bigr) \weq \bigl(\pri{L},d',\brac{}^{\prime}\bigr)$ in  $\cLinf$ has a homotopy inverse
in $\cLinf$ which can be exhibited by \und{any} choice of path objects for $L$ and $L'$.
In particular, there exists a weak equivalence $\Psi \maps \bigl(\pri{L},d',\brac{}^{\prime}\bigr) \weq \bigl(L,d,\brac{} \bigr)$ and homotopies 
\[
\mathcal{H}_L \maps  \bigl(L,d,\brac{} \bigr) \to \bigl (L \ctensor \Om_1, d_\Om, {\brac{}}_\Om \bigr ), \quad  \mathcal{H}_{L'} \maps \bigl(\pri{L},d',\brac{}^{\prime}\bigr) \to \bigl (L' \ctensor \Om_1, d'_\Om, {\brac{}^\prime}_\Om \bigr ), 
\] 
which induce equivalences $\Psi \Phi \simeq \id_{L}$  and $\Phi \Psi \simeq \id_{L'}$.
\end{theorem3}
The machinery in Sec.\ \ref{sec:hmtpy_pullback} also provides us with an explicit description of homotopy pullbacks in $\cLinf$. They can be used  as models for homotopy pullbacks of simplicial Maurer-Cartan sets thanks to the following corollary:
\begin{corollary*}[Cor.\ \ref{cor:MC-hpb}]
The functor $\sMC \maps \cLinf \to \Kan$ preserves homotopy pullbacks, up to homotopy equivalence.
\end{corollary*}
Lie models of homotopy pullbacks are often employed in the study of formal deformation problems within algebraic geometry e.g., \cite{Band-Man,Manetti}.
M.\ Matviichuk, B.\ Pym, and T.\ Schedler have already used our results for this purpose
in their work \cite{MPS} on the moduli space of log symplectic structures.

Another application which follows from Theorem 2 is the ``$\infty$-categorical'' uniqueness theorem for homotopy transferred structures that we sketched in \cite[Corollary 1]{Rogers:2018}. The statement  is roughly the following: Suppose we are given a cochain complex $A$, a
homotopy algebra $B$ of some particular type (e.g, an $A_\infty$,
$L_\infty$, or $C_\infty$-algebra) and a quasi-isomorphism of
complexes $\phi \maps A \to B$.  Then, using the simplicial
Maurer--Cartan functor, we can naturally produce an $\infty$-groupoid
$\fF$ whose objects correspond to solutions to the ``homotopy transfer
problem''. By a solution, we mean a pair consisting of a homotopy
algebra structure on $A$, and a lift of $\phi$ to a
$\infty$-quasi-isomorphism of homotopy algebras $A \weq B$. The fact
that $\sMC$ preserves both weak equivalences and fibrations allows
us to conclude that: (1) the $\infty$-groupoid $\fF$ is non--empty,
and (2) it is contractible. In other words, a homotopy equivalent
transferred structure always exists, and this structure is unique in the strongest possible sense.

Phrasing the homotopy transfer problem within the language of Maurer-Cartan elements reveals that it is a special case of a much more general question:
\begin{lftprob*}
Let $\Phi \maps \bigl(L,d,\brac{} \bigr) \to \bigl(\pri{L},d',\brac{}^{\prime}\bigr)$ be a morphism in $\cLinf$, and let $\al' \in \MC(L')$ be a Maurer-Cartan element in $L'$. Does there exist a Maurer-Cartan element $\al \in \MC(L)$ such that $\Ph(\al)$ is equivalent to $\al'$? 
\end{lftprob*}
\noindent By ``equivalent'', we mean that Maurer-Cartan elements $\Ph(\al)$ and $\al'$ lie in the same path component of $\sMC(L')$. (See Sec.\ \ref{sec:obstruct} for the precise definition.) This is the same relation as ``gauge equivalence'' when $L$ is a nilpotent dg Lie algebra. The above lifting problem lies behind a number of obstruction-theoretic arguments that appear in deformation theory and rational homotopy theory. In Sec.\ \ref{sec:obstruct}, we use the homotopical structure provided by Theorem 1 to characterize the obstructions to solving the lifting problem. We show that the obstruction classes reside in the cohomology of the associate graded mapping cone of the tangent map $\ctan(\Ph)$. See Theorem \ref{thm:lp} for the precise statement of this result.

\subsection{Related work}
There is a substantial amount of literature on the homotopy theory of $L_\infty$-algebras and the properties of the simplicial Maurer-Cartan functor. What follows is our attempt to carefully give attribution to these previous works, and clearly identify their relationship to the results presented in this paper.

\begin{itemize}[leftmargin=20pt]

\item The results of A.\ Lazarev \cite{Lazarev} and B.\ Vallette \cite{Vallette:2014} imply that the full subcategory of fibrant objects in V.\ Hinich's model category \cite{Hinich:2001} of $\Z$-graded conilpotent dg cocommutative coalgebras over $\kk$  is equivalent to the category $\Linf$ of $\Z$-graded $L_\infty$-algebras and $\infty$-morphisms. This gives $\Linf$ the structure of a CFO (with functorial path objects) in which the weak equivalences are $L_\infty$ quasi-isomorphisms and the fibrations are $L_\infty$-epimorphisms (see Def.\ \ref{def:Linf-maps-def}). This homotopical structure plays a crucial role in J.\ Pridham's unified approach to derived deformation theory \cite{Pridham}. 
When combined with our Theorem 1, these results imply that the
forgetful functor $\cLinf \to \Linf$ is an exact functor between CFOs.  

\item In \cite{{Hinich:1997}}, V.\ Hinich established statements (1) and (2) of Theorem 2 for the special case of strict morphisms between nilpotent dg Lie algebras.

\item As previously mentioned, in \cite{Getzler}, E.\ Getzler proved statements (1) and (2) of Theorem 2  for the special case of strict morphisms between nilpotent $L_\infty$-algebras.

\item In \cite{Yalin}, S.\ Yalin established statement (1) of Theorem 2
for the functor $\sMC \maps \strcLinf \to \Kan$, i.e., for the special case of strict filtration-compatible morphisms between complete $L_\infty$-algebras.
Yalin also proved in \cite[Thm.\ 4.2]{Yalin} statement (2) of Theorem 2 for the functor $\sMC \maps \strcLinf \to \Kan$, assuming certain finiteness conditions, as well as strictness of morphisms.

\item In \cite{Band}, R.\ Bandiera proved an analog of Theorem 2  for Getzler's Deligne $\infty$-groupoid functor $\ga_{\bul} \maps \strcLinf \to \Kan$, which can be thought of as a  ``smaller'' homotopy equivalent model for the simplicial Maurer-Cartan construction.
Similar results were also established by D.\ Robert-Nicoud in \cite{RN:2020}. However, we note that the functor $\ga_{\bul}$ is not well-defined on $\infty$-morphisms between $L_\infty$-algebras, and therefore does not extend to the category $\cLinf$.

\item A detailed analysis of the homotopy type of the simplicial Maurer-Cartan set $\sMC(L)$ and  
Deligne $\infty$-groupoid $\ga_{\bul}(L)$ for a complete $L_\infty$-algebra $L \in \strcLinf$ was given
by A.\ Berglund in \cite{Berglund}. Statement (1) of Theorem 2 
for the functor $\sMC \maps \strcLinf \to \Kan$, as well as the analogous statement for 
the functor $\ga_{\bul} \maps \strcLinf \to \Kan$ was proved in \cite[Prop.\ 5.4.]{Berglund}.
Statement (2) of Theorem 2 was also asserted \cite[Prop.\ 5.4.]{Berglund} for both functors $\sMC$ and $\ga_{\bul}$ with domain category $\strcLinf$, but the proof given there is unfortunately incomplete.

\item D.\ Robert-Nicoud and B. Vallette recently proved in \cite{RV:2020} that the Deligne $\infty$-groupoid functor $\ga_{\bul} \maps \strcLinf \to \Kan$ is representatable by a cosimplicial object in $\strcLinf$. Moreover, they equip $\strcLinf$ with a new model category structure \cite[Thm.\ 6.13]{RV:2020} which exhibits $\ga_\bul$ as the right adjoint in a Quillen pair $\sSet \adjunct \strcLinf$. Analogous results for complete dg Lie algebras were obtained by U.\ Buijs, Y.\ F\'{e}lix, A.\ Murillo, and D.\ Tanr\'{e} in their work \cite{BYMT} on Lie models in rational homotopy theory. We note that every strict weak equivalence in the sense of our Theorem 1 is a weak equivalence in the Robert-Nicoud--Vallette model structure. Although $\ga_\bul$ does not extend to a functor $\cLinf \to \Kan$, Robert-Nicoud and Vallette demonstrate in the same work \cite[Prop.\ 3.11]{RV:2020} that $\ga_{\bul}$ does induce a functor $\cLinf \to \cat{Ho(sSet)}$ which takes values in the homotopy category of simplicial sets. 

\item After a preliminary version of our paper appeared on the arXiv in 2020, D.\ Calaque, R.\ Campos, and J.\ Nuiten proved in \cite{CCN}  that our Theorems 1 and 2 can be recovered as a special case of their results on algebras over complete filtered operads.  

\end{itemize}

Finally, we note that the CFO structure explicitly described in Sec.\ \ref{sec:Ch_cfo} for the category $\cCh$ of complete filtered $\Z$-graded cochain complexes over a field  $\FF$ is equivalent to the one obtained by C.\ Di Natale \cite{DiNatale} by other methods.

\vspace{-.25cm}
\subsection*{Acknowledgments}
We thank Ezra Getzler for invaluable conversations  on the homotopy theory of $L_\infty$-algebras and related topics. The axioms \eqref{EA1} and \eqref{EA2} presented in Sec.\ \ref{sec:acyc} stemmed from discussions with Aydin Ozbek. We thank him for sharing his insights. We thank
Joost Nuiten for explaining to us his joint work \cite{CCN} with D.\ Calaque and R.\ Campos. We also thank Brent Pym for his comments on homotopy pullbacks of $L_\infty$-algebras. Finally, we are grateful for the careful reading of this manuscript by the anonymous referee, and we thank them for suggesting several improvements.

This paper is based upon work supported by the National Science Foundation under Grant No. DMS-1440140, while the author was in residence at the Mathematical Sciences Research Institute in Berkeley, California, during the Spring 2020 semester.
Additional support provided by a grant from the Simons Foundation/SFARI (585631,CR).

\section{Conventions} \label{sec:prelim}
Throughout, we work over a field. In Sec.\ \ref{sec:ch-htpy-thy}, the field $\FF$ may have arbitrary characteristic. From Sec.\ \ref{sec:Linf} onward, we set $\FF=\kk$ where  $\chark \kk =0$. 
All graded objects are assumed to be $\Z$-graded and, in general, unbounded. We use \textit{cohomological} conventions for all differential graded (dg) objects. We follow the conventions and notation of  
\cite[Sec.\ 1]{GM_Theorem} and \cite[Sec.\ 1.1.]{Enhanced} for graded linear algebra, Koszul signs, etc. In particular, for a $\Z$-graded vector space $V$, we denote by $\bs V$ (resp. by $\bs^{-1} V$) the suspension (resp. the desuspension) of $V$. In other words, 
\[
(\bs V)^{i}:=V[-1]^{i}:=V^{i-1} \qquad (\bs^{-1} V)^{i}:=V[1]^{i}:=V^{i+1}.
\]

Throughout, $\Vect$ denotes the category of $\Z$-graded vector spaces over $\FF$, and
$\Ch$ denotes the category of unbounded cochain complexes of $\FF$-vector spaces. 
We denote by $\Ch_\proj$ the category $\Ch$ equipped with the projective model structure 
\cite{Hovey, Strick:2020}. The weak equivalences in $\Ch_\proj$ are the quasi-isomorphisms, and the fibrations are those maps which are surjective in all degrees. Note that every object in $\Ch_\proj$ is fibrant. Finally, $\sSet$ denotes the category of simplicial sets, which we tacitly assume is equipped with the Kan-Quillen model structure, and $\Kan \sse \sSet$ denotes the full subcategory of Kan complexes as a category of fibrant objects \cite[Sec.\ I.9]{GJ}.

\section{Categories of fibrant objects} \label{sec:cfo}
In this section, we develop the abstract homotopical framework 
needed for our main results and their applications. The advantage of our approach is that
it is minimal and explicit. We are aiming  for a broad audience of ``end-users''. 
Hence, most of the exposition is self-contained, and the statements can be verified directly using basic category theory. With the exception of Sec.\ \ref{sec:towmodcat}, little knowledge of model categories and other methods in homotopical algebra are needed.

\begin{definition}[Sec.\ 1 \cite{Brown:1973}] \label{def:cfo}
Let $\cat{C}$ be a category with finite products, with terminal object $\ast
\in \cat{C}$, that is equipped with two classes of morphisms called
\df{weak equivalences} and \df{fibrations}. A morphism which
is both a weak equivalence and a fibration is called an
\df{acyclic fibration}. Then $\cat{C}$ is a
\df{category of fibrant objects (CFO)} for a homotopy theory \iff:
\begin{enumerate}
\item\label{item:cfoax1}{Every isomorphism in $\cat{C}$ is an acyclic fibration.}

\item\label{item:cfoax2}
{The class of weak equivalences satisfies  ``2 out of 3''. That is, if
    $f$ and $g$ are composable morphisms in $\cat{C}$ and any two of $f,g, g
    \circ f$ are weak equivalences, then so is the third.}

\item\label{item:cfoax3}{The composition of two fibrations is a fibration.}

\item\label{item:cfoax4}{The pullback of a fibration exists, and is a fibration.
That is, if $Y \xto{g} Z \xleftarrow{f} X$ is a diagram in $\cat{C}$ with $f$
    a fibration, then the pullback $X \times_{Z} Y$ exists, and
   the induced projection $X \times_{Z} Y \to Y$ is a  fibration.}

\item\label{item:cfoax5}{
The projection $X \times_Z Y \to Y$ in (4) is an acyclic fibration, whenever $f$ is an acyclic fibration. 
 }

\item\label{item:cfoax6}{For any object $X \in \cat{C}$ there exists a (not necessarily
    functorial) \df{path object}, that is, an object
    $X^{I}$ equipped with morphisms
\[
X \xto{s} X^{I} \xto{(d_0,d_1)} X \times X,
\]
such that $s$ is a weak equivalence, $(d_0,d_1)$ is a fibration, and their
composite is the diagonal map.}

\item\label{item:cfoax7}{All objects of $\cat{C}$ are \df{fibrant}. That is, for any $X \in \cat{C}$ the unique map 
$ X \to \ast$ is a fibration.}
\end{enumerate}
\end{definition}
\noindent Throughout the paper, we frequently write $X \weq Y$ to denote a weak equivalence, and $X \fib Y$ to denote a fibration between objects $X$ and $Y$ in a CFO.

\subsection{Factorization and right homotopy equivalence} \label{sec:factor}
The axioms of a CFO imply Brown's Factorization Lemma \cite[Sec.\ 1]{Brown:1973}: Every morphism $f \maps X \to Y$ in $\catC$ can be factored as 
$f= p \cc \jm$, where $p$ is a fibration and $\jm$ is a right inverse of an acyclic fibration (and hence a weak equivalence). Let us briefly recall the construction. Let $Y \xto{s} Y^{I} \xto{(d_0,d_1)} Y \times Y$ be a path object for $Y$, and let $X \times_Y Y^I$ be the pullback in the diagram
\begin{equation} \label{diag:factor}
\begin{tikzdiag}{2}{2}
{
X \times_Y Y^I \& Y^I \\
X \& Y\\
};

\path[->,font=\scriptsize]
(m-1-1) edge node[auto] {$\pr_2$} (m-1-2) 
(m-2-1) edge node[auto,swap] {$f$} (m-2-2)
;
\path[->>,font=\scriptsize]
(m-1-1) edge node[auto,swap] {$\pr_1$} node[sloped,above] {$\sim$}(m-2-1) 
(m-1-2) edge node[auto,swap] {$d_0$} node[sloped,above] {$\sim$}  (m-2-2)
;

\pbdiag
\end{tikzdiag}
\end{equation}
Then the commutative diagram
\[
\xymatrix{
X  \ar[d]_{\id_X} \ar[r]^{sf} & Y^I \ar[d]^{d_0}\\
X \ar[r]^{f} & Y \\
}
\]
provides a unique weak equivalence $\jm \maps X \weq X \times_{Y}Y^{I}$ such that $\pr_1 \cc \jm = \id_X$.
Finally, define $p \maps X \times_{Y}Y^{I} \to Y$ to be the composition $p:= d_1 \circ \pr_2$. See \cite{Brown:1973}, or the proof of \cite[Lemma 2.3]{Rogers-Zhu:2020} for the verification that $p$ is indeed a fibration.

\begin{remark}\label{rmk:span}
The above construction combined with axiom \eqref{item:cfoax2} in Def.\ \ref{def:cfo} implies that if $\catC$ is a CFO, then any two objects $X,Y \in \catC$ connected by a weak equivalence $f \maps X \weq Y$ are connected by a span of acyclic fibrations
\[
\begin{tikzdiag}{1}{2.5}
{
X \& X \times_Y{Y^I} \& Y \\
};

\path[->>,font=\scriptsize]
(m-1-2) edge node[auto,swap] {$\pr_1$} node[sloped,below] {$\sim$}(m-1-1)
(m-1-2) edge node[auto] {$p$} node[sloped,below] {$\sim$}(m-1-3)
;
\end{tikzdiag}
\]
\end{remark}

We recall from \cite[Sec.\ 2]{Brown:1973} that two morphisms $f,g \maps X \to Y$ in a CFO are \df{(right) homotopic} \iff there exists a path 
object $Y \xto{s} Y^{I} \xto{(d_0,d_1)} Y \times Y$ and a morphism $h \maps X \to Y^I$ such that $f= d_0h$ and $g=d_1h$. In this case we write ``$ f \simeq g$''. In analogy with the situation in a model category, one shows that homotopy in a CFO is an {equivalence relation} by taking iterated  pullbacks of path objects. See for example the dual of \cite[Lemma 4]{Quillen:1967}.\\

The next proposition, which we recall from \cite{Brown:1973}, describes how homotopy behaves under pre and post-composition. It will be useful for the next section.

\begin{proposition}[Prop.\ 1 \cite{Brown:1973}]\label{prop:rhmtpy}
Let $\catC$ be a category of fibrant objects. Assume $\al, \be \maps A \to B$ are homotopic morphisms in $\catC$.
\begin{enumerate}
\item If $\mu \maps Z \to A$ is an arbitrary morphism, then $\al \mu \heq \be \mu$. 
\item If $\nu \maps B \to C$ is an arbitrary morphism, and $C^I$ is any path object for $C$, then there exists an acyclic fibration $\phi \maps Z \afib A$ such that $\nu \al \phi \heq \nu \be \phi$
 via a homotopy $\ti{h} \maps Z \to C^I.$
\end{enumerate}
\end{proposition}

\subsection{Categories of ``bifibrant objects''} \label{sec:acyc}
We now consider a CFO $\catC$ which satisfies two additional axioms:
\begin{enumerate}[label={(EA\arabic*)}, ref=EA\arabic*, start=1]
\item \label{EA1} Every acyclic fibration in $\catC$ has a right inverse.

\item \label{EA2} $\catC$ is equipped with
\begin{itemize}[leftmargin=12pt]
\item \df{functorial path objects}: An assignment of a path object  
\[
X \quad \mapsto \quad  X \xto{s^X} \Path(X) \xto{(d^X_0,d^X_1)} \Path(X) \times \Path(X)
\]
to each object $X \in \catC$, and to each $f \maps X \to Y$, a morphism $f^I \maps \Path(X) \to \Path(Y)$ such that the following diagram commutes:
\[
\begin{tikzpicture}[descr/.style={fill=white,inner sep=2.5pt},baseline=(current  bounding  box.center)]
\matrix (m) [matrix of math nodes, row sep=2em,column sep=3em,
  ampersand replacement=\&]
  {  
X \& \Path(X) \& X \times X \\
Y \& \Path(Y) \& Y \times Y \\
}; 
  \path[->,font=\scriptsize] 
   (m-1-1) edge node[auto] {$s^X$} node[auto,below] {$\sim$} (m-1-2)
   (m-2-1) edge node[auto] {$s^Y$} node[auto,below] {$\sim$} (m-2-2)
   (m-1-1) edge node[auto] {$f$} (m-2-1)
   (m-1-2) edge node[auto] {$f^I$} (m-2-2)
   (m-1-3) edge node[auto] {$(f,f)$} (m-2-3)
  ;
  \path[->>,font=\scriptsize] 
   (m-1-2) edge node[auto] {$(d^X_0,d^X_1)$} (m-1-3)
   (m-2-2) edge node[auto] {$(d^Y_0,d^Y_1)$} (m-2-3)
;
\end{tikzpicture}
\]

\item \df{functorial pullbacks of fibrations:} 
An assignment to each diagram of the form  $X \xto{g} Z \xleftarrow{f} Y$, in which $f$ is a fibration, a universal cone  $X \xleftarrow{ \tilde{f}} X \times_Z Y \xto{\tilde{g}} Y$. The pullback exists by the CFO axioms, and $\tilde{f}$ is necessarily a fibration.
The universal property implies that this assignment functorial.
\end{itemize}
\end{enumerate}
\begin{remark}\label{rmk:bifib}
\mbox{}
\begin{enumerate}[leftmargin=15pt]
\item Axiom \eqref{EA1} can be understood as saying that all objects in $\catC$ are cofibrant, as well as fibrant. Hence, the term ``bifibrant'' appearing in the title of this section.

\item Axiom \eqref{EA2} implies the existence of \df{functorial factorizations}: 
Each morphism $f \maps Y \to Z$ in $\catC$ can be canonically factored as
\begin{equation} \label{eq:func-fact0}
Y \xto{i_f} P(f) \xto{p_f} Z
\end{equation}

where $P(f)$ is the functorial pullback $Y \times_Z \Path(Z)$ of $f$ along $d_0$ as in diagram \eqref{diag:factor}, and $p_f$ and $i_f$ are the associated fibration and weak equivalence, respectively. The factorization is natural: Given a commutative diagram 
\[
\begin{tikzpicture}[descr/.style={fill=white,inner sep=2.5pt},baseline=(current  bounding  box.center)]
\matrix (m) [matrix of math nodes, row sep=2em,column sep=3em,
  ampersand replacement=\&]
  {  
Y \& Z  \\
Y' \& Z' \\
}; 
  \path[->,font=\scriptsize] 
   (m-1-1) edge node[auto] {$f$} (m-1-2)
   (m-2-1) edge node[auto] {$f'$} (m-2-2)
   (m-1-1) edge node[auto,swap] {$\be$} (m-2-1)
   (m-1-2) edge node[auto] {$\ga$} (m-2-2)
  ;
\end{tikzpicture}
\]

there exists a unique morphism $\ro \maps P(f) \to P(f')$ such that the diagram
\begin{equation} \label{eq:func-fact}
\begin{tikzpicture}[descr/.style={fill=white,inner sep=2.5pt},baseline=(current  bounding  box.center)]
\matrix (m) [matrix of math nodes, row sep=2em,column sep=3em,
  ampersand replacement=\&]
  {  
Y \& P(f) \&  Z  \\
Y' \& P(f') \&  Z'  \\
}; 
  \path[->,font=\scriptsize] 
  (m-1-1) edge node[auto] {$i_f$} node[auto,below] {$\sim$} (m-1-2)
  (m-2-1) edge node[auto] {$i_{f'}$} node[auto,below] {$\sim$} (m-2-2)
  (m-1-1) edge node[auto] {$\be$} (m-2-1)
  (m-1-2) edge node[auto] {$\rho$} (m-2-2)
  (m-1-3) edge node[auto] {$\ga$} (m-2-3)
   ;

 \path[->>,font=\scriptsize] 
  (m-1-2) edge node[auto] {$p_f$} (m-1-3)
  (m-2-2) edge node[auto] {$p_{f'}$} (m-2-3)
;
\end{tikzpicture}
\end{equation}
commutes.

\item We will show in Sec.\ \ref{sec:hmpty-Linf} that the category of complete filtered shifted $L_\infty$-algebras forms a CFO which satisfies axioms \eqref{EA1} and \eqref{EA2}.
Relevant examples of CFOs from the literature which also satisfy these axioms include: the category $\Ch_\proj$ of $\Z$-graded cochain complexes over a field $\F$  equipped with the projective model structure \cite{Hovey, Strick:2020}; the subcategory of fibrant objects in V.\ Hinich's model category on $\Z$-graded conilpotent dg cocommutative coalgebras \cite{Hinich:2001} and the category of Kan simplicial sets $\Kan$ equipped with the CFO structure induced by D.\ Quillen's model structure on $\sSet$ \cite{Quillen:1967}.

\end{enumerate}
\end{remark}

A more noteworthy fact is that our additional axioms imply that every homotopy equivalence can be realized via the functorial path object.

\begin{proposition}\label{prop:bifib-htpy}
Let $\catC$ be a CFO satisfying axioms \eqref{EA1} and \eqref{EA2}. Morphisms $\al,\be \maps A \to B$ are homotopic in $\catC$ \iff there exists a homotopy $h \maps A \to \Path(B)$ such that $\al = d^B_0 h$ and $\be = d^B_1 h$.
\begin{proof}
Suppose $\al \heq \be \maps A \to B$. In statement (2) of Prop.\ \ref{prop:rhmtpy} take $C=B$, 
$\nu = \id_B$ and $C^I = \Path(B)$. Then there exists an acyclic fibration $\phi \maps Z \afib A$ and a homotopy $\ti{h} \maps Z \to \Path(B)$ which gives a homotopy equivalence
$\al \phi \heq \be \phi$. Axiom \eqref{EA1} implies that there exists a morphism $\si \maps A \to Z$ such that $\phi \si = \id_Z$. Define $h \maps A \to \Path(B)$ to be the composition $h:=\ti{h}\si \maps  A \to \Path(B)$. Then, as asserted in statement (1) of Prop.\ \ref{prop:rhmtpy}, $h$ gives the desired homotopy equivalence $\al \heq \be$ via the path object $\Path(B)$.  
\end{proof}
\end{proposition}

Next, we show that axioms \eqref{EA1} and \eqref{EA2} imply that every weak equivalence has a homotopy inverse.

\begin{proposition}\label{prop:bifib-invert}
Let $\catC$ be a CFO satisfying axioms \eqref{EA1} and \eqref{EA2}. Suppose $f \maps X \weq Y$ is a weak equivalence in $\catC$. Then there exist a weak equivalence $g \maps Y \weq X$
and homotopies 
\[
h_X \maps X \to \Path(X), \quad h_Y \maps Y \to \Path(Y)
\]
which induce equivalences $gf \heq \id_X$, and $fg \heq \id_Y$, respectively.
\end{proposition}
\begin{proof}
Let $f \maps X \weq Y$ be a weak equivalence. Note that Prop.\ \ref{prop:bifib-htpy} implies that it is sufficient to exhibit a two-sided homotopy inverse to $f$ using any path objects. Choose a factorization $f=p_f \jm_f$ as constructed in Sec.\ \ref{sec:factor}. As discussed in Remark \ref{rmk:span}, we have a span of acyclic fibrations of the form
\[
\begin{tikzdiag}{2}{2.5}
{
X \& X \times_Y Y^I \& Y \\
};

\path[->>,font=\scriptsize]
(m-1-2) edge node[auto,swap] {$\pr_1$} node[below] {$\sim$}  (m-1-1)
(m-1-2) edge node[auto] {$d_1\pr_2$} node[below] {$\sim$}  (m-1-3)
;
\end{tikzdiag}
\]
where $X \times_Y Y^I$, $\pr_1$ and $\pr_2$ are the pullback and projections from diagram \eqref{diag:factor}. Axiom \eqref{EA1} implies that there exists a section $\si \maps Y \to X \times_Y Y^I$ such that $d_1 \pr_2 \si = \id_Y$. Define $g \maps Y \weq X$ and $h \maps Y \to Y^I$ to be the morphisms $g:=\pr_1  \si$, and  $h:=\pr_2  \si$, respectively.
Then the commutativity of diagram \eqref{diag:factor} implies that $d_0h = f \pr_1 \si = fg$,
and by construction, we have $d_1h = d_1 \pr_2 \si=\id_Y$. Hence
\begin{equation} \label{eq:bifib1}
fg \heq \id_Y
\end{equation}

Next, we repeat the above steps {\it mutatis mutandis} for the weak equivalence $g \maps Y \weq X$. This produces another weak equivalence $\ti{f} \maps X \weq Y$ such that $g \ti{f} \heq \id_X$. 
We proceed by using the standard tricks for this type of situation. Choose a path object $Y^I$ for $Y$, and in statement (2) of Prop.\ \ref{prop:rhmtpy}, take $\nu = f \maps X \weq Y$, $\al=g \ti{f}$, and $\be = \id_X$. Then there exists an acyclic fibration $\phi \maps Z \afib X$ and a homotopy $Z \to B^I$ which gives an equivalence
$f g\ti{f} \phi \heq f \phi$. Axiom \eqref{EA1} implies that $\phi$ has a right inverse, hence from Prop.\ \ref{prop:rhmtpy}(1) we deduce that $f g\ti{f} \heq f$. On the other hand, Prop.\ \ref{prop:rhmtpy}(1)
combined with the equivalence \eqref{eq:bifib1} implies that $f g\ti{f} \heq \ti{f}$. Hence, $f \heq \ti{f}$, since right homotopy equivalence is an equivalence relation.

Finally, we apply Prop.\ \ref{prop:rhmtpy}(2) again. This time we set  $\nu = g \maps Y \weq X$, $\al = f$ and $\be = \ti{f}$. We obtain $gf \psi \heq g \ti{f} \psi$, for some acyclic fibration $\psi$. We compose this with a right inverse of $\psi$ to obtain $gf  \heq g \ti{f}$, and therefore we conclude that $gf \heq \id_X.$
\end{proof}

\subsection{Homotopy pullbacks} \label{sec:hmtpy_pullback}
Throughout this subsection, $\catC$ denotes a CFO satisfying Axiom \eqref{EA1} and Axiom \eqref{EA2}.
\begin{definition}[cf.\ Sec.\ 13.3.1 \cite{Hirschhorn}] \label{def:hmtpy_pb}
Let $g \maps X \to Z$ and $f \maps Y \to Z$ be morphisms in $\catC$. Let  
\[
\factor{X}{g}{Z}, \qquad \factor{Y}{f}{Z}
\]
be the functorial factorizations of $g$ and $f$, respectively, as described in Rmk.\ \ref{rmk:bifib}. The \df{homotopy pullback} of the diagram $X \xto{g} Z \xleftarrow{f} Y$ is the pullback $P(g) \times_Z P(f)$ of the diagram\\ $P(g) \xfib{p_g} Z \xlfib{p_f} P(f).$
\end{definition}
Note that the homotopy pullback  always exists by Axiom 4 of Def.\ \ref{def:cfo}, even if the pullback of the original diagram does not. Homotopy pullbacks in $\catC$ enjoy a number of useful properties, similar to those exhibited in a proper model category. 
In particular, as we demonstrate in the following proposition, the homotopy pullbacks of $X \xto{g} Z \xleftarrow{f} Y$ can be constructed, up to homotopy, by any factorizations of $g$ and $f$ into weak equivalences followed by fibrations, not just the functorial ones used in Def.\ \ref{def:hmtpy_pb}. It is, in fact, sufficient to factor only one of the morphisms in the original diagram. Moreover, we also have homotopy invariance: if two diagrams are (point-wise) weakly equivalent, then their respective homotopy pullbacks are weakly equivalent.

\begin{proposition} \label{prop:hpb}
Let $g \maps X \to Z$ and $f \maps Y \to Z$ be morphisms in $\catC$.
\begin{enumerate}[leftmargin=15pt]

\item If $X \xto{j_g} \ti{X} \xto{\ti{g}} Z$ and $Y \xto{j_f} \ti{Y} \xto{\ti{f}}Z$ are factorizations of $g$ and $f$, respectively, into weak equivalences followed by fibrations, then there is a weak equivalence $X \times_Z \ti{Y} \xto{\sim} \ti{X} \times_Z \ti{Y}$
between the pullback of the diagram $X \xto{g} Z \xleftarrow{\ti{f}} \ti{Y}$ and  the pullback of the diagram $\ti{X} \xto{\ti{g}} Z \xleftarrow{\ti{f}} \ti{Y}$. In particular, we have
$X \times_Z P(f) \simeq P(g) \times_Z P(f)$.

\item Given a commutative diagram of the form 
\[
\begin{tikzdiag}{2}{2}
{
X \& Z  \& Y \\
X' \& Z'  \& Y' \\
};
\path[->,font=\scriptsize]
(m-1-1) edge node[auto] {$g$} (m-1-2)
(m-1-3) edge node[auto,swap] {$f$} (m-1-2)
(m-2-1) edge node[auto] {$g'$} (m-2-2)
(m-2-3) edge node[auto,swap] {$f'$} (m-2-2)
(m-1-1) edge node[auto,swap] {$\al$} node[above,sloped]{$\sim$}  (m-2-1)
(m-1-2) edge node[auto,swap] {$\ga$} node[above,sloped]{$\sim$} (m-2-2)
(m-1-3) edge node[auto,swap] {$\be$} node[above,sloped]{$\sim$} (m-2-3)
;
\end{tikzdiag}
\]
in which all vertical morphisms are weak equivalences, the morphism $\ro \maps P(f) \to P(f')$ arising from the functorial factorization \eqref{eq:func-fact} induces a weak equivalence
\[
(\al,\ro) \maps X \times_Z P(f) \weq X'\times_{Z'}P(f').
\]
In particular, the homotopy pullback of the diagram 
$X \xto{g} Z \xleftarrow{f} Y$ is homotopy equivalent to the homotopy pullback of the diagram
$X' \xto{g'} Z' \xleftarrow{f'} Y'$.

\item If $Y \xto{j_f} \ti{Y} \xto{\ti{f}}Z$ is a factorization $f$, as in statement (1),  then there is a weak equivalence  $X \times_Z P(f) \xto{\sim} X \times_Z \ti{Y}$.
In particular, the homotopy pullback of the diagram $X \xto{g} Z \xleftarrow{f} Y$ is homotopy equivalent to the pullback of the diagram $X \xto{g} Z \xleftarrow{\ti{f}} \ti{Y}$.
\end{enumerate}
\end{proposition}

For the proof, we will need a lemma:

\begin{lemma} \label{lem:hpb}
\mbox{}
\begin{enumerate}[leftmargin=15pt]

\item (``Co-gluing Lemma'') Given a commutative diagram in $\catC$
\[
\begin{tikzdiag}{2}{2}
{
A \& C  \& B \\
A' \& C'  \& B' \\
};
\path[->,font=\scriptsize]
(m-1-1) edge node[auto] {$$} (m-1-2)
(m-2-1) edge node[auto] {$$} (m-2-2)
(m-1-1) edge node[auto,swap] {$$} node[above,sloped]{$\sim$}  (m-2-1)
(m-1-2) edge node[auto,swap] {$$} node[above,sloped]{$\sim$} (m-2-2)
(m-1-3) edge node[auto,swap] {$$} node[above,sloped]{$\sim$} (m-2-3)
;
\path[->>,font=\scriptsize]
(m-2-3) edge node[auto,swap] {$$} (m-2-2)
(m-1-3) edge node[auto,swap] {$$} (m-1-2)
;
\end{tikzdiag}
\]
in which both right hand side horizontal morphisms are fibrations, and in which  all vertical morphisms are weak equivalences, the induced morphism between pullbacks $A \times_C B \to A' \times_{C'} B'$ is a weak equivalence.

\item If $f \maps Y \fib Z$ is a fibration, then for all morphisms $g \maps X \to Z$ in $\catC$, the weak equivalence $i_f \maps Y \weq P(f)$ from the functorial factorization \eqref{eq:func-fact0} induces a weak equivalence between pullbacks 
\[
(\id_X,i_f) \maps X \times_Z Y \weq X \times_Z P(f).
\] 
\end{enumerate}
\end{lemma}
\begin{proof}
Statement (1) is the ``co-gluing lemma'' of \cite[Lem.\ 8.10]{GJ} which holds in any CFO.
For statement (2), the factorization of the fibration $f \maps Y \to Z$ fits into a commutative diagram
\[
\begin{tikzdiag}{2}{2}
{
X \& Z  \& Y \\
X \& Z  \& P(f) \\
};
\path[->,font=\scriptsize]
(m-1-1) edge node[auto] {$g$} (m-1-2)
(m-2-1) edge node[auto] {$g$} (m-2-2)
(m-1-1) edge node[auto,swap] {$\id_X$}  (m-2-1)
(m-1-2) edge node[auto,swap] {$\id_Z$}  (m-2-2)
(m-1-3) edge node[auto,swap] {$i_Y$} node[above,sloped]{$\sim$} (m-2-3)
;
\path[->>,font=\scriptsize]
(m-1-3) edge node[auto,swap] {$f$} (m-1-2)
(m-2-3) edge node[auto,swap] {$p_{f}$} (m-2-2)
;
\end{tikzdiag}
\]  
Applying the co-gluing lemma then completes the proof.
\end{proof}


\begin{proof}[Proof of Prop.\ \ref{prop:hpb}]
All three statements concern the existence of homotopy equivalences. It is therefore sufficient to
exhibit (zig-zags of) weak equivalences in each case, thanks to Prop.\ \ref{prop:bifib-invert}.
\begin{enumerate}[leftmargin=15pt]
\item Factorizations $X \xto{j_g} \ti{X} \xto{\ti{g}} Z$ and $Y \xto{j_f} \ti{Y} \xto{\ti{f}}Z$ of $g$ and $f$ give a commutative diagram
\[
\begin{tikzdiag}{2}{2}
{
X \& Z  \& \ti{Y} \\
\ti{X} \& Z  \& \ti{Y} \\
};
\path[->,font=\scriptsize]
(m-1-1) edge node[auto] {$g$} (m-1-2)
(m-1-1) edge node[auto,swap] {$j_g$} node[above,sloped]{$\sim$}  (m-2-1)
(m-1-2) edge node[auto] {$\id_Z$}  (m-2-2)
(m-1-3) edge node[auto] {$\id_{\ti{Y}}$}  (m-2-3)
;
\path[->>,font=\scriptsize]
(m-2-3) edge node[auto,swap] {$\ti{f}$} (m-2-2)
(m-1-3) edge node[auto,swap] {$\ti{f}$} (m-1-2)
(m-2-1) edge node[auto] {$\ti{g}$} (m-2-2)
;
\end{tikzdiag}
\]
which satisfies the hypotheses of Lemma \ref{lem:hpb}(1). Hence, the induced map between pullbacks gives a weak equivalence $X \times_Z \ti{Y} \xto{\sim} \ti{X} \times_Z \ti{Y}$. 

\item Apply the functorial factorization as in \eqref{eq:func-fact} to the given diagram to
obtain a new commutative diagram:
\[
\begin{tikzdiag}{2}{2}
{
X \& Z  \& P(f) \\
X' \& Z'  \& P(f') \\
};
\path[->,font=\scriptsize]
(m-1-1) edge node[auto] {$g$} (m-1-2)
(m-2-1) edge node[auto] {$g'$} (m-2-2)
(m-1-1) edge node[auto,swap] {$\al$} node[above,sloped]{$\sim$}  (m-2-1)
(m-1-2) edge node[auto,swap] {$\ga$} node[above,sloped]{$\sim$} (m-2-2)
(m-1-3) edge node[auto,swap] {$\ro$} node[above,sloped]{$\sim$} (m-2-3)
;
\path[->>,font=\scriptsize]
(m-1-3) edge node[auto,swap] {$p_f$} (m-1-2)
(m-2-3) edge node[auto,swap] {$p_{f'}$} (m-2-2)
;
\end{tikzdiag}
\]
By the ``2 out of 3'' axiom, since $i_{f'} \cc \be$ and $i_f$ are weak equivalences, $\ro$ is as well. Hence, the cogluing lemma implies that $(\al,\ro)$ is a weak equivalence. Statement (1) of the proposition implies that this is a weak equivalence between the respective homotopy pullbacks. 

\item Given a factorization $Y \xto{j_f} \ti{Y} \xto{\ti{f}}Z$ of $f \maps Y \to Z$, consider the commutative diagram
\[
\begin{tikzdiag}{2}{2}
{
X \& Z  \& {Y} \\
{X} \& Z  \& \ti{Y} \\
};
\path[->,font=\scriptsize]
(m-1-1) edge node[auto] {$g$} (m-1-2)
(m-1-1) edge node[auto,swap] {$\id_g$}   (m-2-1)
(m-1-2) edge node[auto] {$\id_Z$}  (m-2-2)
(m-1-3) edge node[auto] {$j_{f}$} node[below,sloped]{$\sim$}  (m-2-3)
(m-1-3) edge node[auto,swap] {${f}$} (m-1-2)
(m-2-1) edge node[auto] {${g}$} (m-2-2)
;
\path[->>,font=\scriptsize]
(m-2-3) edge node[auto,swap] {$\ti{f}$} (m-2-2)
;
\end{tikzdiag}
\]
Statement (2) of the proposition yields a weak equivalence $X \times_Z P(f) \weq X \times_{Z}P(\ti{f})$,
where $P(\ti{f})$ arises from the functorial factorization $\ti{Y} \xto{i_{\ti{f}}} P(\ti{f}) \xto{p_{\ti{f}}} Z$ of the fibration $\ti{f} \maps \ti{Y} \fib Z$. Lemma \ref{lem:hpb}(2) then implies that $(\id_X,i_{\ti{f}}) \maps X \times_Z \ti{Y} \weq X \times_Z P(\ti{f})$ is a weak equivalence. We therefore obtain a zig-zag
\[
X \times_Z P(f) \weq X \times_{Z}P(\ti{f})  \xleftarrow{\sim} X \times_Z \ti{Y} 
\]
which provides the desired homotopy equivalence.
\end{enumerate}
\itemqed
\end{proof}

\subsection{Exact functors and homotopy pullbacks} \label{sec:exact_func}
The standard way to compare categories of fibrant objects is via left exact functors, which are the analog of right Quillen functors between model categories.  We recall the definition: 

\begin{definition} \label{def:exact_functor}
A functor $F \maps \cat{C} \to \cat{D}$ between categories of fibrant objects is a \df{(left) exact functor} \iff
\begin{enumerate}
\item $F$ preserves the terminal object, fibrations, and acyclic fibrations.
\item Any pullback square in $\cat{C}$ of the form
\[
\begin{tikzpicture}[descr/.style={fill=white,inner sep=2.5pt},baseline=(current  bounding  box.center)]
\matrix (m) [matrix of math nodes, row sep=2em,column sep=3em,
  ampersand replacement=\&]
  {  
P \& X \\
Z \& Y \\
}
; 
  \path[->,font=\scriptsize] 
   (m-1-1) edge node[auto] {$$} (m-1-2)
   (m-1-1) edge node[auto,swap] {$$} (m-2-1)
   (m-1-2) edge node[auto] {$f$} (m-2-2)
   (m-2-1) edge node[auto] {$$} (m-2-2)
  ;

  \begin{scope}[shift=($(m-1-1)!.4!(m-2-2)$)]
  \draw +(-0.25,0) -- +(0,0)  -- +(0,0.25);
  \end{scope}
\end{tikzpicture}
\]
in which $f \maps X \to Y$ is a fibration in $\cat{C}$, is mapped by $F$ to a pullback square in $\cat{D}$.
\end{enumerate}
\end{definition}

Note that the above axioms imply that an exact functor preserves finite products and weak equivalences. The latter statement follows from the factorization discussed in Sec.\ \ref{sec:factor} which implies that a weak equivalence is the composition of an  acyclic fibration with  a right inverse of another acyclic fibration.

Exact functors between two categories of fibrant objects both satisfying the additional axioms \eqref{EA1} and \eqref{EA2} preserve homotopy pullbacks up to homotopy equivalence:

\begin{proposition} \label{prop:exact-hpb}
Let $F \maps \cat{C} \to \cat{D}$ be an exact functor between CFOs satisfying Axiom 
\eqref{EA1} and Axiom \eqref{EA2}. Let $E$ denote the homotopy pullback of a diagram
$X \xto{g} Z \xleftarrow{f} Y$ in $\catC$. Then $F(E)$ is homotopy equivalent to the homotopy pullback 
$K$ of the diagram $F(X) \xto{F(g)} F(Z) \xleftarrow{F(f)} F(Y)$ in $\cat{D}$.
\end{proposition}
\begin{proof}
Let $\factor{Y}{f}{Z}$ denote the functorial factorization of $f \maps Y \to Z$ in $\catC$.
Since $F$ is exact, $F(Y) \xto{F(i_f)} F(P(f)) \xto{F(p_f)} F(Z)$ is a factorization in $\cat{D}$ of the morphism $F(f)$ into a weak equivalence followed by a fibration. Hence, the pullback $F(X) \times_{F(Z)} F(P(f))$ is homotopy equivalent to $K$ by Prop.\ \ref{prop:hpb}(3). On the other hand,  Prop.\ \ref{prop:hpb}(1) provides a weak equivalence $E \weq  X \times_{Z}P(f)$. Therefore, the exactness of $F$ yields a sequence of weak equivalences in $\cat{D}$
\[
F(E) \weq F\bigl(X \times_{Z}P(f)\bigr) \cong F(X) \times_{F(Z)} F(P(f)) \weq K.
\] 
\end{proof}

\subsection{Towers in categories of fibrant objects} \label{sec:towmodcat}
For a category $\cat{C}$, we denote by $\tow(\cat{C})$ the category of towers in
$\cat{C}$, i.e.\ diagrams in $\cat{C}$ of the form
\[
\cdots \to X_{n} \xto{p_{(n-1)}} X_{n-1} \xto{p_{(n-2)}} \cdots  \xto{p_{(1)}} X_1 \xto{p_{(0)}} X_0.
\]
A morphism between towers $X$ and $Y$ is a sequence of morphisms $\{f_i
\maps X_i \to Y_i \}_{i \geq 0}$ in $\cat{C}$ such that the obvious diagram
commutes. 
We will use the next proposition a few times in this paper, for the particular cases when $\cat{M}= \Ch_\proj$ and $\cat{M}=\sSet$.  

\begin{proposition} \label{prop:towmodcat}
Let $\cat{M}$ be a model category and $\catC \sse \cat{M}$ the full subcategory of fibrant objects of $\cat{M}$. Let  $f \maps X \to Y \in \tow(\catC)$ be a morphism between towers in $\catC$. 
\begin{enumerate}
\item \label{item:towmodcat1} Suppose for all $i \geq 0$,  
the maps $p^{X}_{(i)} \maps X_{i+1} \to X_{i}$ and $p^{Y}_{(i)} \maps Y_{i+1} \to Y_{i}$ are fibrations in $\catC$.
If each $f_i \maps X_i \weq Y_i$ is a weak equivalence in $\catC$, then the morphism 
$\plim f \maps \plim X \to \plim Y$ is also a weak equivalence in $\catC$.

\item \label{item:towmodcat2} 
If $f_0 \maps X_0 \fib Y_0$ is a (acyclic) fibration in $\catC$,
and the unique morphism $X_{n+1} \to X_n \times_{Y_n} Y_{n+1}$ 
in the following commutative diagram
\begin{equation} \label{diag:towmodcat}
\begin{tikzpicture}[operad style]
\matrix (m) [matrix of math nodes, row sep=2em,column sep=2em]
  {  
X_{n+1} \& [-1 cm] \& \\ [-0.5cm]
\&     X_n \times_{Y_n} Y_{n+1} \&  Y_{n+1} \\
\& X_n  \& Y_n \\
}; 
  \path[->,font=\scriptsize] 
  
  (m-1-1) edge node[auto] {$$} (m-2-2)
  (m-2-2) edge node[auto] {$$} (m-2-3)
  (m-2-2) edge node[auto,swap] {$$} (m-3-2)
  (m-2-3) edge node[auto] {$p^{Y}_{(n)}$} (m-3-3)
  (m-3-2) edge node[auto] {$f_n$} (m-3-3)
  (m-1-1) edge [bend left=20]  node[auto] {$f_{n+1}$} (m-2-3)
  (m-1-1) edge [bend right=40]  node[auto,swap] {$p^{X}_{(n)}$} (m-3-2)
  ;

  \begin{scope}[shift=($(m-2-2)!.4!(m-3-3)$)]
    \draw +(-0.25,0) -- +(0,0)  -- +(0,0.25);
  \end{scope}
\end{tikzpicture}
\end{equation}
is a (acyclic) fibration in $\catC$ for all $i \geq 0$,  then the morphism 
$\plim f \maps \plim X \to \plim Y$ is a (acyclic) fibration in $\catC$.

\end{enumerate}
\end{proposition}
\begin{proof}
Apply Prop.\ 2.5 and Prop.\ 2.6 of \cite{Chacolski-Scherer} to the model category $\cat{M}^{\op}$
\end{proof}

\section{Homotopy theory of complete cochain complexes} \label{sec:ch-htpy-thy}
Let $\FF$ be a field of arbitrary characteristic. We begin by recalling some basic facts concerning filtered and complete filtered  vector spaces over $\FF$. We refer the reader to \cite[Ch.\ 2]{DSV:2018} and \cite[Sec.\ 7.3.1 -- 7.3.4]{Fresse:BookI} for additional background material.

\subsection{Complete filtered vector spaces} \label{sec:filtvect}
We denote by $\fVect$ the additive category whose objects are 
$\Z$-graded $\FF$-vector spaces $V$  equipped with a decreasing
filtration of subspaces beginning in filtration degree 1:
\[
V= \cF_1V \sss \cF_2V \sss \cF_3V \sss \cdots 
\]
Morphisms in $\fVect$ are degree 0 linear maps $f \maps V \to V'$ that are compatible with the filtrations:
\[
f \bigl( \cF_{n}V \bigr) \sse \cF_n V' \quad \forall n \geq 1.
\]

\subsubsection{Notation for filtered vector spaces}\label{sec:fVect-note}
The following notation will be used regularly throughout the paper. Let $V \in \fVect$.
For each $n \geq 1$, we have the canonical surjections and quotient maps
\[
p_{\pp{n}} \maps V \to \qf{V}{n} \quad \bp_{\pp{n}} \maps \qf{V}{n+1} \to \qf{V}{n}
\]
such that the diagram
\[
\begin{tikzpicture}[descr/.style={fill=white,inner sep=2.5pt},baseline=(current  bounding  box.center)]
\matrix (m) [matrix of math nodes, row sep=2em,column sep=2em,
  ampersand replacement=\&]
  {  
 \& V \&  \\
V/\cF_{n+1}V \& \& V/\cF_{n}V \\
}
; 
  \path[->,font=\scriptsize] 
   (m-1-2) edge node[auto,swap] {$p_{(n+1)}$} (m-2-1)
   (m-1-2) edge node[auto] {$p_{(n)}$} (m-2-3)
  (m-2-1) edge node[auto] {$\bp_{(n)}$} (m-2-3) 
;
\end{tikzpicture}
\]
commutes. We denote by
\[
\hp_{(n)} \maps \plimV \to V/\cF_nV \qquad \forall n \geq 1
\]
the canonical map out of the projective limit, and we let
\[
p_V \maps V \to \plimV
\]
denote the unique map satisfying $\hp_{(n)} \circ p_V =p_{(n)}$ for all $n\geq 1$.

Given a morphism $f \maps V \to V' \in \fVect$, we denote its restriction to each piece of the filtration as 
\[
\cF_n f:= f \vert_{\cF_n V} \maps \cF_{n}V  \to \cF_n V'.
\]
For each $n \geq 1$, the compatibility of $f$ with the filtrations gives two commuting diagrams of short exact sequences in $\Vect$ which appear repeatedly throughout the paper:
\begin{equation}\label{diag:ses1}
\begin{tikzdiag}{2}{2} 
{
0 \& \cF_nV \& V  \& \qf{V}{n} \& 0 \\
0 \& \cF_nV' \& V'  \& \qf{V'}{n} \& 0 \\
}; 
\path[->,font=\scriptsize] 
(m-1-2) edge  node[auto,swap] {$\cF_n f$} (m-2-2) edge  node[auto] {$$}(m-1-3)
(m-1-3) edge  node[auto] {$p_{\pp{n}}$} (m-1-4)
(m-1-3) edge  node[auto,swap] {$f$} (m-2-3)
(m-1-4) edge  node[auto] {$\qq{f}{n}$} (m-2-4)
(m-2-3) edge  node[auto] {$p'_{\pp{n}}$} (m-2-4)
(m-2-2) edge  node[auto] {$$} (m-2-3)
(m-1-1) edge node[auto,swap] {$$} (m-1-2)
(m-2-1) edge node[auto,swap] {$$} (m-2-2)
(m-1-4) edge node[auto,swap] {$$} (m-1-5)
(m-2-4) edge node[auto,swap] {$$} (m-2-5)
;
\end{tikzdiag}\tag{D1}
\end{equation}
and

\begin{equation}\label{diag:ses2}
\begin{tikzdiag}{2}{2} 
{  
0 \& \cF_nV/\cF_{n+1}V \& V/\cF_{n+1}V \& \qf{V}{n}\& 0 \\
0 \& \cF_nV'/\cF_{n+1}V' \& V'/\cF_{n+1}V'  \& \qf{V'}{n} \& 0 \\
}; 
\path[->,font=\scriptsize] 
(m-1-1) edge node[auto,swap] {$$} (m-1-2)
(m-2-1) edge node[auto,swap] {$$} (m-2-2)
(m-1-4) edge node[auto,swap] {$$} (m-1-5)
(m-2-4) edge node[auto,swap] {$$} (m-2-5)
(m-1-3) edge  node[auto] {$\bp_{(n)}$} (m-1-4)
(m-2-3) edge  node[auto] {$\bp'_{(n)}$} (m-2-4)
(m-1-2) edge node[auto,swap] {$$} (m-1-3)
(m-2-2) edge node[auto,swap] {$$} (m-2-3)

(m-1-2) edge node[auto,swap] {$\cF_{n}\qq{f}{n+1}$} (m-2-2)
(m-1-3) edge node[auto,swap] {$\qq{f}{n+1}$} (m-2-3)
(m-1-4) edge node[auto,swap] {$\qq{f}{n}$} (m-2-4)
;
\end{tikzdiag} \tag{D2}
\end{equation}
In the above diagrams, the morphisms $\qq{f}{n+1}$, $\qq{f}{n}$, and $\cF_{n}\qq{f}{n+1}$ are the obvious ones induced by $f$ and the universal property for quotients.

\begin{remark}\label{rmk:filt-iso}
The commutativity of diagram \eqref{diag:ses1} implies that a morphism between filtered vector spaces $f \maps V \to V'$ is an isomorphism in $\fVect$ if and only if $f$ is an isomorphism in $\Vect$ and 
$\qq{f}{n}$ is an isomorphism in $\Vect$ for each $n\geq 1$.   
\end{remark}

\subsubsection{Completion and towers} \label{sec:comp}
We next record notation and conventions for completions of filtered vector spaces.
Given a filtered vector space $V \in \fVect$, the projective limit $\hV:=\plimV$ is also an object in $\fVect$ when equipped with filtration given by the subspaces 
\[
\cF_n\hV := \ker \bigl ( \hp_{(n)} \maps \hV \to \qf{V}{n} \bigr).
\]
The space $\hV$ is complete with respect to the topology induced by this filtration, and the universal map $p_V \maps V \to \hV$  extends to a filtered morphism in $\fVect$.
Throughout we represent elements of $\hV$ as coherent sequences in $V$ or, equivalently, as convergent series in $V.$ See \cite[Lemma 2.14]{DSV:2018} for further details. We denote by  
\[
\cVect \sse \fVect
\]
the full subcategory of \df{complete filtered vector spaces} whose objects are those filtered vector spaces $V \in \fVect$ such that $p_V \maps V \xto{\cong} \hV$
is an isomorphism in $\fVect$.

We denote by $\ttow \maps \fVect \to \tVect$ the functor which assigns to a filtered vector space $V$ the tower
\begin{equation} \label{eq:towfunc}
\ttow(V):= \quad \cdots \to V/\cF_{n+1} V \xto{\bp_{(n)}} V/\cF_n V \xto{\bp_{(n-1)}} 
\cdots \xto{\bp_{(2)}} V / \cF_2V \to 0 \to 0 
\end{equation}
We obtain the completion of a filtered vector space $V \in \fVect$ as a composition of functors
\begin{equation*} 
\begin{split}
\plim \ttow \maps & \fVect \to \cVect, \qquad
V \xto{f} V' \quad \mapsto \quad \wh{V} \xto{\wh{f}} \wh{V'}\\
\end{split}
\end{equation*}
where $\wh{V}:=\plim_n \qf{V}{n}$ and $\wh{f}:=\plim_n \qq{f}{n}$. Note that if $f \maps V \to V'$ is a morphism of filtered vector spaces, then the diagram in $\fVect$ 
\[
\begin{tikzpicture}[descr/.style={fill=white,inner sep=2.5pt},baseline=(current  bounding  box.center)]
\matrix (m) [matrix of math nodes, row sep=2em,column sep=2em,
  ampersand replacement=\&]
  {  
V \& V' \\
\wh{V} \& \wh{V'} \\ 
}; 
  \path[->,font=\scriptsize] 
   (m-1-1) edge  node[auto] {$f$} (m-1-2)
   (m-1-1) edge  node[auto,swap] {$p_V$} (m-2-1)
   (m-1-2) edge  node[auto] {$p_{V'}$} (m-2-2)
   (m-2-1) edge  node[auto] {$\wh{f}$} (m-2-2)
  ;
\end{tikzpicture}
\]
commutes. This gives us the following basic result, which we record for later reference.

\begin{lemma} \label{lem:fvect}
For each $V, V' \in \cVect$, there is a natural isomorphism of abelian groups 
\begin{equation*} 
  \hom_{\tVect}\bigl(\ttow(V), \ttow(V') \bigr) \cong \hom_{\cVect}(V,V').  
 \end{equation*}
\end{lemma}

The next lemma will be crucial for analyzing fibrations in Sec.\ \ref{sec:Ch_cfo} and \ref{sec:Linf_cfo}.

\begin{lemma} \label{lem:split}
If $f \maps V \to V'$ is a morphism in $\cVect$ such that the restriction
$\cF_nf \maps \cF_nV \to \cF_nV'$ is surjective  for all $n \geq 1$, then there exists a morphism $\si \maps V' \to V$ in $\cVect$ such that $f \si = \id_{V'}$.
\end{lemma}

\begin{proof}
In the notation of diagram \eqref{diag:ses2}, 
we will inductively construct linear maps $\si_n \maps \qf{V'}{n} \to \qf{V}{n}$ for each $n\geq 1$ such that $\qq{f}{n} \si_n = \id_{\qf{V'}{n}}$, and such that for all $n\geq 2$  
\begin{equation} \label{eq:split}
\si_{n-1} \qpr{n-1} = \qq{\bp}{n-1} \si_n
\end{equation}
This will give a morphism between towers $ \tow(V') \to \tow(V)$, 
and then Lemma \ref{lem:fvect} will provide the desired right inverse $\si = \plim \si_n$.
 
Set $\si_1=0$, and let $\si_2 \maps \qf{V'}{2} \to \qf{V}{2}$ be a section of the surjective linear map $\qq{f}{2}$. Since $\qf{V}{1}=\qf{V'}{1}=0$, Eq.\ \ref{eq:split} is satisfied trivially. This establishes the base case.

Let $n \geq 2$ and suppose  $\si_{n} \maps \qf{V'}{{n}} \to \qf{V}{n}$ 
is a section of $\qq{f}{n}$ satisfying Eq.\ \ref{eq:split}. Consider the commutative diagram \eqref{diag:ses2}. Choose a linear map $s \maps \qf{V}{n} \to \qf{V}{n+1}$ such that $\bp_{(n)} s =\id_{\qf{V}{n}}$, and  let $\ti{\si}_{n+1} \maps \qf{V'}{n+1} \to \qf{V}{n+1}$ be 
\[
\ti{\si}_{n+1}:= s \cc \si_n \cc \qq{\bp'}{n}.
\] 
Then the image of the map $\tha \maps \qf{V'}{n+1} \to \qf{V'}{n+1}$ defined as 
\[
\tha:= f_{n+1}\ti{\si}_{n+1} - \id_{\qf{V'}{n+1}}
\]
lies in $\ker \qq{\bp'}{n}= \cF_nV'/\cF_{n+1}V'$. Finally, let $\ta \maps \cF_nV'/\cF_{n+1}V' \to \cF_nV/\cF_{n+1}V$ be a section of the surjective map $\cF_n\qq{f}{n+1}$, and 
define
\[
\si_{n+1} \maps \qf{V'}{n+1} \to \qf{V}{n+1}, \quad \si_{n+1}:= \ti{\si}_{n+1} - \ta \tha. 
\]
A direct calculation then shows that $\qq{f}{n+1} \si_{n+1} =\id_{\qf{V'}{n+1}}$, and 
$\si_{n}  \qq{\bp'}{n} = \qq{\bp}{n} \si_{n+1}$.
\end{proof}

We conclude this subsection by recalling a basic fact about kernels of morphisms in the additive category $\cVect$.

\begin{lemma} \label{lem:ker}
Let $f \maps V \to W$ be a morphism in $\cVect$. Let $\ker f \in \Vect$ be the usual kernel of the underlying linear map. Then $\ker f$, equipped with the induced filtration $\cF_n \ker f:= \cF_nV \cap \ker f$ is a complete filtered vector space, and is the kernel object of the map $f \maps V \to W$ in the category $\cVect$.
\end{lemma}
\begin{proof}
See \cite[Sec.\ 7.3.4(b)]{Fresse:BookI}.
\end{proof}


\subsubsection{The completed tensor product in $\cVect$} \label{sec:ctensor}
Let $V \in \fVect$ be a filtered vector space and $A \in \Vect$ a non-filtered graded vector space.  
The filtration on $V$ induces a canonical filtration on the tensor product $V \tensor A$
given by the subspaces  $\cF_n(V \tensor A) := \cF_nV \tensor A$.
Since in our case, $-\tensor_\FF A \maps \Vect \to \Vect$ is an exact functor, there are natural identifications for each $n \geq 1$
\[
V \tensor A /\cF_n(V \tensor A ) \cong (V/\cF_nV) \tensor A.
\]
Therefore, the maps $\bp_{(k)} \maps \qf{V}{k+1} \to \qf{V}{k}$ induce the inverse system of surjections
\[
\bp_{(k)} \tensor \id_A  \maps \qf{V}{k+1} \tensor A \to \qf{V}{k} \tensor A.
\]
We define the \df{completed tensor product} of $V$ and $A$ as 
\[
V \ctensor A : = \plim \ttow(V \tensor A)=\plim_{n} \bigl (V/\cF_nV ~  \tensor A \big).
\] 
If $f \maps A \to A'$ is a linear map in $\Vect$, then the morphism 
$\id_V \tensor f \maps V \tensor A \to V \tensor A'$ is compatible with the respective filtrations. Hence, for a fixed $V \in \cVect$ we obtain a functor 
\begin{equation} \label{eq:ctensor-func}
\begin{split}
V \ctensor - \maps &\Vect \to \cVect, \qquad f \maps A \to A' ~ \mapsto ~  \id_V \ctensor f \maps V \ctensor A \to V \ctensor A', 
\end{split}
\end{equation}
where $\id_V \ctensor f := \plim \id_{\qf{V}{n}} \tensor f$.

\subsection{Complete filtered cochain complexes}  \label{sec:cCh}
We denote by  $\cCh$ the category of complete filtered cochain complexes over $\FF$. An object $(V,d)$ of $\cCh$
consists of a complete vector space $V \in \cVect$ equipped with a degree $1$ differential $d \maps V \to V$ such that each piece of the filtration on $V$ is a sub-cochain complex $(\cF_nV, d)$. Morphisms in $\cCh$ are those morphisms in $\cVect$ which are compatible with the differential. 

\begin{notation} \label{note:cCh}
Given $(V,d) \in \cCh$, for each $n \geq 1$ we have the following short exact sequences in $\Ch$
\[
\begin{split}
0 \to (\cF_nV,d) \to &(V,d)  \to (\qf{V}{n},\qq{d}{n}) \to 0,\\
\end{split}
\] 
\[
\begin{split}
0 \to \bigl(\cF_nV/\cF_{n+1}V, \qq{d}{n+1} \bigr) \to & \bigl(V/\cF_{n+1}V, \qq{d}{n+1} \bigr)  \to \bigl(\qf{V}{n}, \qq{d}{n} \bigr) \to 0.
\end{split}
\] 
Above, the differentials $\qq{d}{n}$ and $\qq{d}{n+1}$ are the usual ones on the quotient complexes.
Given a morphism $f \maps (V,d) \to (V',d')$ in $\cCh$, the diagrams \eqref{diag:ses1} and \eqref{diag:ses2} lift to the category $\Ch$ in the obvious way.
\end{notation}

Note that the category $\cCh$ has finite products. In particular, if $(V,d_V)$, $(W,d_W) \in \cCh$, then the usual product of complexes $(V \times W,d_{\times})$ is complete with respect to the filtration given by the subcomplexes $\cF_n(V \times W):= \cF_n V \times \cF_nW$. 

Let us return briefly to the tensor product from Sec.\ \ref{sec:ctensor}. Given a complete filtered cochain complex $(V,d) \in \cCh$ and a unfiltered complex $(A,\del) \in \Ch$, let
$(V \tensor A, d_A)$ denote the usual tensor product of complexes. The functor \eqref{eq:ctensor-func} extends to cochain complexes $V \ctensor - \maps \Ch \to \cCh$,
and sends $f \maps (A,\del) \to (A',\del')$ in $\Ch$ to the filtered chain map  $\id_V \ctensor f \maps (V \ctensor A, d_A) \to (V \ctensor A', d_{A'})$.

\subsection{$\cCh$ as a category of fibrant objects} \label{sec:Ch_cfo}
In this section, we first give a fairly explicit proof that $\cCh$ forms a category of fibrant objects.
Then we show that this CFO structure satisfies the additional axioms \eqref{EA1} and \eqref{EA2} from Sec.\ \ref{sec:acyc}. We will use these results later in Sec.\ \ref{sec:hmpty-Linf} to construct the CFO structure on  complete shifted $L_\infty$-algebras.

\begin{definition}\label{def:chain-weq}
Let $f \maps (V,d) \to (V',d')$ be a morphism in $\cCh$. We say $f$ is a \df{weak equivalence} if for all $n \geq 1$ the restrictions
\[
\cF_nf \maps (\cF_{n}V, d) \to (\cF_n V', d')  
\]
are quasi-isomorphisms of cochain complexes. We say $f$ is a \df{fibration} if for all $n \geq 1$ the restrictions $\cF_nf$
are degree-wise surjective maps of cochain complexes.
\end{definition}
\begin{remark}\label{rmk:cCh}
\mbox{}
\begin{enumerate}
\item If $f \maps (V,d) \weq (V',d')$ is a weak equivalence in $\cCh$, then
the long exact sequence in cohomology applied to diagram \eqref{diag:ses1} implies that
\[
\qq{f}{n} \maps  (\qf{V}{n},\qq{d}{n}) \weq (\qf{V'}{n},\qq{d'}{n})
\]
is a quasi-isomorphism for each $n \geq 1$. By applying the same argument to diagram \eqref{diag:ses2}, we then conclude that 
\[
\cF_n \qq{f}{n+1} \maps (\cF_nV/\cF_{n+1}V, \qq{d}{n}) \weq (\cF_nV'/\cF_{n+1}V', \qq{d'}{n})
\]
is a quasi-isomorphism for each $n \geq 1$, as well.

\item If $f \maps (V,d) \fib (V',d')$ is a fibration, then the surjectivity of the restriction $\cF_nf$ implies that the inclusion $\qf{\ker f}{n} ~ \emb ~  \ker \qq{f}{n}$ yields an isomorphism
$\qf{\ker f}{n} \cong \ker \qq{f}{n}$ for each $n \geq 1$.

\end{enumerate}

\end{remark}

\begin{theorem}\label{thm:Ch_cfo}
\mbox{}
The category $\cCh$ of complete filtered cochain complexes 
equipped with the weak equivalences and fibrations introduced in Def.\ \ref{def:chain-weq} is a category of fibrant objects with a functorial path object.
\end{theorem}

\begin{proof}
It is easy to see that axioms 1,2,3, and 7 of Def.\ \ref{def:cfo} are satisfied.
We verify the remaining axioms in the following sections.
\end{proof}

\subsubsection{Pullbacks of (acyclic) fibrations in $\cCh$ (axioms 4 \& 5)}
The following proposition is straightforward, but we give an explicit proof following \cite[Thm.\ 4.1]{Vallette:2014} in preparation for the analogous result in Prop.\ \ref{prop:strict-pb}. 

\begin{proposition} \label{prop:Ch-pb}
The pullback of a fibration (resp.\ acyclic fibration) exists in $\cCh$, and is a fibration (resp.\ acyclic fibration).
\end{proposition}

\begin{proof}
Consider a diagram of the form $(W,d_W) \xto{g} (U,d_U) \xleftarrow{f} (V,d_V)$, where $f$ is a fibration in $\cCh$. Since $\cCh$ is an additive category, the pullback object of $f$ along $g$ is a kernel object. Hence Lemma \ref{lem:ker} implies that the pullback object is the usual pullback $P \sse W \times V$ in $\Vect$ equipped with the subspace filtration $\cF_n P := P \cap (\cF_nW \times \cF_n V)$. Since $f$ is a fibration, Lemma \ref{lem:split} implies that there exists a filtered linear map $\si \maps U \to V$ in $\cVect$ such that $f \si = \id_{U}$. Let $\ker f \in \cVect$ denote the kernel of $f$ as in Lemma \ref{lem:ker}. The linear map
\[
h \maps W \times \ker f \to P, \quad h(w,x):=(w, \si g(w) + x)
\] 
is a filtration preserving isomorphism in $\cVect$ with filtered inverse
\[
j \maps P \to W \times \ker f, \quad j(w,v):= \bigl(w, v-\si g(w) \bigr).
\]
It then follows that 
\begin{equation} \label{diag:filt-pb}
\begin{tikzpicture}[descr/.style={fill=white,inner sep=2.5pt},baseline=(current  bounding  box.center)]
\matrix (m) [matrix of math nodes, row sep=2em,column sep=3em,
  ampersand replacement=\&]
  {  
\bigl( W \times \ker f, \ti{d} \bigr) \& (V, d_V) \\
(W,d_W) \& (U,d_U) \\
}
; 
  \path[->,font=\scriptsize] 
   (m-1-1) edge node[auto] {$\pr_V h$} (m-1-2)
   (m-1-1) edge node[auto,swap] {$\pr_Wh$} (m-2-1)
   (m-1-2) edge node[auto] {$f$} (m-2-2)
   (m-2-1) edge node[auto] {$g$} (m-2-2)
  ;

  \begin{scope}[shift=($(m-1-1)!.4!(m-2-2)$)]
  \draw +(-0.25,0) -- +(0,0)  -- +(0,0.25);
  \end{scope}
\end{tikzpicture}
\end{equation}
is a pullback diagram in $\cCh$, where $\ti{d}:= j\circ(d_W \times d_V) \cc h$, and clearly
\[
\cF_n (\pr_W h) = \pr_W h \vert_{\cF_n(W \times \ker f)} \maps \cF_n W \times \cF_n (\ker f) \to \cF_n{W} 
\]
is a surjection for all $n \geq 1$. Hence, $\pr_W h$ is a fibration.

Now suppose that $f \maps (V,d_V) \afib (U,d_U)$ is an acyclic fibration.   
Clearly, $\pr_W h$ in the pullback diagram \eqref{diag:filt-pb} is a surjective quasi-isomorphism. Indeed, as a chain map $f$ is an acyclic fibration in the model category $\Ch_\proj$.
Let $n \geq 2$. By hypothesis, $\cF_n f$ is a surjective quasi-isomorphism. Hence, the complex
$\bigl(\ker \cF_nf, d_V \bigr)$ is acyclic. It follows from the definition of the differential $\ti{d}$ that we have a short exact sequence of cochain complexes
\[
\bigl(\cF_n \ker f, d_V \bigr) \emb \bigl( \cF_n W \times \cF_n \ker f, \ti{d} \bigr) \xto{\pr _{W}h} \bigl(\cF_n W, d_W \bigr).
\]
Therefore, since $\cF_n \ker f = \ker \cF_nf$, we deduce from the long exact sequence in cohomology that $\pr _{W}h \vert_{\cF_n (W \times \ker f)}$ is a quasi-isomorphism.
\end{proof}

\newcommand{\unit}{\boldsymbol{1}}
\newcommand{\mult}{\boldsymbol{\cdotp}}
\subsubsection{Functorial path object for $\cCh$ (axiom 6)}
The existence of a functorial path object stems from a suitable algebraic model for the unit interval.
For our purposes, the following ad-hoc definition is sufficient.
\begin{definition} \label{def:interval} 
A unital dg associative $\FF$-algebra $\J:=(\J,\del, \mult,\unit)$ is a \df{model for the interval} if there exist unital dg algebra morphisms $\ev_0,\ev_1 \maps \J \to \FF$ such that
\begin{enumerate}
\item The composition $\FF \xto{\unit} \J \xto{\ro} \FF \times \FF$ is the diagonal, where
$\ro:=(\ev_0,\ev_1)$.
\item As cochain maps, $\rho$ is an epimorphism, and both $\ev_0$ and $\ev_1$ are quasi-isomorphisms. 
\end{enumerate}  
\end{definition}
Note that the axioms above imply that the unit $\unit$ is a quasi-isomorphism in $\Ch$, and that 
the ``evaluation maps'' $\ev_0$, $\ev_1$ are epimorphisms.
We recall two examples. 
\begin{enumerate}
\item Let $N^{\ast}(\Del^1):= (N^\ast(\Del^1),\del, \cupp, \unit)$ be the normalized cochain algebra 
on the standard 1-simplex with coefficients in $\FF$. As a graded vector space, $N^1(\Del^1):= \FF \vphi_{[1]}$, and  $N^0(\Del^1):= \FF \vphi_0 \times \FF \vphi_1$. The differential is $\del \vphi_0 := \vphi_{[1]}$, and  $\del \vphi_1 := -\vphi_{[1]}$; the multiplication is the usual cup product, and the unit is
$\unit:= \vphi_0 + \vphi_1$. The algebra morphisms $\ev_0, \ev_1$  are defined on degree 0 generators as  $\ev_0(\vphi_0):= \ev_1(\vphi_1):=1$, and $\ev_0(\vphi_1):= \ev_1(\vphi_0):=0$. Although normalized cochains do not appear in the later sections of this paper, they appear in our related work \cite{MR} on complete $A_\infty$-algebras.

\item 
Let $\FF=\kk$ be a field of characteristic zero. Denote by $\Om_1:=(\kk[z,dz], \del_{\dR}, \wedge, 1_\kk)$ the polynomial de Rham algebra on the standard 1-simplex. As a graded vector space, it is concentrated in degrees 0 and 1: $\kkz=\kk[z] \oplus \kk[z]dz$, where $\kk[z]$ denotes the usual polynomial algebra.
The differential is the unique derivation sending $z \mapsto dz$ and $dz \mapsto 0$.
The algebra homomorphisms $\ev_0,\ev_1 \maps \Om_1 \to \kk$ are the evaluation maps at $z=0$, and $z=1$, respectively.
\end{enumerate}

\begin{proposition}\label{prop:Ch-pathobj}
Fix a model for the interval $\J:=(\J,\del, \mult,\unit)$.
The assignment $(V,d)  \mapsto  \bigl(V \ctensor \J, d_{\tensor} \bigr)$
to each $(V,d) \in \cCh$ is a functorial path object for the category $\cCh$.
\end{proposition}
The proof of the proposition will follow almost immediately from the next lemma.
\begin{lemma} \label{lem:ctensor}
Let $(V,d_V) \in \cCh$ be a complete cochain complex. 
\begin{enumerate}
\item If $f \maps (A,\del) \weq (A',\del')$ is a quasi-isomorphism in $\Ch$, then 
$\id_V \ctensor f \maps V \ctensor A \to V \ctensor A'$ is a weak equivalence in $\cCh$.   

\item If $f \maps A \fib A'$ is an epimorphism in $\Ch$, then $\id_V \ctensor f \maps V \ctensor A \to V \ctensor A'$ is a fibration in $\cCh$.
\end{enumerate}
\end{lemma}
\begin{proof}
\mbox{}
\begin{enumerate}[leftmargin=15pt]
\item 
We need to verify that $\id_V \ctensor f \maps V \ctensor A \to V \ctensor A'$ is a quasi-isomorphism 
and that  
\[
\cF_n (\id_V \ctensor f) \maps \cF_n(V \ctensor A) \to \cF_n(V \ctensor A')
\]  
is a quasi-isomorphism for each $n \geq 2$. The functor $\qf{V}{n} \tensor_\FF - \maps \Ch \to \Ch$ is exact, and therefore preserves quasi-isomorphisms.  This gives a commutative ladder of cochain complexes:
\[
\begin{tikzpicture}[operad style,row sep=2em,column sep=3em,]
\matrix (m) [matrix of math nodes]
{  
\cdots \& \qf{V}{n+1} \tensor A \& \qf{V}{n} \tensor A \& \cdots \\
\cdots \& \qf{V}{n+1} \tensor A' \& \qf{V}{n} \tensor A' \& \cdots \\
}; 
\path[->>,font=\scriptsize] 
   (m-1-1) edge node[auto] {$$} (m-1-2)
   (m-1-2) edge node[auto] {$$} (m-1-3)
   (m-1-3) edge node[auto] {$$} (m-1-4)
   (m-2-1) edge node[auto] {$$} (m-2-2)
   (m-2-2) edge node[auto] {$$} (m-2-3)
   (m-2-3) edge node[auto] {$$} (m-2-4)
;
\path[->,font=\scriptsize] 
   (m-1-2) edge node[auto] {$\id_{\qf{V}{n+1}} \tensor f$} node[sloped,below] {$\sim$} (m-2-2)
   (m-1-3) edge node[auto] {$\id_{\qf{V}{n}} \tensor f$} node[sloped,below] {$\sim$}(m-2-3)
  ;
\end{tikzpicture}
\]
in which every horizontal map is an epimorphism, and every vertical map is a quasi-isomorphism. 
Recall from Sec.\ \ref{sec:prelim} that epimorphisms and quasi-isomorphisms are the fibrations and weak equivalences, respectively, for the model category $\Ch_\proj$.
Therefore, by applying Prop.\ \ref{prop:towmodcat} to $\Ch_\proj$, we then deduce that 
$\id_V \ctensor f = \plim (\id_{\qf{V}{n}} \tensor f)$ is a quasi-isomorphism. As a consequence, for each $n \geq 2$, we obtain a morphism between short exact sequences
\[
\begin{tikzpicture}[operad style,row sep=2em,column sep=3em,]
\matrix (m) [matrix of math nodes]
{  
0 \& \cF_n (V \ctensor A) \& V \ctensor A \& \qf{V}{n} \tensor A  \& 0 \\
0 \& \cF_n (V \ctensor A') \& V \ctensor A' \& \qf{V}{n} \tensor A' \& 0\\
}; 
\path[->,font=\scriptsize] 
(m-1-1) edge node[auto] {$$} (m-1-2)
(m-1-2) edge node[auto] {$$} (m-1-3)
(m-1-3) edge node[auto] {$\qq{\ha{q}}{n}$} (m-1-4)
(m-1-4) edge node[auto] {$$} (m-1-5)
(m-2-1) edge node[auto] {$$} (m-2-2)
(m-2-2) edge node[auto] {$$} (m-2-3)
(m-2-3) edge node[auto] {$\qq{\ha{q}'}{n}$} (m-2-4)
(m-2-4) edge node[auto] {$$} (m-2-5)
(m-1-2) edge node[auto,swap] {$\cF_n (\id_V \ctensor f)$} (m-2-2)
(m-1-3) edge node[auto,swap] {$\id_V \ctensor f$} node[sloped,above] {$\sim$} (m-2-3)
(m-1-4) edge node[auto] {$\id_{\qf{V}{n}} \tensor f$} node[sloped,below] {$\sim$} (m-2-4)
;
\end{tikzpicture}
\]
in which two of the three vertical maps are quasi-isomorphisms. Hence, the third map
is a quasi-isomorphism as well.

\item Suppose $f \maps (A,\del) \to (A',\del')$ is an epimorphism. We need to verify that 
$\id_V \ctensor f \maps V \ctensor A \to V \ctensor A'$ is surjective, and that 
$\cF_n (\id_V \ctensor f) \maps \cF_n(V \ctensor A) \to \cF_n(V \ctensor A')$ is surjective for all $n\geq 2$. It suffices to verify these properties in the categories $\cVect$ and $\Vect$, respectively. Let $\si \maps A' \to A$ in $\Vect$ be a right inverse to $f$ as a map between graded vector spaces. Then $\id_V \tensor \si \maps A' \tensor V \to A \tensor V$ is a filtration preserving right inverse to $\id_V \tensor f$. Therefore
$\ha{\si}:= \plim_n \id_{\qf{V}{n}} \tensor \si$ is a right inverse to $\id_V \ctensor f$, and so the latter map is surjective. For each $n\geq 2$, recall that $\cF_n (V \ctensor A) := \ker \qq{\ha{q}}{n}$ and $\cF_n (V \ctensor A') := \ker \qq{\ha{q}'}{n}$, where 
$\qq{\ha{q}}{n}$ and $\qq{\ha{q}'}{n}$ are the canonical surjections to 
$\qf{V}{n} \tensor A$ and $\qf{V}{n} \tensor A'$, respectively. Hence, 
the commutativity of the diagram
\[
\begin{tikzdiag}{2}{2}
{
V \ctensor A' \& \qf{V}{n} \tensor A'\\
V \ctensor A \& \qf{V}{n} \tensor A  \\
};
\path[->,font=\scriptsize]
(m-1-1) edge node[auto,swap] {$\ha{\si}$} (m-2-1)
(m-1-2) edge node[auto] {$\id_{\qf{V}{n}} \tensor \si$} (m-2-2)
;
\path[->>,font=\scriptsize]
(m-1-1) edge node[auto] {$\qq{\ha{q}'}{n}$} (m-1-2)
(m-2-1) edge node[auto] {$\qq{\ha{q}}{n}$} (m-2-2)
;
\end{tikzdiag}
\]
implies that the image of the restriction $\si \vert_{\cF_n (V \ctensor A')}$ is contained in 
$\cF_n (V \ctensor A)$. Therefore, $\si \vert_{\cF_n (V \ctensor A')}$ is a right inverse to $\cF_n (\id_V \ctensor f)$.
\end{enumerate}
\itemqed
\end{proof}

\begin{proof}[Proof of Prop.\ \ref{prop:Ch-pathobj}]
Let $(V,d) \in \cCh$. There are canonical isomorphisms of complete cochain complexes:
\[
V \ctensor \FF  = \plim_n \Bigl (\qf{V}{n} \tensor \FF \Bigr ) \cong V, \quad
V \ctensor (\FF \times \FF)  = \plim_n \Bigl (\qf{V}{n} \tensor (\FF \times \FF) \Bigr) \cong V \times V.
\]
Therefore, by applying Lemma \ref{lem:ctensor} to the dg algebra morphisms $\unit$ and $\ro$ from
Def.\ \eqref{def:interval}, we obtain the desired factorization $V \xto{\id_V \ctensor \unit} V \ctensor \J \xto{\id_V \ctensor \ro} V \times V$ of the diagonal map in $\cCh$.
\end{proof}

\subsection{Homotopy inverses for weak equivalences in $\cCh$}
Propositions \ref{prop:Ch-pb} and \ref{prop:Ch-pathobj} imply that 
the CFO structure on $\cCh$ 
satisfies the additional Axiom \eqref{EA2} from Sec.\ \ref{sec:acyc}. We now verify that Axiom \eqref{EA1} is also satisfied by proving a stronger statement: every acyclic fibration in $\cCh$ is a deformation retraction.
\begin{proposition} \label{prop:Ch-acyc}
Let $f \maps (V,d) \afib (V',d')$ be an acyclic fibration in $\cCh$. Then
\begin{enumerate}[leftmargin=15pt]
\item There exists a morphism $\tau \maps (V',d') \to (V,d)$ in $\cCh$  such that $f\tau = \id_{V'}$.   
\item There exists a degree $-1$ linear map $h \maps V \to \ker f$ such that $h(\cF_nV) \sse \ker f \cap \cF_nV$ for all $n \geq 1$, and such that $\id_V - \tau f = dh + hd$.
\end{enumerate} 
\end{proposition}
\newcommand{\tsi}{\ti{\si}}

\begin{proof}
\mbox{}

\begin{enumerate}[leftmargin=15pt]

\item We proceed by induction. Set $\ta_1:=0$. Since $\qq{f}{2} \maps \qf{V}{2} \to \qf{V'}{2}$ is a surjective quasi-isomorphism, we may choose a right inverse $\tau_2 \maps \qf{V'}{2} \to \qf{V}{2}$ in $\Ch$. Now assume $n \geq 2$, and suppose we have a chain map   $\ta_{n} \maps (\qf{V'}{{n}},\qq{d'}{n}) \to (\qf{V}{n}, \qq{d}{n})$ satisfying
\begin{equation} \label{eq:acyc1}
\qq{f}{n} \ta_n = \id_{\qf{V'}{n}}, \quad \ta_{n-1}\qpr{n-1} = \qpp{n-1}\ta_n.
\end{equation}
\myindent As in the proof of Lemma \ref{lem:split}, we use the inductive hypothesis to first produce a {\it linear} map $\si_{n+1} \maps
\qf{V'}{n+1} \to \qf{V}{n+1}$ such that $\qq{f}{n+1} \si_{n+1} = \id$, and 
$\ta_{n} \cc \qq{\bp'}{n} = \qq{\bp}{n} \cc \si_{n+1}$. Now let $\tha \maps \qf{V'}{n+1} \to \qf{V}{n+1}$ be the degree $1$ linear map $\tha:= \qq{d}{n+1} \si_{n+1} - {\si}_{n+1}\qq{d'}{n+1}$. Then $\tha$ is a degree $1$ coboundary in the hom-complex $\bigl(\Hom_\FF(\qf{V'}{n+1}, \qf{V}{n+1}),\pa \bigr)$. Since $\tau_n$ is a morphism of cochain complexes, we have the equalities  $\qpp{n}\tha = \qpp{n} \cc \pa \si_{n+1} = (\qq{d}{n}\ta_{n} - \ta_n \qq{d'}{n}) \qpr{n}=0$.
Hence $\im \tha \sse \ker  
\qpp{n} = \cF_nV/\cF_{n+1}V$. Moreover, since $\qq{f}{n+1}$ is a chain map and $\si_{n+1}$ is a linear right inverse of $\qq{f}{n+1}$, we deduce that $\im \tha \sse \ker \qpp{n} \cap \ker \qq{f}{n+1}= \ker \cF_{n}\qq{f}{n+1}$.

\myindent It follows from Remark \ref{rmk:cCh} that $\cF_n \qq{f}{n+1}$ is a quasi-isomorphism, and since $\cF_n \qq{f}{n+1}$ is also surjective, we conclude that  $\ker \cF_{n}\qq{f}{n+1}$ is contractible. Hence, $\tha$ is a degree $1$ cocycle in the acyclic complex $\Hom_\FF(\qf{V'}{n+1}, \ker \cF_{n}\qq{f}{n+1})$, so there exists a degree zero linear map 
$\eta \maps \qf{V'}{n+1} \to \ker \cF_{n}\qq{f}{n+1}$ such that 
\[
\tha = \qq{d}{n+1} \eta - \eta  \qq{d'}{n+1}.
\]

\myindent Finally, let $\ta_{n+1}:= \si_{n+1} - i \cc \eta$, where $i \maps \cF_nV/\cF_{n+1}V \emb \qf{V}{n+1}$ is the inclusion. By construction, $\ta_{n+1}$ is a chain map satisfying the equalities 
$\qq{f}{n+1} \ta_{n+1} = \id$, and $\ta_{n} \cc \qq{\bp'}{n} = \qq{\bp}{n} \cc \ta_{n+1}$. This completes the inductive step, and therefore $\ta:=\plim \ta_k$ is a right inverse to $f$ in $\cCh$.

\item First, since $f$ is a fibration, Remark \ref{rmk:cCh}(2) implies that we have an isomorphism of complexes 
\[
\qf{\ker f}{k} \cong \ker \qq{f}{k}
\]
for each $k\geq 1$. So we proceed as in part (1), and inductively construct for each $k \geq 1$ a degree $-1$ linear map
$h_k \maps \qf{V}{k} \to \ker \qq{f}{k}$ such that 
\begin{equation} \label{eq:acyc2}
\qpp{k-1} h_k = h_{k-1} \qpp{k-1}, \quad \qq{d}{k} h_k + h_k\qq{d}{k}=\qq{g}{k},
\end{equation}
where $\qq{g}{k} \maps \qf{V}{k} \to \ker \qq{f}{k}$ is the chain map  $\qq{g}{k}:=\id_{\qf{V}{k}} - \ta_k \qq{f}{k}$,  and $\ta_k$  is the chain map satisfying \eqref{eq:acyc1} from part (1). 

\myindent Set $h_1:=0$. Since $\qq{f}{2}$ is a surjective quasi-isomorphism, the complex $\ker \qq{f}{2}$ is acyclic. Since we work over a field, we may choose a null homotopy $h \maps \ker \qq{f}{2} \to \ker \qq{f}{2}$. Let $h_2 \maps \qf{V}{2} \to \ker \qq{f}{2}$ be the degree $-1$ linear map $h_2:= h\qq{g}{2}$. Then, by construction, $h_2$ satisfies the equalities \eqref{eq:acyc2} with $k=2$. Assume $n \geq 2$, and suppose $h_n \maps \qf{V}{n} \to \ker \qq{f}{n}$
satisfies the equalities \eqref{eq:acyc2} for $k=n$.  Consider the commutative diagram of cochain complexes
\begin{equation}\label{diag:acyc2}
\begin{tikzdiag}{2}{3}
{
\cF_nV/\cF_{n+1}V \& \qf{V}{n+1} \& \qf{V}{n} \\
\ker \cF_{n}\qq{f}{n+1} \&  \ker \qq{f}{n+1}  \& \ker \qq{f}{n} \\
};

\path[->,font=\scriptsize]
(m-1-1) edge node[auto] {$i$} (m-1-2) edge node[auto,swap] {$\cF_n \qq{g}{n+1}$} (m-2-1) 
(m-1-2) edge node[auto,swap] {$\qq{g}{n+1}$} (m-2-2) 
(m-1-3) edge node[auto] {$\qq{g}{n}$} (m-2-3)
(m-2-1) edge node[auto] {$i$} (m-2-2)
(m-2-2) edge node[auto] {$\qpp{n} \vert_{\ker}$} (m-2-3)  
;
\path[->>,font=\scriptsize]
(m-1-2) edge node[auto] {$\qpp{n}$} (m-1-3) 
;
\end{tikzdiag}
\end{equation}
The hypotheses on $f$ imply that $\cF_{n}\qq{f}{n+1}$, $\qq{f}{n+1}$, and $\qq{f}{n}$ are all 
surjective quasi-isomorphisms. Therefore each complex in the bottom row of diagram \eqref{diag:acyc2} is acyclic. The top row of \eqref{diag:acyc2} splits in the category of graded vector spaces; the bottom row does as well, since $g_{(n)}$ is surjective. Hence, $\qpp{n} \vert_{\ker}$ is a surjective quasi-isomorphism of cochain complexes, and therefore there exists a chain map
\[
\eta \maps (\ker \qq{f}{n}, \qq{d}{n}) \to (\ker \qq{f}{n+1}, \qq{d}{n+1})
\] 
such that $\qpp{n} \vert_{\ker} \cc \eta = \id$. Define $\ka \maps \qf{V}{n+1} \to \ker \qq{f}{n+1}$
be the degree $-1$ linear map $\ka := \eta \cc h_n \cc \qpp{n}.$ By construction, 
we have $ \qpp{n}\ka = h_n \qpp{n}$. 

\myindent Next, we modify $\ka$ to obtain the desired homotopy satisfying \eqref{eq:acyc2}.
Let $\lam \maps \qf{V}{n+1} \to \ker \qq{f}{n+1}$ be the degree $0$ linear map
\[
\lam := \qq{d}{n+1}\ka + \ka \qq{d}{n+1} - \qq{g}{n+1} = \eta \cc  \qq{g}{n} \cc \qpp{n} - \qq{g}{n+1}. 
\]
Note that the last equality above follows from the commutativity of the diagram \eqref{diag:acyc2}. A direct computation shows that $\qpp{n}\lam =0$. Hence, $\lam$ is a degree $0$ cocycle in the acyclic complex $\bigl(\Hom_\FF(\qf{V}{n+1}, \ker \cF_{n}\qq{f}{n+1}),\pa \bigr).$

\myindent Let $\rho \maps \qf{V}{n+1} \to \ker \cF_{n}\qq{f}{n+1}$ be a degree $-1$ map satisfying $\lam = \pa \ro$. Define $h_{n+1}:= \ka -i\cc \ro$, where $i$ is the inclusion map in diagram \eqref{diag:acyc2}. By construction, we have $\qpp{n}h_{n+1} = \qpp{n}\ka = h_n \qpp{n}$, and
$\pa h_{n+1} = \lam + \qq{g}{n+1} - \pa \ro = \qq{g}{n+1}.$
This completes the inductive step. 

\myindent Finally, by applying Lemma \ref{lem:fvect} to the morphism of towers
$\{h_n\} \maps \ttow(V) \to \ttow(\ker f[-1])$ we 
obtain the sought after filtered homotopy $h:= \plim h_{n}$ satisfying $\id - \tau f = dh + hd$.
\end{enumerate}
\itemqed
\end{proof}

\begin{corollary} \label{cor:Ch-quasi-inverse}
If $f \maps (V,d) \weq (V',d')$ is a weak equivalence in $\cCh$, then there exists a weak equivalence
$g \maps (V',d') \weq (V,d)$, as well as homotopies
$h_V \maps (V,d) \to (V \ctensor \J, d_\tensor)$ and $h_{V'} \maps (V',d') \to (V' \ctensor \J, d'_\tensor)$,
which induce equivalences $gf \heq \id_V$ and $fg \heq \id_{V'}$
\end{corollary}
\begin{proof}
Follows from the above proposition along with Prop.\ \ref{prop:bifib-invert} from Sec.\ \ref{sec:acyc}. 
\end{proof}

\section{Complete shifted $L_\infty$-algebras} \label{sec:Linf}
From here on, we work over a field $\FF=\kk$ of characteristic zero.

\subsection{Reminder on coderivations and coalgebra morphisms} \label{sec:coder}
For a graded vector space $V \in \Vect$, we let $\S(V)$ denote the reduced cofree conilpotent cocommutative coalgebra generated by $V$. Let $\Phi \maps \S(V) \to \S(V')$ be a $\kk$-linear map.
For $p,m \geq 1$ the  notation $\Phi^p_{m}$ is reserved for the restriction-projections 
\begin{equation*} 
\Phi^p_{m} \maps \S^{m}(V) \to \S^{p}(V') \qquad  \Phi^p_{m}:= \pr_{\S^{p}(V')} \circ \Phi \vert_{\S^{m}(V)} 
\end{equation*}
Furthermore, we denote by $\Phi^1 \maps \S(V) \to V'$ the linear map $\Phi^{1} := \pr_{V'}\cc\Phi$.

We recall that coderivations $Q \maps \S(V) \to \S(V)$ are in one to one correspondence with linear maps $Q^1 \maps \S(V) \to V$ via the formula $Q(v_1,v_2, \ldots, v_n) = Q^1_n(v_1,v_2, \ldots, v_n) + \sum_{t \geq 1}^n Q^t_n(v_1,v_2,\ldots,v_n),$
where
\begin{equation} \label{eq:coder}
\begin{split}
Q^{t}_{n}(v_{1}, \ldots, v_{n}) = \hspace{-.5cm}\sum_{\sigma \in \Sh(n-t+1,t-1)} \hspace{-.5cm}
\epsilon(\sigma) ~ Q^1_{n-t+1}(v_{\sigma(1)}, \ldots, v_{\sigma(n-t+1)}) v_{\sigma(n-t+2)} \cdots v_{\sigma(n)}.
\end{split}
\end{equation}
Above, $\Sh(p,q) \sse S_{n}$ is the set of $(p,q)$ shuffles. 
Similarly, a coalgebra morphism $\Phi \maps \S(V) \to \S(V')$ is uniquely determined by the degree $0$ linear map $\sPhi \maps \S(V) \to V'$ via the formula 
\begin{multline}
\label{eq:map}
\Phi(v_1, v_2, \ldots, v_n) = \\
 \sum_{t \ge 1}
\sum_{\substack{ k_1+ \dots + k_t = n \\[0.1cm] k_j \ge 1 }} 
\sum_{\si }  
\sPhi(v_{\si(1)}, \ldots, v_{\si(k_1)}) \sPhi(v_{\si(k_1+1)}, \ldots, v_{\si(k_1 + k_2)}) 
\cdots \sPhi(v_{\si(n-k_t+1)}, \ldots, v_{\si(n)})\,,
\end{multline}
where the index $\si$ ranges over those $(k_1,k_2,\ldots,k_t)$ shuffles in $S_{n}$
that satisfy the condition $\si(1) < \si(k_1+1) <  \si(k_1+k_2+1) < \dots < \si(n-k_t+1)$.

Let $Q$ and $Q'$ be coderivations on $\S(V)$ and $\S(V')$, respectively, and let 
$\S(V) \xto{\Phi} \S(V') \xto{\Psi} \S(V'')$ be coalgebra morphisms. 
Recall that the above formulas imply that the various compositions allowed between $Q$, $Q'$, $\Phi$, and $\Psi$ are uniquely determined by linear maps $(\Phi Q)^1 \maps \S(V) \to V'$, $(Q'\Phi)^1 \maps \S(V) \to V'$, and $(\Psi \Phi)^1 \maps \S(V) \to V''$, where
\begin{equation} \label{eq:comp}
(\Phi Q)^1_n = \sum^n_{t=1} \Phi^1_t Q^t_n, \quad
(Q'\Phi)^1_n = \sum^n_{t=1} Q'^1_t \Phi^t_n, \quad
(\Psi \Phi)^1_n = \sum^n_{t=1} \Psi^1_t \Phi^t_n.
\end{equation}

\subsubsection{Filtered coderivations and coalgebra morphisms}
If $V \in \fVect$, then the filtration on $V$ 
induces one on the tensor powers $ V^{\tensor k} = \cF_1 V^{\tensor k}  \sss \cF_2 V^{\tensor k}  \sss \cdots$:
\begin{equation*} 
 \begin{split}
 \cF_n V^{\tensor k} & : = \hspace{-15pt} \bigoplus_{i_1 + i_2 + \cdots + i_k = n} \cF_{i_1} V \tensor \cF_{i_2} V \tensor 
\cdots \tensor \cF_{i_k} V,
\end{split}
\end{equation*}
which in turn induces a filtration on the tensor algebra $\ba{T}(V)$ given by
$\cF_n \ba{T}(V) := \bigoplus_{k\geq 1}  \cF_n V^{\tensor k}$.
Finally, this decreasing filtration induces one on the subspace $\S(V)$:  
\begin{equation} \label{eq:filtsym}
\begin{split}
\S(V)= \cF_1\S(V) \sss \cF_2 \S(V) \sss \cdots, \qquad  \cF_n \S(V) := \S(V) \cap \cF_n \ba{T}(V).
\end{split}
\end{equation}
We summarize the above discussion with the following lemma, whose proof is straightforward.

\begin{lemma} \label{lem:fcompat}
Let $V,V' \in \fVect$. Let $\sPhi \maps \S(V) \to V'$ be a degree $0$ linear map, and $\sQ \maps \S(V) \to V$ be a linear map of arbitrary degree. If for all $m \geq  1$ and $i_1,\ldots,i_m \geq 1$ we have
\begin{equation} \label{eq:fcompat}
\begin{split}
\sPhi( \cF_{i_1} V \otimes  \cF_{i_2} V \otimes \dots \otimes  \cF_{i_m} V) &\sse \cF_{i_1 + i_2 + \dots + i_m} V', \\
\sQ( \cF_{i_1} V \otimes  \cF_{i_2} V \otimes \dots \otimes  \cF_{i_m} V) &\sse \cF_{i_1 + i_2 + \dots + i_m} V, \\
\end{split}
\end{equation}
then the corresponding coalgebra morphism $\Phi \maps \S(V) \to \S(V')$ and coderivation $Q \maps \S(V) \to \S(V)$ are compatible with the filtrations \eqref{eq:filtsym}.

Moreover, if, as above, $Q$, $Q'$ and $\Phi$, $\Psi$ are filtration-compatible coderivations and coalgebra morphisms, respectively, then their compositions \eqref{eq:comp} are also compatible with the filtrations \eqref{eq:filtsym}, whenever they are well-defined.
\end{lemma}

\subsection{$\sinf$-algebras} \label{sec:shiftLinf}
In order to make certain calculations more straightforward, as well as  match the conventions in our previous work \cite{GM_Theorem}, we will use shifted $L_\infty$-algebras. 
A shifted $L_\infty$-structure on $L$ is equivalent to an $L_\infty$-structure on $\bs L$,
the suspension of $L$. More precisely, a {\bf $\sinf$-algebra} $(L,d,Q)$ is a cochain complex $(L,d) \in \Ch$ for
which the reduced cocommutative coalgebra $\S(L)$ is equipped with a
degree 1 coderivation $Q$ such that $Q x=d x$ for all $x \in L$
and $Q^2=0$. This structure is equivalent to specifying a sequence of \textit{degree one} 
brackets $\{\cdot,\cdot, \ldots, \cdot \}_m \maps \S^m(L) \to L$ for each $ m \geq 1$
satisfying compatibility conditions which can be thought of as higher--order Jacobi and Leibniz identities. (See Eq.\ 2.5. in \cite{GM_Theorem}.)

An \df{$\infty$-morphism}, or \df{weak $L_\infty$-morphism} $\Phi \maps (L,d,Q) \to (L',d',Q')$
between two $\sinf$-algebras is a dg coalgebra morphism $\Phi \maps \bigl( \S(L), Q \bigr) \to \bigl( \S(\pri{L}), \pri{Q} \bigr)$.
We denote by $\SLinf$ the category of $\sinf$-algebras and $\infty$-morphisms.
The compatibility of $\Phi$ with the codifferentials $Q$ and $Q'$ can be expressed 
as a sequence of equations in terms of the brackets $\brac{m}$ and linear maps $\sPhi_k$. (See Eq.\ 2.8 in \cite{Enhanced}.) In particular, $\Phi Q = Q' \Ph$ implies that the linear term, or \df{tangent map} of the dg coalgebra morphism $\Ph$ induces a map of cochain complexes:
\begin{equation*} \label{eq:tangent} 
\tan(\Ph):=\Phi^1_1= \pr_{\ti{L}} \Phi \vert_{L}  \maps (L,d) \to (L',d').
\end{equation*}
Following the standard terminology in deformation theory, we call $(L,d)$ the \df{tangent complex} of the $\sinf$-algebra $(L,d,Q)$. This defines a functor
\begin{equation} \label{eq:tan}
\tan \maps \SLinf \to \Ch.
\end{equation}
\begin{definition} \label{def:Linf-maps-def}
A morphism $\Ph \maps (L,d,Q) \to (L',d',Q')$ in $\SLinf$
is an \df{$\infty$-quasi-isomorphism} if
$\tan(\Ph)$ is a quasi--isomorphism of cochain complexes. 
Similarly,  $\Ph$ is an \df{$\infty$-epimorphism} if $\tan(\Ph)$
is degreewise surjective. Finally we say $\Ph \maps (L,d,Q) \to (L',d',Q')$ is \df{strict} if it is completely determined by its tangent map i.e.
\begin{equation} \label{eq:strict}
\sPhi(x_1,\ldots,x_m)=0 \quad \forall m \geq 2.
\end{equation}
\end{definition}
Note that in the strict case, we have $\Ph(x_1,x_2, \ldots, x_m)= \sPhi_1(x_1)\sPhi_1(x_2) \cdots \sPhi_1(x_m) \in \S^m(L')$ and
\[
\sPhi_1 \sQ_m(x_1, x_2, \ldots, x_m) =  Q'^{1}_m \bigl( \sPhi_1(x_1),\sPhi_1(x_2), \ldots, \sPhi_1(x_m) \bigr) \quad \forall m \geq 1.
\]
\begin{notation}
We will often write a  strict $\infty$-morphism $\Phi \maps (L,d,Q) \to (L',d',Q')$ as 
\[
\Ph=\sPh
\]
in order to emphasize the equalities \eqref{eq:strict}. In doing so, we tacitly identify $\Phi$ 
with its tangent map.
\end{notation}

\subsection{Complete $\sinf$-algebras} \label{sec:FLie}
In this section, we recall the definitions of complete $\sinf$-algebras and their $\infty$-morphisms. We record some basic properties, and analyze some important examples, including what we call ``bounded'' complete $\sinf$-algebras.
\begin{definition} \label{def:filtcompLinf}
A \df{filtered $\sinf$-algebra} $(L,d,Q)$ is 
a filtered cochain complex $(L,d)$ whose
reduced cocommutative coalgebra $\S(L)$ is equipped with a
degree 1 coderivation $Q$ such that $Q \vert_L =d$ and $Q\circ Q =0$, and  
for all $m \geq 2$ and $i_1,\ldots,i_m \geq 1$
\[
\sQ \bigl(\cF_{i_{1}}L \tensor \cF_{i_{2}}L \tensor \cdots \tensor \cF_{i_{m}}L \bigr) \subseteq
\cF_{i_{1} + i_{2} + \cdots + i_{m}} L.
\]
A \df{complete $\sinf$-algebra}, or {\boldmath $\cinf$}\df{-algebra} for short, is a 
filtered $\sinf$-algebra $(L,d,Q)$ such that $(L,d) \in \cCh$.  
\end{definition}

\begin{remark}
A $\cinf$-algebra in our sense is a shifted analog of a 
complete $L_\infty$-algebra, in the sense of A.\ Berglund \cite[Def.\ 5.1]{Berglund}.
See also \cite[Ch. 6]{Markl-book}.
\end{remark}

Let $\FLie$ be the category whose objects are filtered
$\sinf$-algebras and whose morphisms are \df{filtered  $\infty$-morphisms}, i.e.,
$\infty$-morphisms $\Phi \maps (L,d,Q) \to (\pri{L},d',\pri{Q})$ that satisfy
\begin{equation}
\label{eq:fmorph}
\Phi^1( \cF_{i_1} L \otimes  \cF_{i_2} L \otimes \dots \otimes  \cF_{i_m} L) \sse \cF_{i_1 + i_2 + \dots + i_m} \pri{L}.
\end{equation}
We denote by $\cLie \sse \FLie$ the full subcategory whose objects are $\cinf$-algebras.

\begin{remark} \label{rmk:disc-filt} 
Every cochain complex $(V,d)$ can be given the structure of an \df{abelian $\sinf$-algebra} by defining $\sQ_1 := d$ and $\sQ_{k \geq 2}:=0$. The corresponding functor $\Ch \to \SLie$ is obviously not full. Let $(L,d,Q) \in \SLie$. Give $L$ the ``discrete filtration'' $\cF_1L:=L$, and $\cF_{k}L =0$ $\forall k \geq 2$. Then this filtration is compatible with the $\sinf$-structure if and only if $(L,d,Q)$ is abelian. In particular, equipping a cochain complex $(V,d) \in \Ch$ with the discrete filtration, gives a functor $\mathrm{disc} \maps \Ch \to \cLie$, which realizes $\Ch$ as a \textit{full}
subcategory of $\cLie$.

\end{remark}

We conclude this subsection by noting that the familiar characterization of $\sinf$-isomorphisms extends to isomorphisms in $\cLie$.
\begin{proposition}
A morphism $\Ph \maps (L,d,Q) \to (L',d',Q')$ in $\cLie$ is an isomorphism if and only if 
$\tan(\Ph)$ is an isomorphism in $\cVect$.
\end{proposition}
\begin{proof}
Assume $\psi \maps \pri{L} \to L$ is inverse to $\tan(\Ph)=\sPh$ in  $\cVect$. One constructs, using $\psi$ and the structure maps $\sPhi_m$, an inverse $\Psi \maps \S(L') \to \S(L)$ recursively in the usual way (e.g., see the formulas in the proof of \cite[Prop.\ 7.5]{Markl-book}).
Since $\psi$ and $\Phi$ are compatible with the filtrations on $L'$ and $\S(L)$, $\Psi$ is a morphism in $\cLie$. The converse direction is obvious.
\end{proof}

\subsubsection{Finite products in $\cLie$} \label{sec:product}
The product of  $\cinf$-algebras $(L,d,Q)$ and $(L',d',Q') \in \cLie$ is denoted by $(L\times L', d_\times,Q_\times)$, where $(L\times L',d_\times)$ is the product of complete cochain complexes, and $Q_\times$ is the degree one codifferential on $\S(L \times L')$ whose structure maps are
\begin{equation} \label{eq:product1}
(Q_\times)^1_{k} \bigl( (x_1, x'_1), (x_2,x'_2), \ldots, (x_k,x'_k) \bigr):= \Bigl(\sQ_{k}(x_1,x_2,\ldots,x_k) ,\spQ_{k}(x'_1,x'_2,\ldots,x'_k) \Bigr).
\end{equation}
The usual projections $\pr \maps L \times L' \to L$, $\pr' \maps L \times L' \to L'$ lift to strict $L_\infty$-epimorphisms
\[
(L,d,Q) \xleftarrow{\Pr} (L \times L', d_\times,Q_\times) \xto{\Pr'} (L',d',Q')
\]
We will frequently identify the (counital) cocommutative coalgebra $S(L \times L')$ with the categorical product $S(L)\tensor_\kk S(L')$, and this will be consistent with our convention for writing $\infty$-morphisms as dg coalgebra morphisms. In particular, given $\infty$-morphisms $\Ph_1 \maps (L_1,d^{L_1},Q^{L_1}) \to (L'_1,d^{L'_1},Q^{L'_1})$ and $\Ph_2 \maps (L_2,d^{L_2},Q^{L_2}) \to (L'_2,d^{L'_2},Q^{L'_2})$, we write
\begin{equation*} 
\Ph_1 \tensor \Ph_2 \maps (L_1 \times L_2, d_\times Q_{\times}) \to (L'_1 \times L'_2, d'_\times,Q'_{\times})
\end{equation*}
for the unique morphism in $\cLie$ which satisfies the usual universal property. The next proposition is elementary. We will use it in Sec.\ \ref{sec:qinverse-Lie} to analyze acyclic fibrations in $\cLie$. 
\begin{proposition}\label{prop:Linf-inc}
Let $(L,d,Q)$ and $(L',d',Q') \in \cLie$. The inclusion map of
complete vector spaces $L \emb L \times L'$ extends to a strict morphism
$\LdQ{} \to (L\times L', d_\times,Q_\times)$.
\end{proposition}

\subsubsection{Nilpotent $\sinf$-algebras} \label{sec:Nilpot} 
The main results in Sec.\ \ref{sec:sMC-htpy} concern the compatibility of
the homotopy theory for $\cinf$-algebras and $\infty$-morphisms with the simplicial Maurer-Cartan functor (defined below in Sec.\ \ref{sec:sMC}). The proofs crucially rely on  E.\ Getzler's 
previous work \cite{Getzler} involving the Maurer Cartan theory for nilpotent $L_\infty$-algebras and \textit{strict} morphisms. The purpose of the present section  is to clearly spell out the relationships between the category $\cLie$, and the category of nilpotent $\sinf$-algebras with strict morphisms between them.

Let $\LdQ{} \in \SLie$.  
Following A.\ Berglund \cite{Berglund}, let $\{\Ga_k L\}_{k \geq 1}$ denote the level-wise intersection of all filtrations on $L$ compatible with the $\sinf$-structure $Q$. Then  $\{\Ga_kL \}$ is also a compatible filtration called the \df{lower central series} of $(L,d,Q)$. The explicit formula \cite[Def.\ 6.10]{Markl-book} for $\Ga_k L$ is $\Ga_1 L := L$, and for $k \geq 2$
\begin{equation} \label{eq:LC_series}
\Ga_k L: = \sum_{m \geq 2} ~ \sum_{\substack{ i_1 + i_2 + \cdots + i_m \geq k \\ i_1,i_2,\ldots,i_m < k}} 
\sQ_{m}\bigl(\Ga_{i_1}L \tensor \Ga_{i_2}L\tensor \cdots \tensor \Ga_{i_m}L \bigr) 
\end{equation}
If there exists $N \geq 1$ such that $\Ga_N L =0$, then we say $\LdQ{}$ is \df{nilpotent}.
Remark \ref{rmk:disc-filt} implies that every abelian $\cinf$-algebra is nilpotent. More generally, the $L_\infty$-algebras considered by E.\ Getzler in \cite{Getzler} are nilpotent in this sense.

\paragraph{Weak versus strict morphisms of nilpotent $\sinf$-algebras.}
In \cite{Band,Berglund,Getzler,Yalin}, the morphisms between nilpotent $L_\infty$-algebras are considered without requiring compatibility with the lower central series filtration. This is not an issue because the authors of \cite{Band,Berglund,Getzler,Yalin} only consider \textit{strict} $L_\infty$-morphisms. Indeed, the explicit description \eqref{eq:LC_series} of the lower central series
makes the following proposition evident:

\begin{proposition} \label{prop:nil-strict}
Let $(L,d,Q)$ and $(L',d',Q')$ be nilpotent $\sinf$-algebras. Let $\Ph=\sPhi_1 \maps (L,d,Q) \to (L',d',Q')$ be a strict morphism in $\SLie$. Then $\Ph$ is compatible with the lower central series on $L$ and $L'$.
\end{proposition}
We denote by $\nilstLie \sse \Lie$ the category whose objects are nilpotent $\sinf$-algebras and whose morphisms are {strict} $\infty$-morphisms. Every nilpotent $\sinf$-algebra equipped with its lower central series is a $\cinf$-algebra. This gives us a functor, thanks to Prop.\ \ref{prop:nil-strict}:
\begin{equation*} 
\nilstLie \to \cLie.
\end{equation*}
Throughout, we will frequently identify $\nilstLie$ with its image under this inclusion functor.

\subsubsection{Bounded $\cinf$-algebras} \label{sec:bnd}
All of the nilpotent $\sinf$-algebras considered in this paper arise as $\cinf$-algebras equipped with a specified bounded descending filtration, which is not necessarily the one induced by its lower central series. It will be convenient at times to keep track of the depth of the bound. 

\begin{definition} \label{def:bndfltLie}
A  $\cinf$-algebra $(L,d,Q)$  with filtration $\{ \cF_kL \}_{k\geq 1}$
is \df{bounded at} {\boldmath $N$} \iff $\cF_NL=0.$
\end{definition}

\begin{remark}\label{rmk:bnd-nil}
Any filtered $\sinf$-algebra equipped with a bounded filtration as in Def.\ \ref{def:bndfltLie} is obviously complete and hence a bounded $\cinf$-algebra. Furthermore, any bounded $\cinf$-algebra $(L,d,Q)$ is nilpotent since $\Ga_kL \sse \cF_kL$ for all $k\geq 1$.  
\end{remark}
Denote by $\bndfiltLie$ the full subcategory of $\cLie$ whose objects are bounded 
 $\cinf$-algebras. Remark \ref{rmk:bnd-nil} implies that we have a forgetful functor   
\begin{equation} \label{eq:bnd-nil}
\bndfltstLie \xto{\jmath_{\str}} \nilstLie
\end{equation}
from the category of bounded $\cinf$-algebras and \textit{strict} $\infty$-morphisms. Note that abelian $\sinf$-algebras equipped with the discrete filtration as in Remark \ref{rmk:disc-filt} form a full subcategory of $\bndfiltLie$.

Here is our main example of a bounded $\cinf$-algebra.
Let $(L,d,Q) \in \FLie$ be a filtered $\sinf$-algebra. For each $n\geq 1$, the filtration on $L$ induces a filtration on the quotient vector space $\qf{L}{n}$
\[
\qf{L}{n}=\ov{\cF}_1 \bigl(\qf{L}{n} \bigr)  \sss \ov{\cF}_2 \bigl(\qf{L}{n} \bigr) \sss \cdots,
\]
and the canonical surjections $p_{\pp{n}} \maps L \to \qf{L}{n}$ and $\bp_{\pp{n}} \maps \qf{L}{n+1} \to \qf{L}{n}$ from Sec.\ \ref{sec:fVect-note} upgrade to morphisms between filtered vector spaces. 
We define degree $1$ linear maps $d_{\pp{n}} \maps \qf{L}{n} \to \qf{L}{n}$, and $(\qq{Q}{n})^1_k \maps \S^k(\qf{L}{n}) \to \qf{L}{n}$ for $k \geq 2$:
\begin{equation} \label{eq:quot-Linf}
d_{\pp{n}}(\ba{x}) := dx, \quad (\qq{Q}{n})^1_k(\bx_1,\bx_2, \ldots, \bx_k) := \sQ_k(x_1,x_2,\ldots,x_k),
\end{equation}
where $\bx=\qq{p}{n}(x)$ and $\bx_i=\qq{p}{n}(x_i)$ for all $i=1,\ldots,k$. 

The statements in the next proposition follow directly from the fact that the $\sinf$-structures
and morphisms in $\FLie$ are compatible with the filtrations in the sense of Def.\ \ref{def:filtcompLinf} and Eq.\ \ref{eq:fmorph}, respectively.

\begin{proposition}\label{prop:quot-filt}
Let $\LdQ{}$ be a filtered $\sinf$-algebra with filtration $\{\cF_kL\}_{k \geq 1}$.
\begin{enumerate}
\item \label{item:propqf1} For each $n\geq 1$, the linear maps \eqref{eq:quot-Linf} give the quotient $\qf{L}{n}$ the structure of a $\cinf$-algebra $(\qf{L}{n}, \qq{d}{n}, \qq{Q}{n})$ bounded at $N=n$.

\item\label{item:propqf2} The canonical surjections from Sec.\ \ref{sec:fVect-note} induce strict morphisms 
\[
\qq{p}{n}  \maps \LdQ{} \to (\qf{L}{n}, \qq{d}{n}, \qq{Q}{n}), \quad
\qq{\bp}{n}\maps(\qf{L}{n+1}, \qq{d}{n+1}, \qq{Q}{n+1}) \to (\qf{L}{n}, \qq{d}{n}, \qq{Q}{n})
\]
in $\FLie$ and $\bndfltstLie$, respectively.

\item\label{item:propqf3} 
For each $n \geq 1$, the inclusion of the subcomplex 
\[
\ker \qq{\bp}{n} = (\cF_{n}L/\cF_{n+1}L,\, \qq{d}{n+1}) ~ \sse ~ (L/\cF_{n+1}L,\qq{d}{n+1}) 
\]
extends to a morphism $(\cF_{n}L/\cF_{n+1}L,\, \qq{d}{n+1},0) \emb (\qf{L}{n+1}, \qq{d}{n+1}, \qq{Q}{n+1})$ in $\bndfltstLie$
which embeds $\ker \qq{\bp}{n}$ as an abelian $\sinf$-subalgebra.

\item\label{item:propqf4} If $\Ph \maps \LdQ{} \to \LdQ{\prime}$ is a morphism in $\FLie$, then for each $n \geq 1$, the linear maps
\[
\begin{split}
(\qq{\Ph}{n})^1_{k \geq 1} \maps \S^{k \geq 1}(\qf{L}{n}) \to \qf{L'}{n}, \qquad
(\qq{\Ph}{n})^1_k(\bx_1,\bx_2, \ldots, \bx_k):=\Ph^1_k(x_1,x_2, \ldots, x_k), 
\end{split}
\]
induce a unique morphism $\qq{\Ph}{n} \maps (\qf{L}{n}, \qq{d}{n}, \qq{Q}{n}) \to (\qf{L'}{n}, \qq{d'}{n}, \qq{Q'}{n})$
of $\cinf$-algebras which satisfies the equality $\qq{p'}{n} \Phi = \qq{\Ph}{n}\qq{p}{n}$.
\end{enumerate}
\end{proposition}

\subsubsection{Towers of  $\sinf$-algebras} \label{sec:Lietower}
Statement \ref{item:propqf4} of Prop.\ \ref{prop:quot-filt} above implies that
the functor $\ttow \maps \fVect \to \tVect$ defined in \eqref{eq:towfunc} extends to a functor
$\ttow \maps \FLie \to \tbdfltLie.$
Moreover, given  $(L,d,Q) \in \FLie$, the level-wise $\sinf$-structures on the the corresponding tower $\ttow(L,d,Q)$ induces a unique 
$\sinf$-structure on the completion $\wh{L}=\plim_n \qf{L}{n}$ such that the 
linear map $p_L \maps L \to \wh{L}$ from  Sec.\ \ref{sec:fVect-note} extends to a strict morphism of filtered $\sinf$-algebras $p_L \maps (L,d,Q) \to (\wh{L}, \cd, \wh{Q})$.
It is straightforward to check that this assignment of a filtered $\sinf$-algebra to its completion defines a functor
\begin{equation} \label{eq:Liecompfunc}
\plim\ttow \maps \FLie \to \cLie.
\end{equation}

\subsubsection{Tensoring with cdgas} \label{sec:cdgas}
Let $\cdga$ denote the category whose objects are unital commutative dg $\kk$-algebras and whose morphisms are unital dg algebra homomorphisms. Given $(L,d,Q) \in \cLie$ and $(B,\del) \in \cdga$, we denote
by $(L \tensor B, d_B, Q_B)$ the $\sinf$-algebra whose underlying cochain complex is the usual tensor product of complexes $(L \tensor B, d_B)$, where $d_B:=d \tensor \id + \id \tensor \del$,
and whose higher brackets are defined as:
\begin{equation*} 
(Q_B)^1_k \bigl(x_1 \tensor b_1, \ldots, x_k \tensor b_k \bigr):= (-1)^{\varepsilon} \sQ_k(x_1,\ldots,x_k) \tensor b_1b_2 \cdots b_k.
\end{equation*} 
Above $\varepsilon :=  \sum_{1 \leq i < j \leq k} \deg{b_i}\deg{x_j}$ is the usual Koszul sign.
Similarly, if $\Phi \maps (L, d,Q) \to (L',d',Q')$ is a morphism in $\SLie$, then the maps $(\Phi_B)^1_k\maps \S^{k}(L \tensor B) \to L' \tensor B$ defined as
\begin{equation} \label{eq:FB}
(\Phi_{B})^1_k\bigl(x_1 \tensor b_1, \ldots, x_k \tensor b_k \bigr):= (-1)^{\varepsilon} \sPhi_k(x_1,\ldots,x_k) \tensor b_1b_2 \cdots b_k.
\end{equation} 
assemble together to give a $\infty$-morphism $\Phi_{B} \maps (L\tensor B, d_B, Q_{B}) \to (L' \tensor B, d'_B, Q^{\prime}_{B})$.

The vector space $L \tensor B$ is equipped with the induced filtration from Sec.\ \ref{sec:ctensor}, and it is easy to see that it is compatible with the $\sinf$-structure $Q_{B}$. Hence, $(L \tensor B, d_B, Q_{B}) \in \FLie$, and so for any $B \in \cdga$ we obtain a functor:
$-\ctensor B \maps \cLie \to \cLie$
which sends $(L,d,Q)$ to the completion $(L \ctensor B, \wh{d}_B,\wh{Q}_{B})$ 
via the functor \eqref{eq:Liecompfunc}.
The proof of the next proposition follows from the definition of the filtration on the tensor product.
\begin{proposition}\label{prop:nilbnd-cdga}
Let $(B,\del) \in \cdga$. If $(L,d,Q) \in \bdfltLie$ is bounded at $N$, then the filtered $\sinf$-algebra $(L \tensor B, d_B, Q_{B})$ is also bounded at $N$, and therefore
$(L \ctensor B, \cd_B, \cQ_B) \cong (L \tensor B, d_B, Q_B)$.
\end{proposition}

\subsection{Maurer--Cartan elements and twisting} \label{sec:MC}
Our reference for this section is \cite[Sec.\ 2]{Enhanced}. Let $(L,d,Q) \in \cLie$. The \df{curvature} $\curv \maps L^0 \to L^1$ is the set-theoretic function 
\begin{equation*} 
\curv(a)=d a + \sum_{m \geq 2} \frac{1}{m!}
\sQ_m(\underbrace{a\cdots a}_{\text{$m$ {\rm times}}}) 
\qquad \forall a \in L^0,
\end{equation*}
which is well-defined due to: (1) the compatibility of the $\sinf$-structure with the filtrations,
(2) the fact that $L^0=\cF_1L^0$, and (3) the fact that $L$ is complete. Elements of the set
$\MC(L):= \{ \alpha \in L^0 ~ \vert ~ \curv(\alpha)=0 \}$ are called the \df{Maurer--Cartan (MC) elements} of $L$. Note that
MC elements of $L$ are elements of \underline{degree 0}. 
Similarly, let
$\Phi \maps (L,d,Q) \to (L',d',Q')$ be a morphism in $\cLie$. Such a morphism induces a function
$\Phi_\ast \maps L^0 \to L^{'0}$: 
\begin{equation}\label{eq:MC_F}
\Phi_\ast(a) := \sum_{m \geq 1} \frac{1}{m!} \sPhi(\underbrace{a\cdots a}_{\text{$m$ {\rm times}}})  \qquad \forall a \in L^0.
\end{equation}
The curvature function satisfies the following useful identities (which are proved in \cite[Prop.\ 2.2]{Enhanced}): 
\begin{equation} \label{eq:bianchi}
d(\curv(a)) + \sum_{m=1}^{\infty} \frac{1}{m!} \sQ_m \bigl(\underbrace{a
  \cdots a}_{\text{$m$ {\rm times}}} \cdot \, \curv(a)\bigr)= 0 \qquad \forall a \in L^0,
\end{equation}
and
\begin{equation} \label{eq:curv-morph}
\curv \big( \Phi_*(a) \big) = 
\sum_{m \ge 0} \frac{1}{m!} \sPhi_{m+1} \bigl( \underbrace{a
  \cdots a}_{\text{$m$ {\rm times}}} \cdot \,  \curv(a) \bigr) \qquad \forall a \in L^0.
\end{equation}

If $\Ph \maps (L,d,Q) \to (L',d',Q')$ is a morphism in $\cLie$, and $\al \in \MC(L)$ is a MC element of $L$, then
Eq.\ \ref{eq:curv-morph} implies that $\Phi_\ast(\al)$ is a MC element of $L'$. Using the composition formulas \eqref{eq:comp}, it is straightforward to check that the assignment $L \mapsto \MC(L)$ gives a functor $\MC \maps \cLie \to \Set.$

\begin{remark} \label{rmk:MC} 
Eq.\ \ref{eq:bianchi} above is known as the \df{Bianchi identity} \cite{Getzler}. Note that if
$(L,d) \in \Ch$ is a cochain complex considered as an abelian $\sinf$-algebra, then $\MC(L)$ is equal to the subspace of degree 0 cocycles $Z^0(L)$. Finally, it is easy to show that for any $(L,d,Q) \in \cLie$, there is a natural isomorphism $\MC(L) \cong \plim_n \MC(L/\cF_nL)$. See, for example, \cite[Prop.\ 3.17]{Markl-book}.
\end{remark}

\subsubsection{Twisting}
Elements of degree 0 in a $\cinf$-algebra $(L,d,Q)$ can be used to ``twist'' the differential and multi-brackets on $L$. Let $a \in L^0$ be \textit{any element of degree $0$}. We define the degree 1 maps $d^{a} \maps L \to L$, and  $(Q^a)^1_k \maps \S^k(L) \to L$ as
\begin{equation} \label{eq:twist}
\begin{split}
d^{a}(x) &: = dx  + \sum_{k=1}^{\infty} \frac{1}{k!} \sQ_{k+1}(\mtimes{a}{k} x) \qquad \forall x \in L,\\
(Q^{a})^1_m(x_1x_2\cdots x_m) &:=
\sum_{k=0}^{\infty} \frac{1}{k!} \sQ_{k+m}(\mtimes{a}{k} \cdot x_1 x_2\cdots x_m) \qquad \forall x_1,\ldots,x_m \in L.
\end{split}
\end{equation}  
Then the following equalities hold \cite[Prop.\ 2.2]{Enhanced}:
\begin{equation} \label{eq:twist-ids}
\begin{split}
d^{a} \circ d^{a} (x) &=  - \pr_LQ^a(\curv(a) \cdot x) \qquad \forall a \in L^0 \quad  \forall x \in L,\\
\curv(a + b) &= \curv(a) +  d^{a}(b) + \sum_{m=2}^{\infty} \frac{1}{m!} 
(Q^{a})^1_m(\mtimes{b}{m}) \qquad \forall a, b \in L^0.
\end{split}
\end{equation}
In particular, given an MC element $\al \in \MC(L)$,  we can construct a new 
$\cinf$-algebra $(L^\al,d,^\al,Q^\al)$ using the $\al$-twisted structures defined above. 
Here $L^{\alpha}:= L$ as a complete filtered graded vector space. 
Morphisms between $\cinf$-algebras can be twisted as well. See, for example,  \cite[Sec.\ 2]{Enhanced}. 

\begin{remark} \label{rmk:exp}
\mbox{}
\begin{enumerate}  
\item In the  proof of \cite[Prop.\ 2.2]{Enhanced}), the following fact is used 
for verifying some of the above identities concerning curvature and twisting:
For any $a \in L^0$, the expression $\exp(a) -1$ is a well-defined element of the usual completion of $\S(L)$ defined by the corresponding formal power series. By extending $Q$ in the natural way, a straightforward calculation shows that $Q \bigl( \exp( a) - 1 \bigr) = \exp(a) \curv(a)$,
and hence
\begin{equation} \label{eq:exp2}
\curv(a) = \pr_{L} \circ Q \bigl( \exp( a) - 1 \bigr).
\end{equation}

\item Similarly, if $\Phi \maps (L,d,Q) \to (L',d',Q')$ is a morphism in $\cLie$, then
we may extend $\Phi$ to the completions of $\S( L)$ and $\S(L')$ by formal power series, and  
a straightforward calculation gives us the equality 
\begin{equation} \label{eq:exp3}
\Phi \bigl( \exp( a) - 1 \bigr) = \exp( \Phi_\ast(a)) - 1.
\end{equation}  

\item Suppose $(L,d,Q)$ and $(L',d',Q')$ are complete $L_\infty$-algebras, and $\Phi \maps \S(L) \to \S(L')$ is a coalgebra morphism that is compatible with the filtrations, but \textit{not necessarily} compatible with the differentials. Then we note that the function $\Phi_\ast$ defined in Eq.\ \ref{eq:MC_F} is still well-defined, and the equality \eqref{eq:exp3} still holds. 
  
\end{enumerate}
\end{remark}

\subsubsection{Maurer-Cartan elements of bounded  $\cinf$-algebras}
Suppose $\LdQ{} \in \bdfltLie$ is bounded at $N$ (Def.\ \ref{def:bndfltLie}). 
Then clearly the  codifferential $Q$ is necessarily truncated at $N$, i.e., $\sQ_k =0$ for $k \geq N$. However, the compatibility of the $\sinf$ structure with the filtration also implies a bit more. The following technical results will be used in the proof of
Thm.\ \ref{thm:sMC-exact}.

\begin{lemma} \label{lem:bnd-curv}
Let $\LdQ{} \in \bdfltLie$ be an $\cinf$-algebra that is bounded at $N$.
\begin{enumerate}

\item \label{item:bnd-curv1}
If $a \in L^0$ and $y \in \cF_{N-1}L^0$, then
\begin{equation} \label{eq:bnd-curv}
\curv(a+y) = \curv(a) + dy.
\end{equation}
\item \label{item:bnd-curv2}
If $a \in L^0$ and $\curv(a) \in \cF_{N-1}L^{1}$, then $d \curv(a)=0.$

\end{enumerate}

\end{lemma}

\begin{proof}
\mbox{}
\begin{enumerate}[leftmargin=15pt]

\item  The identities \eqref{eq:twist-ids} imply that
\begin{equation} \label{eq:lem-bnd1}
\curv(a + y) = \curv(a) + d^{a}(y) + \sum_{m=2}^{\infty} \frac{1}{m!} 
(Q^{a})^1_m(\mtimes{y}{m}).
\end{equation}
Recall $L^0=\cF_1L^0$, therefore for all $k \geq 0$ and $m \geq 2$, we have
\[
\sQ_{k+m}(\mtimes{a}{k} ~ \mtimes{y}{m}) \in \cF_{m(N-1)+k}L \sse \cF_NL =0.
\]
It then follows from the definition of $(Q^{a})^1_m$ in \eqref{eq:twist} that the infinite summation on the right-hand side of \eqref{eq:lem-bnd1} above vanishes. Now consider the remaining term
\[
d^{a}(y) = dy  + \sum_{k=1}^{\infty} \frac{1}{k!} \sQ_{k+1}(\mtimes{a}{k} y).
\]
Using again the hypothesis on the filtration degree of $y$, we deduce that for all $k \geq 1$
\[
\sQ_{k+1}(\mtimes{a}{k} y) \in \cF_{N+k -1}L \sse \cF_NL =0.
\]
Hence, $d^{a}(y) = dy$, and the desired equality then follows.

\item We proceed using the same idea as in the proof of part 1. The hypothesis on the filtration degree of $\curv(a)$, along with the compatibility of $Q$ with the filtration, implies 
that infinite sum on the left-hand side of the Bianchi identity \eqref{eq:bianchi} vanishes.
Therefore, the identity implies that  $d\curv(a)=0$.
\end{enumerate}
\itemqed
\end{proof}

\subsection{Obstruction theory: a prelude} \label{sec:obsthy}
Lemma \ref{lem:bnd-curv} gives a quick proof of a fact that is well-known to experts: Every  $\cinf$-algebra $(L,d,Q)$ encodes a canonical ``obstruction theory'' for its Maurer-Cartan elements. This idea (along with an appropriate choice for $L$) is behind many of the familiar obstruction-theoretic arguments used in deformation theory and rational homotopy theory. 

Let $(L,d,Q) \in \cLie$. Recall that the $\sinf$-algebra $\bigl( \qf{L}{2},\qq{d}{2},\qq{Q}{2}\bigr)$ is abelian and is just the  
the cochain complex $(\qf{L}{2},\qq{d}{2})$. Hence, the set of MC elements in $\qf{L}{2}$ is the subspace of degree 0 cocycles $Z^0(\qf{L}{2})$. Let $\be$ be such a 0-cocycle, and let $\qq{p}{2} \maps \LdQ{} \fib (\qf{L}{2},\qq{d}{2})$
in $\cLie$ be the canonical surjection, as usual.
\begin{liftprob*}
Does $\be$ lift through $\qq{p}{2}$ to a Maurer-Cartan element $\al \in \MC(L)$?
\end{liftprob*}

The next proposition -- which we claim no originality for -- characterizes the existence and uniqueness of such Maurer-Cartan elements $\al \in \MC(L)$ 
via the  
associate graded cochain complex of $(L,d)$. In Sec.\ \ref{sec:obstruct}, this result is extended considerably to characterize a general lifting problem. 

\begin{proposition} \label{prop:obstruct}
Let $n \geq 2$. Let $\be \in \MC(\qf{L}{n-1})$ be a Maurer-Cartan element of the $\sinf$-algebra
$(\qf{L}{n-1},\qq{d}{n-1},\qq{Q}{n-1})$, and suppose $a \in (\qf{L}{n})^0$ is any degree 0 vector satisfying 
\[
\qpp{n-1}(a)=\be \in (\qf{L}{n-1})^0. 
\]
Then:
\begin{enumerate}[leftmargin=15pt]

\item The element $\bcurv{n}(a)$ is a 1-cocycle in the subcomplex 
$\ker \qpp{n-1}:=(\cF_{n-1}L/\cF_nL, \qq{d}{n})$, where $\bcurv{n}$ is the curvature function of the $\sinf$-algebra $(\qf{L}{n},\qq{d}{n},\qq{Q}{n})$.

\item The class 
\[
\bigl [\bcurv{n}( a ) \bigr] \in H^1(\cF_{n-1}L/\cF_nL)
\]
is independent of the choice of $a \in (\qf{L}{n})^0$ lifting the degree 0 vector $\be$.

\item There exists a Maurer-Cartan element $\al \in \MC(\qf{L}{n})$  lifting $\be \in \MC(\qf{L}{n-1})$ through the function $$\bp_{(n-1) \ast} \maps \MC(\qf{L}{n}) \to \MC(\qf{L}{n-1})$$
if and only if the class $[\bcurv{n}( a)]$ is trivial.

\item The fiber $\qqa{\bp}{n-1}^{\, \, -1}\bigl(\{\be\}\bigr) \sse \MC(\qf{L}{n})$ is either empty or 
a torsor over the abelian group 
\[
\MC(\cF_{n-1}L/\cF_nL)= Z^0(\cF_{n-1}L/\cF_nL).
\]
\end{enumerate}
\end{proposition}

\begin{proof}
\mbox{}
\begin{enumerate}[leftmargin=15pt]
\item Since $\qpp{n-1} \maps (\qf{L}{n},\qq{d}{n},\qq{Q}{n})  \to (\qf{L}{n-1},\qq{d}{n-1},\qq{Q}{n-1})$ is a strict $\infty$-morphism, we have $\qpp{n-1}\bigl(\bcurv{n}(a) \bigr) = \bcurv{n-1}\bigl(\qpp{n-1}(a)\bigr) =0$.
Therefore, $\bcurv{n}(a) \in \cF_{n-1}L/\cF_nL$. Proposition \ref{prop:quot-filt} implies that
$(\qf{L}{n},\qq{d}{n},\qq{Q}{n})$ is bounded filtered at $n$, and since the element $\bcurv{n}(a)$ lies in filtration degree $n-1$, it follows from Lemma \ref{lem:bnd-curv} that $\qq{d}{n} \bcurv{n}(a) =0$.

\item Suppose $a' \in (\qf{L}{n})^0$ is another element satisfying $\qpp{n-1}(a')=\be$.
Then the vector $\eta:= a' - a \in (\cF_{n-1}L/\cF_{n}L)^0$ has filtration degree $n-1$, and so 
Lemma \ref{lem:bnd-curv} implies that
\[
\bcurv{n}(a') = \bcurv{n}(a + \eta) = \bcurv{n}(a) + \qq{d}{n}\eta.
\]

\item If there exists $\al \in \MC(\qf{L}{n})$ lifting $\be$, then statement (2) implies that
the obstruction class is trivial. Conversely, suppose there  exists $\eta \in (\cF_{n-1}L/\cF_{n}L)^0$ such that $\bcurv{n}(a) = \qq{d}{n} \eta$. Let $\al:= a - \eta$. Then Lemma \ref{lem:bnd-curv} implies that $\bcurv{n}(\al) = \bcurv{n}(a) - \qq{d}{n}\eta =0$.

\item Suppose $\al,\al' \in \MC(\qf{L}{n})$ satisfy $\qqa{\bp}{n-1}(\al)=\qqa{\bp}{n-1}(\al')=\be$.
Since $\qpp{n-1}$ is a strict $\infty$-morphism, it follows that $\eta:= \al' - \al \in 
(\cF_{n-1}L/\cF_{n}L)^0$.  Lemma \ref{lem:bnd-curv} implies that $\bcurv{n}(\al') = \bcurv{n}(\al) + \qq{d}{n}\eta$. Since $\al$ and $\al'$ are Maurer-Cartan, it then follows that $\eta \in Z^0((\cF_{n-1}L/\cF_{n}L)) = \MC (\cF_{n-1}L/\cF_{n}L)$.  Conversely, the same argument implies that $\al':= \al + \eta$ is a Maurer-Cartan element of $\qf{L}{n}$ satisfying $\qqa{\bp}{n-1}(\al') = \be$
whenever $\al \in \MC(\qf{L}{n})$ is, provided  $\eta \in Z^0(\cF_{n-1}L/\cF_{n}L)$. 
\end{enumerate}
\itemqed
\end{proof}

\subsection{The simplicial Maurer-Cartan functor} \label{sec:sMC}
The MC elements of a $\cinf$-algebra are the vertices of a simplicial set. We recall the construction. Let $\Omega_n$ denote the de Rham-Sullivan algebra of polynomial 
differential forms on the geometric simplex $\Delta^n$ with 
coefficients in $\Bbbk$. (See, for example, \cite[Ch.\ 5]{Markl-book}.)  These assemble together 
into a simplicial cdga $\{\Omega_n\}_{n\geq0}$. 
The \df{simplicial Maurer--Cartan functor}
\begin{equation} \label{eq:sMC}
\sMC \maps \cLie \to \sSet
\end{equation}
assigns to each  $\cinf$-algebra $(L,d,Q)$ the simplicial set whose $n$-simplices are 
\begin{equation*} 
\sMC_n(L) : = \MC \bigl(L \ctensor \Omega_n, \wh{d}_{\Om_n},\wh{Q}_{\Om_n} \bigr)
\end{equation*}
and to each morphism $\Phi \maps (L,d,Q) \to (L',d',Q') \in \cLie$ the simplical map
\begin{equation*} 
\sMC_n(\Phi):=(\wh{\Phi}_{\Om_n})_\ast \maps \MC(L \ctensor \Omega_n)
\to \MC(L' \ctensor \Omega_n).
\end{equation*}
Above, $(\wh{\Phi}_{\Om_n})_\ast$ is the function obtained from Eq.\ \ref{eq:FB} followed by
Eq.\ \ref{eq:MC_F}. 


\section{Homotopy theory in $\cLie$} \label{sec:hmpty-Linf}
We first give a fairly explicit proof that $\cLie$ forms a category of fibrant objects.
Then we show that this CFO structure satisfies the additional axioms \eqref{EA1} and \eqref{EA2} from Sec.\ \ref{sec:acyc}. This implies Thm.\ \ref{thm:Linf-qinv}, which we interpret as an analog of Whitehead's Theorem for $\cinf$-algebras.

\subsection{$\cLie$ as a category of fibrant objects} \label{sec:Linf_cfo}
By definition, the tangent complex $(L,d)$ of a $\cinf$-algebra $(L,d,Q)$ is a complete cochain complex. The tangent functor \eqref{eq:tan} upgrades to a functor
\[
\ctan \maps \cLie \to \cCh.
\] 
From the point of view of many applications, lifting the CFO structure on $\cCh$ (Thm.\ \ref{thm:Ch_cfo}) through the tangent functor $\ctan(-)$ provides good notions of weak equivalence\footnote{See, however, the main result of \cite{Schwarz:2020}.} and fibration in $\cLie$.
\begin{definition}\label{def:Lie-weq}
Let $\Phi \maps (L,d,Q) \to (\pri{L},d',\pri{Q})$ be a morphism in $\cLie$.
We say $\Phi$ is a \df{weak equivalence} if $\ctan(\Phi)$ is a weak equivalence in $\cCh$ in the sense of Def.\ \ref{def:chain-weq}. Similarly,  
we say $\Phi$ is a \df{fibration} if $\ctan(\Ph)$ is a fibration in $\cCh$.

\end{definition}

\begin{theorem} \label{thm:Linf_cfo}
\mbox{}
The category $\cLie$ of complete shifted $L_{\infty}$-algebras has the structure of a category of fibrant objects with functorial path objects in which the weak equivalences and fibrations are those morphisms that satisfy the defining criteria given in Def.\ \ref{def:Lie-weq}
\end{theorem}

\begin{proof}
As in the proof of Thm.\ \ref{thm:Ch_cfo}, we will proceed by directly
verifying axioms 1 through 7 of Def.\ \ref{def:cfo}. Finite products in $\cLie$ were discussed in Sec.\ \ref{sec:product}, and axioms 1,2,3, and 7 are trivial. We verify the remaining axioms 
in the sections below by building on the results of Sec.\ \ref{sec:Ch_cfo}.
\end{proof} 

\subsubsection{Pullbacks of (acyclic) fibrations in $\cLie$ (axioms 4 \& 5)}
We first record a filtered analog of a useful simplification due to B.\ Vallette \cite[Prop.\ 4.3]{Vallette:2014} which allows us to ``strictify'' any fibration in $\cLie$ without altering the linear term.

\begin{proposition} \label{prop:strict_fib}
Let $\Phi \maps (L,d,Q) \fib (L',d',Q')$  be a fibration in $\cLie$ 
Then there exists a $\cinf$-algebra $(\bev{L}, \bev{d},\bev{Q})$ and an 
isomorphism $\Psi \maps (\bev{L},\bev{d},\bev{Q}) \xto{\cong} ({L},d,Q)$ in $\cLie$ such that
\[
\Phi\Psi \maps (\bev{L},\bev{d},\bev{Q}) \to (L',d',Q') 
\]
is a strict fibration with linear term  $(\Phi \Psi)^1_1 = \sPhi_1$.
\end{proposition}
\begin{proof}
Consider the linear term $\sPhi_1 \maps L \to L'$ as a morphism in $\cVect$.
By definition, the restriction 
\[
\cF_n\sPhi_1 \maps \cF_nL \to \cF_nL' 
\]
is a surjection for all $n \geq 1$. Therefore, Lemma \ref{lem:split} implies that there exists a filtration preserving linear map $\si \maps L' \to L$ such that $\sPhi_1 \circ \si = \id_{L}$. 

Define $\bev{L}:=L$ and $\Psi^1_1 :=\id_{\bev{L}}$. Let $m\geq 2$ and suppose we have defined a sequence of degree 0 filtration-preserving linear maps $\{\Psi^1_k \maps \S^k(\bev{L}) \to L\}^{m-1}_{k \geq 1}$. It follows from Eq.\ \ref{eq:map} that this sequence gives us well defined linear maps $\Psi^k_m \maps \S^m(\bev{L}) \to  \S^k( L)$ for $2 \leq k \leq m$. Now define $\Psi^1_m \maps \S^m( \bev{L}) \to L$ as:
\[
\Psi^1_m = -\sum_{k \geq 2}^m \sigma \Phi^1_k \Psi^k_m. 
\]
Clearly $\Psi^1_m$ is filtration preserving, since it is the composition of filtration preserving maps.
This construction inductively yields a coalgebra isomorphism $\Psi \maps \S(\bev{L}) \xto{\cong} \S(L)$. By defining a new codifferential  $\bev{Q}:= \Psi^{-1} Q \Psi$, we promote $\Psi$ to an isomorphism $\Psi \maps \bigl (\bev{L}, \bev{d},\bev{Q} \bigr) \xto{\cong} \bigl (L, d, Q \bigr)$ in $\cLie$. And since $\Psi^1_1 = \id$, it follows from Eq.\ \ref{eq:comp} that we have $(\Phi\Psi)^1_1 = \Phi^1_1 \maps  L \to  L'$.
Finally, to show that $\Phi\Psi$ is strict, i.e.\ $(\Phi \Psi)^1_m=0$ for all $m\geq 2$, we observe that Eq.\ \ref{eq:comp} and the construction of $\Psi$ imply that
\[
(\Phi\Psi)^1_m = - \Phi^1_1 \bigl (\sum_{k \geq 2}^m \sigma \Phi^1_k \Psi^k_m \bigr)  + \sum_{k \geq 2}^m \Phi^1_k \Psi^k_m. 
\]
This completes the proof.
\end{proof}
\begin{remark}\label{rmk:strict_fib} 
By construction, the tangent functor $\ctan \maps \cLie \to \cCh$ is invariant under the ``strictification'' process described in Prop.\ \ref{prop:strict_fib}. That is,
\[
\ctan \Bigl( (\bev{L},\bev{d},\bev{Q}) \xto{\Phi\Psi} (L',d',Q')  \Bigr)  = 
(L,d) \xto{\Phi^1_1} (L',d').
\]
\end{remark}

Next we show that strict fibrations and acyclic fibrations are preserved under pullbacks in $\cLie$. The general case for all (acyclic) fibrations will then follow as a corollary.  

In the proof of Thm.\ 4.1 \cite{Vallette:2014}, B.\ Vallette provides a useful construction of pullbacks along strict $\infty$-epimorphisms in $\SLie$. The construction easily extends to the complete filtered case. Suppose $\Ph=\sPh \maps (L,d,Q) \fib (L'',d'',Q'')$ is a strict fibration,  and $\Theta \maps (L',d',Q') \to (L'',d'',Q'')$ is an arbitrary morphism in $\cLie$. Since $\sPh \maps (L,d) \fib (L',d')$ is a fibration in $\cCh$, Prop.\ \ref{prop:Ch-pb} implies that we have a pullback square of the following form 
\[
\begin{tikzpicture}[descr/.style={fill=white,inner sep=2.5pt},baseline=(current  bounding  box.center)]
\matrix (m) [matrix of math nodes, row sep=2em,column sep=3em,
  ampersand replacement=\&]
  {  
\bigl( \ti{L}, \ti{d} \bigr) \& (L, d) \\
(L',d') \& (L'', d'') \\
}
; 
  \path[->,font=\scriptsize] 
   (m-1-1) edge node[auto] {$\pr \circ h $} (m-1-2)
   (m-1-1) edge node[auto,swap] {$\pr'$} (m-2-1)
   (m-1-2) edge node[auto] {$\sPh$} (m-2-2)
   (m-2-1) edge node[auto] {$\Theta^1_1$} (m-2-2)
  ;

  \begin{scope}[shift=($(m-1-1)!.4!(m-2-2)$)]
  \draw +(-0.25,0) -- +(0,0)  -- +(0,0.25);
  \end{scope}
\end{tikzpicture}
\]
with $\ti{L}= L' \times \ker \sPh$, and $\ti{d}= j\circ (d' \times d) \circ h$. Recall that 
$h \maps L' \times L \xto{\cong} L' \times L$ and $j=h^{-1} \maps L' \times L \xto{\cong} L' \times L$ are the filtered linear maps 
\[
h(v',v):=(v', \si \sThe_1(v') + v), \quad j(v',v):=(v', v - \si \sThe_1(v') ),
\]
where $\si \maps L'' \to L$ is a section in $\cVect$ of the surjection $\sPh$.

We denote elements of the product $ L' \times L$ as vectors  $\bb{v}:=(v',v)$.
In order to lift the above pullback diagram to $\cLie$, we follow \cite{Vallette:2014} and
define the linear maps $H^1_k \maps \S^k(L' \times L) \to  L' \times  L$ to be
\begin{equation*} 
 H^1_1(\bb{v}):= h(\bb{v}), \quad
H^1_{k}(\bb{v}_1,\ldots,\bb{v}_k):= \bigl(0,  \si\sThe_{k}(v'_1,\ldots,v'_k) \bigr).
\end{equation*}
Similarly, let $J^1_k \maps \S^k( L' \times L) \to  L' \times L$ denote the linear maps
\begin{equation} \label{eq:J}
 J^1_1(\bb{v}):= j(\bb{v}), \quad
J^1_{k}(\bb{v}_1,\ldots,\bb{v}_k):= \bigl(0, - \si \sThe_{k}(v'_1,\ldots,v'_k) \bigr).
\end{equation}
A direct calculation shows that $H  J = J  H = \id_{\S(L' \times L)}$.

Finally, since $\S(\ti{L}) \sse \S(L' \times L)$, we can define 
a degree 1 codifferential $\ti{Q} \maps \S(\ti{L}) \to \S(\ti{L})$ via
\begin{equation} \label{eq:deltilde}
\ti{Q}:= J Q_{\times}  H \vert_{\S(\ti{L})}
\end{equation}
where $Q_{\times}$ is the $\cinf$-structure  on $L' \times L$ induced by $Q'$ and $Q$. By construction, $\ti{Q}^1_1 = \ti{d}$ and $\ti{Q}^2 =0$, and a direct calculation shows that indeed $\ti{Q}$ is well-defined, i.e., $ \im \ti{Q} \sse \S(\ti{L})$. (See, for example, \cite[Sec.\ 4.1]{Rogers:2020} for details of an almost identical calculation.) 

Since $H$, $J$, and $\ti{Q}$ are constructed from filtration preserving maps, we conclude that
$(\ti{L},\ti{Q}) \in \Lie$, and 
\begin{equation} \label{eq:H-morph}
H \vert_{\S(\ti{L})} \maps (\ti{L}, \ti{d},\ti{Q}) \to (L' \times L, d_{\times},Q_\times) 
\end{equation} 
is a morphism in $\cLie$. Let $\Pr'$ and $\Pr$ denote the canonical projections out of the product $(L' \times L, d_{\times}, Q_\times)$. We then have:

\begin{proposition} \label{prop:strict-pb}
If $\Ph=\sPh \maps (L,d,Q) \fib (L'',d'',Q'')$ is a strict fibration (acyclic fibration) 
and $\Theta \maps (L',d',Q') \to (L'',d'',Q'')$ is an arbitrary morphism in $\cLie$, then 
\begin{equation} \label{diag:strict_pullback3}
\begin{tikzpicture}[descr/.style={fill=white,inner sep=2.5pt},baseline=(current  bounding  box.center)]
\matrix (m) [matrix of math nodes, row sep=2em,column sep=3em,
  ampersand replacement=\&]
  {  
(\ti{L},\ti{d},\ti{Q})  \& \bigl (L,d, Q  ) \\
\bigl   (L', d', Q') \& \bigl (  L'', d'', Q'' ) \\
}; 
  \path[->,font=\scriptsize] 
   (m-1-1) edge node[auto] {$\ppr H$} (m-1-2)
   (m-1-1) edge node[auto,swap] {$\ppr'H$} (m-2-1)
   (m-1-2) edge node[auto] {$\Ph$} (m-2-2)
   (m-2-1) edge node[auto] {$\Theta$} (m-2-2)
  ;

  \begin{scope}[shift=($(m-1-1)!.4!(m-2-2)$)]
  \draw +(-0.25,0) -- +(0,0)  -- +(0,0.25);
  \end{scope}
\end{tikzpicture}
\end{equation}
is a pullback diagram in the category $\cLie$.
Moreover, the morphism $\ppr' H \maps( \ti{L}, \ti{d},\ti{Q}) \to (L',d',Q')$ is a fibration (acyclic fibration) in $\cLie$.
\end{proposition}

\begin{proof}
Theorem 4.1 of \cite{Vallette:2014} implies that the diagram \eqref{diag:strict_pullback3}
is the pullback in the category $\SLinf$. Indeed, if $(V,d^V,Q^V) \xto{\Psi} (L,d,Q)$ and $(V,d^V,Q^V) \xto{\Psi'} (L',d',Q')$ are morphisms in $\Lie$ such that $\Phi\Psi=\Theta \Psi'$, then the morphism
$J \cc (\Psi \tensor \Psi') \maps (V,d^V,Q^V) \to (\ti{L}, \ti{d},\ti{Q})$
is the unique map which makes the relevant diagrams commute. (See Sec.\ \ref{sec:product} for a reminder of the notation $\Psi \tensor \Psi'$.) 

It is then clear that if $\Psi$ and $\Psi'$ are morphisms in $\cLie$, then so is 
$J \cc (\Psi \tensor \Psi')$. Hence, diagram \eqref{diag:strict_pullback3} lifts to a pullback diagram in $\cLie$. Finally, it follows from Prop.\ \ref{prop:Ch-pb} that $\Pr'H$ is a fibration (acyclic fibration).
\end{proof}

Now we remove the assumption in Prop.\ \ref{prop:strict-pb} that the fibration $\Ph$ is strict.
\begin{corollary} \label{cor:pb}
Let $\Ph\maps (L,d,Q) \fib (L'',d'',Q'')$ be a fibration (acyclic fibration).
Let\\ ${\Theta \maps (L',d',Q') \to (L'',d,'',Q'')}$ be an arbitrary morphism in $\cLie$.
Then the pullback of the diagram 
\[
(L',d',Q') \xto{\Theta} (L'',d'',Q'') \xleftarrow{\Ph} (L,d,Q)
\]
exists in $\cLie$. Furthermore, the canonical projection from the pullback to $(L',d',Q')$ is a (acyclic) fibration.
\end{corollary}

\begin{proof}
Proposition \ref{prop:strict_fib} implies that there exists an isomorphism
$\Psi \maps (\bev{L},\bev{Q}) \xto{\cong} ({L},Q)$ in $\cLie$ such that 
$\bev{\Ph}:=\Ph \Psi \maps (\bev{L},\bev{d},\bev{Q}) \to (L'',d'',Q'')$ is a strict fibration with 
linear term  $(\bev{\Ph})^1_1 = \sPh$. Now we apply Prop.\ \ref{prop:strict-pb} to the strict fibration $\bev{\Ph}$ and $\Theta$ and obtain the pullback diagram
\[
\begin{tikzpicture}[operad style,row sep=2em,column sep=3em,]
\matrix (m) [matrix of math nodes]
  {  
(\ti{L}, \ti{d},\ti{Q})  \& \bigl (  \bev{L}, \bev{d},\bev{Q}  ) \\
\bigl   (L', d',Q') \& \bigl (  L'', d'', Q'' ) \\
}; 
  \path[->,font=\scriptsize] 
   (m-1-1) edge node[auto] {$\ppr H$} (m-1-2)
   (m-1-1) edge node[auto,swap] {$\ppr'H$} (m-2-1)
   (m-1-2) edge node[auto] {$\bev{\Ph}$} (m-2-2)
   (m-2-1) edge node[auto] {$\Theta$} (m-2-2)
  ;

  \begin{scope}[shift=($(m-1-1)!.4!(m-2-2)$)]
  \draw +(-0.25,0) -- +(0,0)  -- +(0,0.25);
  \end{scope}
\end{tikzpicture}
\]
in which $\Pr'H$ is a (acyclic) fibration. Then, since $\Psi$ is an isomorphism in $\cLie$, 
the diagram 
\begin{equation*}  
\begin{tikzpicture}[operad style,row sep=2em,column sep=4em,]
\matrix (m) [matrix of math nodes]
  {  
(\ti{L},\ti{d},\ti{Q})  \& \bigl (L,d, Q  ) \\
\bigl   (L',d', Q') \& \bigl (  L'', d'', Q'' ) \\
}; 
  \path[->,font=\scriptsize] 
   (m-1-1) edge node[auto] {$\Psi\circ (\ppr H)$} (m-1-2)
   (m-1-1) edge node[auto,swap] {$\ppr'H$} (m-2-1)
   (m-1-2) edge node[auto] {$\Phi$} (m-2-2)
   (m-2-1) edge node[auto] {$\Theta$} (m-2-2)
  ;

  \begin{scope}[shift=($(m-1-1)!.4!(m-2-2)$)]
  \draw +(-0.25,0) -- +(0,0)  -- +(0,0.25);
  \end{scope}
\end{tikzpicture}
\end{equation*}
is also a pullback in $\cLie$.
\end{proof}

\subsubsection{Functorial path object for $\cLie$ (axiom 6)}
From our definition of weak equivalences and fibration in $\cLie$, Prop.\ \ref{prop:Ch-pathobj} immediately implies
\begin{proposition}\label{prop:cLie-pathobj}
To each  $\cinf$-algebra $\LdQ{}$, the assignment 
\[
\LdQ{} \quad \mapsto \quad \bigl(L \ctensor \Om_1, d_{\Om}, Q_{\Om}\bigr)
\]
is a functorial path object.
\end{proposition}

\subsection{Homotopy inverses of weak equivalences in $\cLie$} \label{sec:qinverse-Lie}

We begin with a decomposition lemma for acyclic fibrations.
\begin{lemma} \label{lem:min_mod}
Let $\Ph \maps (L, d, Q) \afib (L',d',Q')$ be an acyclic fibration in $\cLie$.
Let $(\ker \sPh, d)$ denote the kernel of the chain map $\ctan(\Ph)=\sPh \maps (L, d) \to (L',d')$ considered as an abelian $\sinf$-algebra. Then there exists a morphism 
$\Psi \maps (L, d, Q) \to (\ker \sPh, d)$ in $\cLie$
such that the morphism $\bigl(\Ph,\Psi \bigr) \maps (L,d,Q) \to (L' \times \ker \sPh, d_\times, Q_\times)$ induced via the universal property of the product
is an isomorphism of $\cinf$-algebras.

\end{lemma}

\begin{proof}
Since $\Phi$ is an acyclic fibration, the chain map $\sPh \maps (L, d) \to (L',d')$
is an acyclic fibration in $\cCh$. Proposition \ref{prop:Ch-acyc} implies that 
there exists a filtered chain map $\tau \maps (L',d') \to (L,d)$ in $\cCh$ and a filtered chain homotopy $h \maps L \to L[-1]$  such that
\begin{equation} \label{eq:min_mod1}
\sPh \cc \tau = \id_{L'}, \quad \sPsi_1 = dh + hd.
\end{equation}
where  $\sPsi_1 \maps (L,d) \to (\ker \sPh , d)$ is the chain map $\sPsi_1 := \id_L - \tau \cc \sPh$.
For each $n \geq 2$, let $\sPsi_n \maps \S^n(L) \to \ker \sPh$ be the linear map
\begin{equation*} 
\sPsi_{n }:= \sPsi_1 \circ h \circ \sQ_n.
\end{equation*}
Since $\ta$, $h$ and $\sQ_k$ are filtration preserving, so is $\sPsi \maps \S(L) \to \ker \sPh$. A direct calculation using \eqref{eq:min_mod1} and the fact that $Q^2=0$, shows
that for each $n \geq 1$
\[
d \vert_{\ker} \circ \sPsi_n = \sPsi_1 \sQ_n + \sum_{k=2}^n \sPsi_k Q^k_n.
\]
Hence $\sPsi$ defines a morphism in $\cLie$.
Finally, we note that the linear map $\tha \maps  L' \times \ker \sPh \to L$, defined as
$\tha(x',z):= \ta(x') + z$
is filtration preserving and inverse to $(\sPh,\sPsi_1)$ in $\cVect$. Hence, the morphism
$(\Psi,\Phi)$ is an isomorphism in $\cLie$.
\end{proof}

The above lemma implies that every acyclic fibration in $\cLie$ admits a right inverse.

\begin{corollary} \label{cor:acyc-retract}
Let $\Ph \maps (L, d, Q) \afib (L',d',Q')$ be an acyclic fibration in $\cLie$. Then there exists
a morphism $\chi \maps (L',d',Q') \to (L,d,Q)$ in $\cLie$ such that $\Ph\circ \chi = \id_{L'}$.
\end{corollary}

\begin{proof}
Let $\bigl(\Ph,\Psi \bigr) \maps (L,d,Q) \xto{\cong} (L' \times \ker \sPh, d_\times, Q_\times)$ denote the isomorphism from Lemma \ref{lem:min_mod}. From Prop.\ \ref{prop:Linf-inc} we have the injective $L_\infty$-morphism $i_{L'} \maps (L',d',Q') \to (L' \times \ker \sPh, d_\times, Q_\times)$. Let $\chi:= (\Phi,\Psi)^{-1} \cc i_{L'}$. Then $\Phi \cc \chi = \mathrm{Pr}_{L'}(\Phi,\Psi) \cc (\Phi,\Psi)^{-1} \cc i_{L'} = \id_{L'}.$
\end{proof}

It now follows that every weak equivalence in $\cLie$ has a homotopy inverse via any choice of path object. In particular
\begin{theorem} \label{thm:Linf-qinv}
If $\Phi \maps \LdQ{} \weq \LdQ{\prime}$ is a weak equivalence in $\cLie$, then there exists a 
weak equivalence $\Psi \maps \LdQ{\prime} \weq \LdQ{}$ and homotopies 
\[
\mathcal{H}_L \maps (L,d,Q) \to (L \ctensor \Om_1, d_\Om,Q_\Om), \quad  \mathcal{H}_{L'} \maps (L',d',Q') \to (L' \ctensor \Om_1, d'_\Om, Q'_\Om), 
\] 
which induce equivalences $\Psi \Phi \simeq \id_{L}$  and $\Phi \Psi \simeq \id_{L'}$.
\end{theorem}
\begin{proof}
Cor.\ \ref{cor:acyc-retract} implies that $\cLie$ satisfies Axiom A1 from Sec.\ \ref{sec:acyc}. Hence, the result follows from Prop.\ \ref{prop:bifib-invert}.
\end{proof}

\section{The functor $\sMC \maps \cLie \to \Kan$} \label{sec:sMC-htpy}
In this section, we prove that the simplicial Maurer-Cartan functor is an exact functor (Def.\ \ref{def:exact_functor}) between the category of complete $\sinf$-algebras and the category of Kan simplicial sets. As a corollary, we deduce that the simplicial Maurer-Cartan functor preserves homotopy pullbacks up to homotopy equivalence. 

\subsection{Kan conditions for $\sMC(L)$}
Let us first consider the bounded case.  We recall E.\ Getzler's theorem  concerning the simplicial Maurer-Cartan set for nilpotent $\sinf$-algebras.
\begin{theorem}[Prop.\ 4.7 \cite{Getzler}] \label{thm:sMCnil}
Let $\Ph=\sPh \maps \LdQ{} \to \LdQ{\prime}$  be a strict $\infty$-epimorphism in $\nilstLie$. Then
$\sMC(\Ph) \maps \sMC(L) \to \sMC(L')$ is a Kan fibration. In particular, $\sMC(L)$ is a Kan simplicial set for all $(L,d,Q)\in \nilstLie$.  
\end{theorem}

Recall from \eqref{eq:bnd-nil} the discussion of the forgetful functor 
$\bndfltstLie \xto{\jmath_{\str}} \nilstLie$ which sends a bounded filtered $\cinf$-algebra to its underlying nilpotent $\sinf$-algebra. Remark \ref{rmk:bnd-nil} implies that
$\MC(L) = \MC \bigl(\jmath_{\str}(L)\bigr)$. Therefore, from Prop.\ \ref{prop:nilbnd-cdga}, we have 
$\MC(L \ctensor B) = \MC \bigl(\jmath_{\str}(L) \tensor B\bigr)$ 
for every $\LdQ{} \in \bdfltstLie$ and every $(B,\del)\in \cdga$. The next proposition is just for the purposes of bookkeeping. It follows immediately from Thm.\ \ref{thm:sMCnil}. 

\begin{proposition} \label{prop:sMCbnd}
Let $\Ph=\sPh \maps \LdQ{} \fib \LdQ{\prime}$ be a strict fibration in $\bdfltstLie$. Then
$\sMC(\Ph) \maps \sMC(L) \to \sMC(L')$.
is a Kan fibration. In particular, $\sMC(L)$ is a Kan simplicial set for all $\LdQ{} \in \bdfltstLie$
\end{proposition}
\begin{corollary} \label{cor:sMCcomp}
Let $\LdQ{} \in \cLie$. Then $\sMC(L)$ is a Kan simplicial set, and the functor \eqref{eq:sMC} factors as $\sMC \maps \cLie \to \Kan$.
\end{corollary}
\begin{proof}
This follows from Prop.\ \ref{prop:sMCbnd} and the fact that $\sMC(L) \cong \plim \sMC(\qf{L}{n})$. See, for example, \cite[Prop.\ 4.1]{Enhanced} and \cite[Thm.\ 6.9]{Markl-book}.
\end{proof}

\subsection{$\sMC$ as an exact functor} \label{sec:exact}
The goal of this subsection is to prove:
\begin{theorem} \label{thm:sMC-exact}
The simplicial Maurer-Cartan functor $\sMC \maps \cLie \to \Kan$ is an exact functor between categories of fibrant objects.
\end{theorem}
Note that both $\cLie$ and $\Kan$ satisfy the additional axioms \eqref{EA1} and \eqref{EA2}. Hence, Thm.\ \ref{thm:sMC-exact} combined with Prop.\ \ref{prop:exact-hpb} immediately implies:
\begin{corollary}\label{cor:MC-hpb}
The functor $\sMC \maps \cLie \to \Kan$ preserves homotopy pullbacks, up to homotopy equivalence.
\end{corollary}

Our proof of Thm.\ \ref{thm:sMC-exact} will proceed by directly
verifying the axioms of Def.\ \ref{def:exact_functor} via the following collection of propositions. 

\begin{proposition}\label{prop:sMC-prod}
Let $\LdQ{},\LdQ{\prime} \in \cLie$. Then $\sMC(L \times L') \cong \sMC(L) \times \sMC(L')$.

\end{proposition}
\begin{proof}
From the definition \eqref{eq:product1} for the product in $\cLie$ we have
$\curv_{\times} = \curv_{L} \times \curv_{L'}$, where $\curv_{\times}$ is the curvature function for
$(L \times L', d_{\times}, Q_{\times})$. From here, it is straightforward to show that there is a natural isomorphism of sets $\MC \bigl ( (L' \times L) \ctensor B \bigr) \cong \MC(L \ctensor B) \times \MC(L' \ctensor B)$
for any $L,L' \in \cLie$ and $B \in \cdga$. See, for example, \cite[Ex.\ 6.17]{Markl-book}.
\end{proof}

Let us next address weak equivalences.  In previous work \cite{GM_Theorem} with V.\ Dolgushev, the following generalization of the classical Goldman-Millson theorem from deformation theory was established.

\begin{proposition}[Theorem 1.1 \cite{GM_Theorem}] \label{prop:GM}
Let $\Phi \maps \LdQ{} \weq \LdQ{\prime}$ be a weak equivalence in $\cLie$. Then
$\sMC(\Phi) \maps \sMC(L) \to \sMC(L')$ is a homotopy equivalence between simplicial sets.
\end{proposition}
Hence, it remains to verify that $\sMC$ preserves fibrations and pullbacks of fibrations.
We split the work across the next two subsections.

\subsubsection{$\sMC$ preserves fibrations}
Our strategy for verifying that $\sMC$ preserves fibrations involves combining the ``strictification''
result of Prop.\ \ref{prop:strict_fib} along with the results from Sec.\ \ref{sec:towmodcat} on towers in CFOs.   
 
\begin{proposition} \label{prop:sMC-fib} 
Let $\Ph \maps (L,d,Q) \fib (L',d',Q')$ be a fibration in $\cLie$. Then the simplicial map
$\sMC(\Ph) \maps \sMC(L) \to \sMC(L')$ is a Kan fibration.
\end{proposition} 

\begin{proof}
By Prop.\ \ref{prop:strict_fib}, it suffices to consider the case when $\Ph=\sPh$ is a strict fibration in $\cLie$. Applying the functor $\ttow \maps \cLie \to \tbndfiltLie$ gives us a morphism, as in statement 4 of Prop.\ \ref{prop:quot-filt},
between towers of bounded  $\cinf$-algebras,
in which all horizontal and all vertical morphisms are strict fibrations.
It then follows from Prop.\ \ref{prop:sMCbnd} that we obtain a map of towers in $\Kan$ in which all 
horizontal and all vertical maps are Kan fibrations 
\begin{equation}\label{diag:L_tow}
\begin{tikzpicture}[operad style,row sep=2em,column sep=3em,]
\matrix (m) [matrix of math nodes]
{  
\cdots \&  \sMC (L / \cF_{n}L) \& \sMC(L / \cF_{n-1}L)  \&  \cdots \\
\cdots \&  \sMC(\pri{L} / \cF_{n}\pri{L}) \& \sMC(\pri{L} / \cF_{n-1} \pri{L})  \&  \cdots \\
}; 

\path[->>,font=\scriptsize] 
(m-1-2) edge node[auto,swap] {$\qq{\phi}{n} $} (m-2-2)
  (m-1-3) edge node[auto,swap] {$ \qq{\phi}{n-1} $} (m-2-3)
;

\path[->>,font=\scriptsize] 
(m-1-1) edge node[auto] {$\qq{\ro}{n}$} (m-1-2)
(m-2-1) edge node[auto] {$\qq{\ro'}{n}$} (m-2-2)
(m-1-2) edge node[auto] {$\qq{\ro}{n-1}$} (m-1-3)
(m-2-2) edge node[auto] {$\qq{\ro'}{n-1}$} (m-2-3)
(m-1-3) edge node[auto] {$\qq{\ro}{n-2}$} (m-1-4)
(m-2-3) edge node[auto] {$\qq{\ro'}{n-2}$} (m-2-4)
;
\end{tikzpicture}
\end{equation}
Above, for the sake of brevity, we have used the notation 
$\qq{\phi}{k} := \sMC(\qq{\Ph}{k})$ and $\qq{\ro}{k} := \sMC(\qq{\bp}{k})$. Since all $\infty$-morphisms involved are strict, we have
\begin{equation*} 
\qq{\phi}{k} = (\qq{\Ph}{k})^1_1\tensor \id_{\Om_\bul}, \qquad \qq{\ro}{k} = \qq{\bp}{k}\tensor \id_{\Om_\bul},
\qquad \qq{\ro'}{k} = \qq{\bp'}{k}\tensor \id_{\Om_\bul}.
\end{equation*}
We will use Prop.\ \ref{prop:towmodcat} to show that $\plim \qq{\phi}{k} \maps \sMC(L) \to \sMC(L')$ is a Kan fibration. 
Let $n \geq 1$. From diagram \ref{diag:towmodcat} in Prop.\ \ref{prop:towmodcat}, we see that it suffices to prove that the map
\begin{equation*} 
(\qq{\ro}{n},\qq{\phi}{n+1}) \maps \sMCq{n+1} \to \sMCq{n} \times_{\sMCqq{n}} \sMCqq{n+1}
\end{equation*}
is a fibration. So suppose we have a horn  $\gamma \maps   \horn{m}{k} \to \sMCq{n+1}$ and commuting diagrams of simplical sets
\begin{equation}\label{diag:sMC-fib-horns}
\begin{tikzpicture}[operad style,row sep=2em,column sep=2em,]
\matrix (m) [matrix of math nodes]
 {  
\horn{m}{k} \& \sMCq{n+1} \\
\Delta^m \& \sMCq{n}\\
}; 
  \path[->,font=\scriptsize] 
   (m-1-1) edge node[auto] {$\gamma$} (m-1-2)
   (m-1-1) edge node[auto] {$$} (m-2-1)
   (m-2-1) edge node[auto] {${\beta}$} (m-2-2)

  ;
  \path[->>,font=\scriptsize] 
   (m-1-2) edge node[auto] {$\sMCr{n}$} (m-2-2)
;
\end{tikzpicture}
\quad 
\begin{tikzpicture}[operad style,row sep=2em,column sep=2em,]
\matrix (m) [matrix of math nodes]
  {  
\horn{m}{k} \& \sMCq{n+1} \\
\Delta^m \& \sMCqq{n+1}\\
}; 
  \path[->,font=\scriptsize] 
   (m-1-1) edge node[auto] {$\gamma$} (m-1-2)
   (m-1-1) edge node[auto] {$$} (m-2-1)
   (m-2-1) edge node[auto] {$\pri{\beta}$} (m-2-2)
   (m-1-2) edge node[auto] {$\sMCp{n+1}$} (m-2-2)
  ;
\end{tikzpicture}
\quad
\begin{tikzpicture}[operad style,row sep=2em,column sep=2em,]
\matrix (m) [matrix of math nodes]
  {  
\Delta^m \& \sMCqq{n+1}\\
\sMCq{n} \& \sMCqq{n}\\
}; 
  \path[->,font=\scriptsize] 
   (m-1-1) edge node[auto] {$\beta'$} (m-1-2)
   (m-1-1) edge node[auto,swap] {$\beta$} (m-2-1)
   (m-2-1) edge node[auto] {$\sMCp{n}$} (m-2-2)
  ;

  \path[->>,font=\scriptsize] 
  (m-1-2) edge node[auto] {$\sMCrr{n}$} (m-2-2)
  ;
\end{tikzpicture}
\end{equation}
To complete the proof, we will construct an $m$-simplex $\ti{\alpha} \maps \Delta^m \to \sMCq{n+1}$ filling $\gamma$ and satisfying 
\begin{equation} \label{eq:sMC-fib2}
\sMCr{n}  \ti{\al} = \beta, \quad \sMCp{n+1} \ti{\al} = \pri{\beta}.
\end{equation}


To begin with, since $\sMCr{n}$ is a fibration, there exists $\al \maps \Simp{m} \to \sMCq{n+1}$ that fills $\ga$ and satisfies $\sMCr{n}\al = \be$. But, in general, $\sMCp{n+1}\al \neq \be'$.
By definition, $\al \in \MC(\qf{L}{n+1} \ctensor \Om_m)$. Proposition \ref{prop:quot-filt} implies that
$(\qf{L}{n+1}, \qq{d}{n+1}, \qq{Q}{n+1})$ is bounded at $n+1$. Hence,  by Prop.\ \ref{prop:nilbnd-cdga}, the $\cinf$-algebra $(\qf{L}{n+1} \tensor \Om_m, d_{\pp{n+1} \Om}, Q_{\pp{n+1} \Om})$
is also bounded  at $n+1$, and therefore already complete. So we may consider $\al$ as a vector in $(\qf{L}{n+1} \tensor \Om_m)^{0}$. Similarly, we may consider $\sMCp{n+1}(\al)$ and $\be'$ as vectors in $(\qf{L'}{n+1} \tensor \Om_m)^0$. Let $\eta \in  (\qf{L'}{n+1} \tensor \Om_m)^{0}$ denote the vector
\[
\eta := \be' - \sMCp{n+1}(\al). 
\] 
From diagrams \eqref{diag:L_tow} and \eqref{diag:sMC-fib-horns} we deduce that
$\eta \in \ker (\qq{\bp'}{n} \tensor \id_{\Om_m}) \cong \cF_{n}L' / \cF_{n+1}L' \tensor \Om_m.$

Next, recall from Prop.\ \ref{prop:quot-filt}, that the cochain complex $(\cF_{n}L' / \cF_{n+1}L', \qq{d'}{n+1})$ is an abelian sub-$\sinf$-algebra of 
$(\qf{L'}{n+1}, \qq{d'}{n+1}, \qq{Q'}{n+1})$.  We claim that 
\[
\eta \in \sMC_m \bigl(\cF_{n}L' / \cF_{n+1}L' \bigr) \cong Z^0\Bigl (\cF_{n}L' / \cF_{n+1}L' \tensor \Om_m, d'_{\pp{n+1} \Om} \Bigr). 
\]
Indeed, $(\qf{L'}{n+1} \tensor \Om_m, d'_{\pp{n+1} \Om}, Q'_{\pp{n+1} \Om})$ is a $\cinf$-algebra bounded at $n+1$ and $\eta$ has filtration degree $n$. Therefore Lemma \ref{lem:bnd-curv} implies that
\[
\curv'_{\pp{n+1} \Om}(\be') = \curv'_{\pp{n+1} \Om}\bigl(\sMCp{n+1}(\al) +\eta \bigr) = 
\curv'_{\pp{n+1} \Om}\bigl(\sMCp{n+1}(\al) \bigr) + d'_{\pp{n+1} \Om} \eta.
\] 
Hence, $d'_{\pp{n+1} \Om} \eta =0$, since $\be'$ and $\sMCp{n+1}(\al)$ are Maurer-Cartan elements.

Furthermore, since the diagrams from \eqref{diag:sMC-fib-horns} imply that both 
$\be'$ and $\sMCp{n+1}(\al)$ fill the horn $\sMCp{n+1}\ga$, the $m$-simplex $\eta$ fits into the commutative diagram of Kan complexes
\begin{equation} \label{diag:sMC-eta}
\begin{tikzpicture}[operad style,row sep=2em,column sep=2em,]
\matrix (m) [matrix of math nodes]
  {  
\horn{m}{k} \& \sMC(\cF_nL/\cF_{n+1}L) \\
\Delta^m \& \sMC(\cF_nL'/\cF_{n+1}L')\\
}; 
  \path[->,font=\scriptsize] 
   (m-1-1) edge node[auto] {$0$} (m-1-2)
   (m-1-1) edge node[auto] {$$} (m-2-1)
   (m-2-1) edge node[auto] {$\eta$} (m-2-2)

  ;
  \path[->>,font=\scriptsize] 
   (m-1-2) edge node[auto] {$\sMCp{n+1} \vert_{\cF_nL} $} (m-2-2)
;
\end{tikzpicture}
\end{equation}
In the above diagram, the restriction $\sMCp{n+1} \vert_{\cF_nL}$ of $\sMCp{n+1}$ to the fiber $\sMC(\cF_nL/\cF_{n+1}L)$ is, equivalently, the Maurer-Cartan functor applied to the strict morphism
of abelian $\cinf$-algebra
\[
\cF_{n}\qq{(\Ph^1_1)}{n+1} \maps \bigl(\cF_nL/\cF_{n+1}L, \qq{d}{n+1} \bigr) \to 
\bigl(\cF_nL'/\cF_{n+1}L', \qq{d'}{n+1} \bigr). 
\] 
Since $\Phi=\Ph^1_1$ is a fibration, the linear map $\cF_n \Phi^1_1 \maps \cF_nL \to \cF_nL'$ is surjective. Therefore $\cF_{n}\qq{(\Ph^1_1)}{n+1}$ is a fibration in $\bdfltstLie$, and so,  
by Prop.\ \ref{prop:sMCbnd} we conclude that $\sMCp{n+1} \vert_{\cF_nL}$ is a Kan fibration.

Finally, let $\tha \maps \Delta^m \to \sMC(\cF_nL/\cF_{n+1}L)$ be a lift of $\eta$ 
through the fibration $\sMCp{n+1} \vert_{\cF_nL}$ in diagram \eqref{diag:sMC-eta} such that
$\tha \vert_{\horn{m}{k}} =0$. Define $\ti{\al} \in (\qf{L}{n+1} \tensor \Om_m)^0$ to be the vector
\[
\ti{\al}:= \al + \tha.
\]
We claim that $\ti{\al} \in \MC(\qf{L}{n+1} \tensor \Om_m)$. Indeed, since $\al$ is a Maurer-Cartan element and since 
\[
\tha \in Z^0(\cF_nL/\cF_{n+1}L\tensor \Om_m, d_{\pp{n+1} \Om})
\]
has filtration degree $n$ in the bounded $\cinf$-algebra $\bigl(\qf{L}{n+1} \tensor \Om_m, d_{\pp{n+1} \Om}, Q_{\pp{n+1} \Om}\bigr)$, Lemma \ref{lem:bnd-curv} implies that
\[
\curv_{\pp{n+1} \Om}(\ti{\al}) = \curv_{\pp{n+1} \Om}(\al) + d_{\pp{n+1} \Om} \tha =0.
\]
Hence, we have an $m$-simplex $\ti{\al} \maps \Delta^m \to \sMCq{n+1}$ which, by construction, fills $\gamma$ and satisfies the equalities \eqref{eq:sMC-fib2}. This completes the proof.
\end{proof}

\subsubsection{$\sMC$ preserves pullbacks of fibrations} \label{sec:MC-pb}
Lastly, we prove that the simplicial Maurer-Cartan functor preserves pullbacks of fibrations.
As in the proof of Prop.\ \ref{prop:sMC-fib}, we exploit the fact that
every fibration in $\cLie$ can be ``strictified''. Let $\Ph \maps (L,d,Q) \fib (L'',d'',Q'')$ be a fibration in $\cLie$ and $\Tha \maps (L',d', Q') \to (L'',d'',Q'')$ be an arbitrary morphism. Let $(\ti{L},\ti{d},\ti{Q})$ be the pullback of the diagram
$(L',d',Q') \xto{\Tha} (L'',d'',Q'') \xleftarrow{\Ph} (L,d,Q)$.  We verify that $\sMC(\ti{L})$ is the pullback of the diagram
$\sMC(L') \xto{\sMC(\Tha)} \sMC(L'') \xleftarrow{\sMC(\Ph)} \sMC(L)$ in $\Kan$.

We consider the special case when $\Ph =\sPh \maps (L,d,Q) \fib (L'',d'',Q'')$ is a strict fibration, since, as in the proof of Cor.\ \ref{cor:pb}, the more general case will follow almost automatically.
We have the pullback diagram \eqref{diag:strict_pullback3} from Prop.\ \ref{prop:strict-pb}:
\[
\begin{tikzpicture}[operad style,row sep=2em,column sep=4em,]
\matrix (m) [matrix of math nodes]
  {  
(\ti{L}, \ti{d},\ti{Q})  \& \bigl (  L, d,Q  ) \\
\bigl   (L', d', Q') \& \bigl (  L'', d'',  Q'' ) \\
}; 
  \path[->,font=\scriptsize] 
   (m-1-1) edge node[auto] {$\ppr H$} (m-1-2)
   (m-2-1) edge node[auto] {$\Tha$} (m-2-2)
  ;

  \path[->>,font=\scriptsize] 
   (m-1-1) edge node[auto,swap] {$\ppr'H$} (m-2-1)
   (m-1-2) edge node[auto] {$\Ph$} (m-2-2)
;
  \begin{scope}[shift=($(m-1-1)!.4!(m-2-2)$)]
  \draw +(-0.25,0) -- +(0,0)  -- +(0,0.25);
  \end{scope}
\end{tikzpicture}
\]
We wish to show that for all $m \geq 0$
\[
\begin{tikzpicture}[operad style,row sep=2em,column sep=4em,]
\matrix (m) [matrix of math nodes]
  {  
\sMC_m(\ti{L})  \& \sMC_m ( L ) \\
  \sMC_m(L') \& \sMC_m(L'') \\
}; 
  \path[->,font=\scriptsize] 
   (m-1-1) edge node[auto] {$\sMC_m(\ppr H)$} (m-1-2)
   (m-2-1) edge node[auto] {$\sMC_m(\Tha)$} (m-2-2)
   (m-1-1) edge node[auto,swap] {$\sMC_m(\ppr'H)$} (m-2-1)
   (m-1-2) edge node[auto] {$\sMC_m(\Ph)$} (m-2-2)
;
\end{tikzpicture}
\]
is a pullback diagram of sets. As mentioned in Remark \ref{rmk:MC}, there is a natural isomorphism
\[
\sMC_m ( L )= \MC(L \ctensor \Om_m) \cong \plim_n \MC(L/\cF_nL \tensor \Om_m).   
\]
Since projective limits commute with pullbacks, it suffices to prove the following:

\begin{proposition} \label{prop:MC-pb}
Suppose $\Ph \maps (L,d,Q) \fib (L'',d'',Q'')$ is a strict fibration in $\bndfiltLie$ and $\Tha \maps (L',d',Q') \to (L'',d'',Q'')$ is an arbitrary morphism. Let $(\ti{L},\ti{d},\ti{Q})$ be the pullback of the diagram
$(L',d',Q') \xto{\Tha} (L'',d'',Q'') \xleftarrow{\Ph} (L,d,Q)$ in $\cLie$
as constructed in Prop.\ \ref{prop:strict-pb}. Then for any $B \in \cdga$, 
the commutative diagram of sets
\[
\begin{tikzpicture}[operad style,row sep=2em,column sep=4em,]
\matrix (m) [matrix of math nodes]
  {  
\MC(\ti{L} \tensor B)  \& \MC( L \tensor B ) \\
  \MC(L'\tensor B) \& \MC(L''\tensor B) \\
}; 
  \path[->,font=\scriptsize] 
   (m-1-1) edge node[auto] {$(\ppr H)_{B\ast}$} (m-1-2)
   (m-2-1) edge node[auto] {$\Tha_{B \ast}$} (m-2-2)
   (m-1-1) edge node[auto,swap] {$(\ppr'H)_{B \ast}$} (m-2-1)
   (m-1-2) edge node[auto] {$\Ph_{B \ast}$} (m-2-2)
;
\end{tikzpicture}
\]
is a pullback diagram.
\end{proposition}
\begin{proof}
First, we fix some notation. Let $P \sse (L' \tensor B)^0  \times (L \tensor B)^0$ denote the pullback of the diagram 
\[
(L' \tensor B)^0 \xto{\Tha_{B \ast}} (L'' \tensor B)^0 \xleftarrow{\Ph_{B\ast}} (L \tensor B)^0
\]
in the category of sets, and let $P_{\MC}$ denote the pullback of the
diagram
$\MC(L' \tensor B) \xto{\Tha_{B \ast}} \MC(L'' \tensor B)
\xleftarrow{\Ph_{B\ast}} \MC(L \tensor B)$. We have
$P_{\MC} \sse \MC \bigl( (L' \tensor B) \times (L \tensor B) \bigr)
\cong \MC \bigl ( (L' \times L) \tensor B \bigr)$. From the
$\infty$-morphism \eqref{eq:H-morph}, we obtain the function
$H_{B\ast} \maps \MC(\ti{L} \tensor B) \to \MC \bigl ( (L' \times L)
\tensor B \bigr)$. From the commutativity of diagram
\eqref{diag:strict_pullback3} from Prop.\ \ref{prop:strict-pb}, we
deduce that $H_{B \ast}$ factors through the pullback:
\[
H_{B \ast} \maps \MC(\ti{L} \tensor B) \to P_{\MC}. 
\]
To complete the proof, we will construct an inverse to $H_{B \ast}$. Let $f \maps P \to (L' \tensor B)^0 \times (L \tensor B)^0$ denote the function
\[
f(a',a):= \bigl(a', a - (s \tensor \id_B) \circ \Tha_{B \ast}(a') \bigr).
\]
In other words, $f$ is the restriction of $J_{B \ast}$ to $P \sse (L' \times L)  \tensor B$,
where $J_{B \ast}$ is the polynomial function induced from the coalgebra map $J \maps \S(L' \times L) \to \S(L' \times L)$ defined in \eqref{eq:J}. Since $J$ is the inverse to $H$, in order to show that $f \vert_{P_{\MC}}$ is the inverse to $H_{B \ast}$, it is sufficient to verify that $\im f \vert_{P_{\MC}} \sse \MC(\ti{L} \tensor B)$.

We first observe that $\im f \sse ( \ti{L} \tensor B)^0$. Indeed, since $\Ph$ is strict, we have $\Ph_{B \ast} = \sPh \tensor \id_B$. Therefore, if $(a',a)\in (L' \tensor B)^0 \times (L \tensor B)^0$ such that $\Ph_{B \ast}(a)=\Tha_{B \ast}(a')$, then it follows that 
\[
a - (s \tensor \id_B) \circ \Tha_{B \ast}(a') \in  \ker (\sPh \tensor \id_B) = \ker (\sPh) \tensor B.
\]
Hence $f(a',a) \in ( \ti{L} \tensor B)^0$.

Now suppose $(\al',\al) \in P_{\MC}$. We will show that $\tcurv_{B}(f(\al',\al))=0$, where $\tcurv_{B}$ is the curvature function of $(\ti{L} \tensor B,\wti{Q}_B)$. From Eq.\ \ref{eq:exp2} we see that it is sufficient to prove that $\wti{Q}_B\bigl( \exp(f(\al',\al))-1 \bigr) =0$. As mentioned above, by definition, we have $f(\al',\al)= J_{B \ast}(\al',\al)$. Therefore Eq.\ \ref{eq:exp3} in Remark \ref{rmk:exp} implies that $\exp(f(\al',\al))-1 = J_B(\exp(\al',\al)-1).$ On the other hand, from the definition of $\wti{Q}$ in Eq.\ \ref{eq:deltilde}, we have $\wti{Q}_B = J_B \circ Q_{\times B} \circ H_{B}$, where $Q_{\times}$ is the $\sinf$-structure on the product $(L' \times L)$. Therefore, since $HJ=\id_{\S(L' \times L)}$, the equalities
\[
\wti{Q}_B\bigl( \exp(f(\al',\al))-1 \bigr) = \wti{Q}_BJ_B(\exp(\al',\al)-1)
= J_BQ_{\times B}(\exp(\al',\al)-1)  
\]
hold. Applying Eq.\ \ref{eq:exp2} to $Q_{\times B}(\exp(\al',\al)-1)$ then gives us 
\[
\wti{Q}_B\bigl( \exp(f(\al',\al))-1 \bigr) = J_B \Bigl (\exp(\al',\al)\,  \curv_{\times B}(\al',\al) \Bigr)
\]
where $\curv_{\times B}$ is the curvature function for $(L' \times L) \tensor B$. And finally, since $(\al',\al) \in P_{\MC}$, we have
$\curv'_B(\al')=\curv_B(\al)=0$, and hence $\curv_{\times B}(\al',\al)=0$.  Therefore, we conclude that $\wti{Q}_B\bigl( \exp(f(\al',\al))-1 \bigr) =0$, and this completes the proof.

\end{proof}

\section{Obstruction theory for lifting Maurer-Cartan elements} \label{sec:obstruct}
Let us say that two Maurer-Cartan elements $\al, \be \in \MC(L)$ in a $\cinf$-algebra $\LdQ{}$ are \df{equivalent} if $[\al]=[\be] \in \pi_0(\sMC(L))$. The results of the previous sections provide the means to completely characterize the general problem of lifting Maurer-Cartan elements through an $\infty$-morphism in $\cLie$. 
\begin{liftprob} \label{lp1}
Let $\Ph \maps \LdQ{} \to \LdQ{\prime}$ be a morphism in $\cLie$, and let $\al' \in \MC(L')$ be a Maurer-Cartan element in $L'$. Does there exist a Maurer-Cartan element $\al \in \MC(L)$ such that $\Ph_\ast(\al)$ is equivalent to $\al'$?
\end{liftprob}
It is conceptually useful to reformulate the problem purely in terms of simplicial Maurer-Cartan sets.
Let $f \maps X \to Y$ in $\Kan$ denote the map $\sMC(\Ph) \maps \sMC(L) \to \sMC(L')$. We seek a vertex $\al \in X_0$ such that $[f(\al)] = [\al'] \in \pi_0(Y)$. Recall that $Y$ is fibrant in $\sSet$, while every simplicial set is cofibrant. Hence,
a (left or right) homotopy between any two maps into $Y$ can be realized by \textit{any chosen path object} $Y \to Y^{I} \xto{(d_0,d_1)} Y \times Y$. (See, for example, \cite[Cor.\ II.1.9]{GJ}.) Therefore, the problem rephrased asks for maps $\al \maps \Del^0 \to X$, and $\eta \maps \Del^0 \to Y^I$ such that $f \al = d_0\eta$ and $ \al' = d_1 \eta$. Now consider the pullback diagram 
in $\Kan$:
\[
\begin{tikzdiag}{2}{2}
{
X \times_Y Y^I \& Y^I \\
X \& Y\\
};

\path[->,font=\scriptsize]
(m-1-1) edge node[auto] {$\pr_2$} (m-1-2) 
(m-2-1) edge node[auto] {$f$} (m-2-2)
;
\path[->>,font=\scriptsize]
(m-1-1) edge node[auto,swap] {$\pr_1$} node[sloped,above] {$\sim$}(m-2-1) 
(m-1-2) edge node[auto,swap] {$d_0$} node[sloped,above] {$\sim$}  (m-2-2)
;

\pbdiag
\end{tikzdiag}
\]
This is precisely the diagram \eqref{diag:factor} recalled in Sec.\ \ref{sec:factor} which gives the factorization $X \to X \times_Y Y^I \xto{p_f} Y$ of the morphism $f$ into a weak homotopy equivalence followed by the Kan fibration $p_f:=d_1 \pr_2$. Hence, solutions to the lifting problem are in one-to-one correspondence with vertices $(\al,\eta) \maps \Del^0 \to X \times_Y Y^I$ such that $p_f \cc (\al,\eta) = \al'$.   

Since the functor $\sMC \maps \cLie \to \Kan$ is exact, it preserves path objects, and pullbacks along fibrations. Hence, it preserves factorizations via path objects in the sense of Sec.\ \ref{sec:factor}. Let us choose a path object  $L' \to L^{\prime I} \xto{(d_0,d_1)} L' \times L'$ in $\cLie$ for $L'$. Then the above discussion implies that Lifting Problem \ref{lp1} is equivalent to:

\begin{liftprobbis}{lp1}[Equivalent]  \label{lp1p}
Let $\Ph \maps \LdQ{} \to \LdQ{\prime}$ be any morphism in $\cLie$, and $\al' \in \MC(L')$  a Maurer-Cartan element in $L'$. Denote by $P_{\Ph} \maps L \times_{L'} L^{\prime I} \fib L'$ the
fibration in $\cLie$  obtained by factoring $\Ph$ via the path object $L^{\prime I}$. Does there exist a Maurer-Cartan element \newline $\al \in \MC(L \times_{L'} L^{\prime I})$ such that  ${P_{\Ph}}_\ast(\al) = \al'$  in $\MC(L')$?
\end{liftprobbis}

\newcommand{\vps}{\psi}

\subsection{Reminder on cones and the associated graded complex}
Theorem \ref{thm:lp} below characterizes the obstruction to solving \eqref{lp1p} in terms of the degree zero cohomology of the associate graded mapping cone $\Cone (\Gr(\sPh)) \in \Ch$ for the tangent map $\ctan(\Ph)$ of $\Ph \maps \LdQ{} \to \LdQ{\prime}$. 

\begin{notation}\label{note:cone}
We denote by $(\Cone(\vps),\pa_\vps) \in \Ch$ the mapping cone of a morphism $\vps \maps (V,d) \to (V',d')$ of cochain complexes. Its underlying graded vector space in degree $n$ is  $\Cone(\vps)^n:= V^n \dsum {V'}^{n-1}$, and the differential  $\pa_\vps \maps \Cone(\vps)^n \to \Cone(\vps)^{n+1}$ is
\begin{equation} \label{eq:conediff}
\pa_\vps(v,v'):= \bigl( dv, \vps(v)-d'v' \bigr).
\end{equation}
\end{notation}

If $\vps \maps (V,d) \to (V',d')$ is a morphism in $\cCh$, then 
$(\Cone(\vps),\pa_\vps)$ is a complete cochain complex as well with respect to the filtration
$\cF_{n} \Cone(\vps) := \cF_nV \times \cF_n{\bs V'}$. 

Recall that the associated graded complex of $(V,d) \in \cCh$ is the cochain complex $(\Gr(V),d_{\Gr}) \in \Ch$ with underlying graded vector space
\begin{equation*} 
\Gr(V):= \bigoplus_{n > 1} \Gr_n(V), \quad \Gr_n(V):= \cF_{n-1}V/\cF_{n}V,
\end{equation*}
and differential $d_{\Gr} \vert_{\Gr_n(V)}:= \qq{d}{n} \maps  \cF_{n-1}V/\cF_{n}V \to \cF_{n-1}V/\cF_{n}V$.
When we wish to emphasize the induced bigrading on the cohomology of $(\Gr(V), d_{\Gr})$, we use the standard notation: 
\[
H^{p-n,n}(\Gr(V)):=H^p(\Gr_n(V),\qq{d}{n}).
\]  
Given a morphism $\vps \maps (V,d) \to (V',d')$ in $\cCh$, the induced map on the associated graded complexes $\Gr(\vps) \maps (\Gr(V),d_{\Gr}) \to (\Gr(V'),d'_{\Gr})$ is defined as $\Gr(\vps):= \bigoplus_n \Gr_n(\vps)$ where
\[
\Gr_n(\vps):= \cF_{n-1} \qq{\vps}{n} \maps \cF_{n-1}V/\cF_{n}V \to \cF_{n-1}V'/\cF_{n}V'.
\]
Hence, ``taking the associated graded'' commutes with the mapping cone construction, and so we make the tacit identification
\[
\Bigl (\Gr \bigl ( \Cone(\vps) \bigr), {\pa_\vps}_{\Gr} \Bigr ) = \Bigl (\Cone \bigl (\Gr (\vps) \bigr), \pa_{\Gr(\vps)} \Bigr).
\]

\subsection{Lifting problem for strict fibrations}

We first consider the special case of Lifting Problem \ref{lp1}, when the $\infty$-morphism is a strict fibration. In what follows, let $\vps=\Psi^1_1 \maps \LdQ{} \fib \LdQ{\prime}$ be a strict fibration in $\cLie$, and let $\al' \in \MC(L')$ be a  Maurer-Cartan element of $L'$. For each $m \geq 1$, we denote by $\qq{\al'}{m}:= \qq{p'}{m}(\al')$ the image of $\al'$ in the Maurer-Cartan set 
$\MC(\qf{L'}{m})$ of the quotient $(\qf{L'}{m}, \qq{d'}{m}, \qq{Q'}{m})$.

\begin{proposition} \label{prop:lpf}
Let $n > 1$ and suppose $\qq{\al}{n-1} \in \MC(\qf{L}{n-1})$ is a Maurer-Cartan element such that $\qq{\vps}{n-1}(\qq{\al}{n-1}) = \qq{\al'}{n-1}$. Let $a \in (\qf{L}{n})^0$ be any degree 0 element such that 
$\qpp{n-1}{a} = \qq{\al}{n-1}$. Then:
\begin{enumerate}

\item The element 
\[
w_n^\vps(a):= \bigl( \bcurv{n}(a), \qq{\vps}{n}(a) - \qq{\al'}{n} \bigr) \in
\bigl( \cF_{n-1}L/\cF_{n}L \bigr)^1 \dsum  \bigl( \cF_{n-1}L'/\cF_{n}L' \bigr)^0 
\]
is a degree 0 cocycle in the subcomplex $\Cone( \cF_{n-1} \qq{\vps}{n})
\sse \Cone( \Gr(\vps))$.

\item The class 
\[
W^\vps_n(\al'):=\bigl [w_n^\vps(a) \bigr] \in H^{-n,n}\bigl(\Cone \bigl (\Gr (\vps) \bigr) \bigr)
\]
is independent of the choice of  the element $a$ in the pre-image $\ba{p}_{(n-1)}^{\, -1}(\qq{\al}{n-1}) \sse (\qf{L}{n})^0$.
\item There exists a Maurer-Cartan element $\qq{\al}{n} \in \MC(\qf{L}{n})$ such that
\[
\qq{\vps}{n} \bigl(\qq{\al}{n} \bigr) = \qq{\al'}{n}, \quad \text{and} \quad  \qpp{n-1}(\qq{\al}{n}) = \qq{\al}{n-1}
\]  
if and only if $W^\vps_n(\al')=0$. 
\item The intersection of fibers 
\[
\ba{p}_{(n-1) \ast}^{\, -1}(\qq{\al}{n-1}) ~ \cap ~   \psi_{(n) \ast}^{\, -1} (\qq{\al'}{n}) 
~ \sse \MC(\qf{L}{n})
\]
is either empty or 
a torsor over the abelian group of degree 0 cocycles $Z^0( \ker \Gr_n(\vps))$.

\end{enumerate}
\end{proposition}

\begin{proof}
The proofs of the four statements are quite similar to those of Prop.\ \ref{prop:obstruct}.
\begin{enumerate}[leftmargin=15pt]

\item{ Statement 1 of Prop. \ref{prop:obstruct} implies that $\qq{d}{n}\bcurv{n}(a) =0$. 
Therefore, to verify $\pa_{\Gr(\vps)} w_n^\vps(a)=0$, we see from Eq.\ \ref{eq:conediff} that it suffices to show that 
\[
\qq{\vps}{n}(\bcurv{n}(a)) - \qq{d'}{n}y =0
\]
where $y:=\qq{\vps}{n}(a) -\qq{\al'}{n}$.
Since $\vps$ is a strict morphism in $\cLie$, we have $\vps_\ast = \vps$, and so Eq.\ \ref{eq:curv-morph} implies that  
\[
\qq{\vps}{n}(\bcurv{n}(a)) = \bcurvp{n}\bigl(\qq{\vps}{n}(a)\bigr).
\]
From the equality $\qpr{n} \cc \qq{\vps}{n} = \qq{\vps}{n-1} \cc \qpp{n-1}$ along with the hypotheses on $a$ and $\qq{\al}{n-1}$, it follows that $y \in \ker \qpr{n}=\cF_{n-1}L'/\cF_nL'$. Therefore, Eq.\ \ref{eq:bnd-curv} implies that
\[
\begin{split}
\bcurvp{n}\bigl(\qq{\vps}{n}(a)\bigr) = \bcurvp{n}( \qq{\al'}{n} + y) =  \bcurvp{n}(\qq{\al'}{n}) + \qq{d'}{n}y.
\end{split}
\] 
Since $\qq{\al}{n}$ is Maurer-Cartan, we conclude that $\bcurvp{n}\bigl(\qq{\vps}{n}(a)\bigr) =\qq{d'}{n}y$. }

\item{Suppose $\ti{a} \in  \ba{p}_{(n-1)}^{\, -1}(\qq{\al}{n-1})$. Then $x:=\ti{a} -a \in \ker
\ba{p}_{(n-1)} = \cF_{n-1}L/\cF_nL$. Since the $\cinf$-algebra $\qf{L}{n}$ is bounded at $n$, Eq.\ \ref{eq:bnd-curv} of Lemma \ref{lem:bnd-curv} implies that $\bcurv{n}(\ti{a})      
- \bcurv{n}(a) = \qq{d}{n}x$. Hence, we have $w_n^\vps(\ti{a}) -  w_n^\vps(a) = \pa_{\Gr(\vps)} (x,0)$.
}

\item {
If $\qq{\al}{n} \in \MC(\qf{L}{n})$ is a Maurer-Cartan element such that
\[
\qq{\vps}{n} \bigl(\qq{\al}{n} \bigr) = \qq{\al'}{n}, \quad \qpp{n-1}(\qq{\al}{n}) = \qq{\al}{n-1},
\]  
then $w_n^\vps(\qq{\al}{n})=0$. Hence, by (2) above, we have $[w_n^\vps(a) \bigr] = 
[w_n^\vps(\qq{\al}{n})]=0$. Conversely, suppose there exists a degree $-1$ element
\[
(y,y') \in \bigl( \cF_{n-1}L/\cF_{n}L \bigr)^0 \dsum  \bigl( \cF_{n-1}L'/\cF_{n}L' \bigr)^{-1} 
\] 
such that $w_n^\vps(a) =  \pa_{\Gr(\vps)}(y,y')$, i.e.\
\begin{equation} \label{eq:triv}
\bcurv{n}(a) = \qq{d}{n}y, \qquad \qq{\vps}{n}(a) - \qq{\al'}{n}  = \qq{\vps}{n}(y) -\qq{d'}{n}y'.  
\end{equation}
Since $\vps$ is a fibration, $\cF_{n-1}\qq{\vps}{n}$ is surjective. Therefore, there exists $\ti{y} \in (\cF_{n-1}L/\cF_{n}L \bigr)^{-1}$ such that $\qq{\vps}{n}(\ti{y})=y'$.
Define $\qq{\al}{n}:= a - y + \qq{d}{n} \ti{y} \in (\qf{L}{n})^0$. Then $\qpp{n-1}(\qq{\al}{n}) = \qq{\al}{n-1}$, since $y$ and  $\ti{y}$ are elements of $\ker \qpp{n-1}$. Furthermore, since $y$ and $ \qq{d}{n}\ti{y} \in 
(\cF_{n-1}L/\cF_{n}L \bigr)^{0}$, Eq.\ \ref{eq:bnd-curv} of Lemma \ref{lem:bnd-curv} along with Eq.\ \ref{eq:triv} above imply that
\[
\bcurv{n}(\qq{\al}{n})= \bcurv{n}(a) + \qq{d}{n}(-y + \qq{d}{n}\ti{y})=0.
\]
Finally, since $\qq{\vps}{n}(\qq{\al}{n}) = \qq{\vps}{n}(a) - \qq{\vps}{n}(y) + \qq{d'}{n}\vps(\ti{y})$, Eq.\ \ref{eq:triv} implies that  $\qq{\vps}{n}(\qq{\al}{n}) = 
\qq{\al'}{n}$.
}

\item{ 
We proceed exactly as in the proof of statement 4 of Prop.\ \ref{prop:obstruct}. Suppose $\qq{\al}{n}, \qq{\ti{\al}}{n} \in \MC(\qf{L}{n})$ are both Maurer-Cartan elements such that 
$\qpp{n-1}_\ast(\qq{\al}{n})=\qpp{n-1}_\ast(\qq{\ti{\al}}{n}) = \qq{\al}{n-1}$, and
${\qq{\vps}{n}}_\ast(\qq{\al}{n})={\qq{\vps}{n}}_\ast(\qq{\ti{\al}}{n}) = \qq{\al'}{n}$. Then their difference $\eta:=\qq{\al}{n} - \qq{\ti{\al}}{n}$ lies in $(\cF_{n-1}L/\cF_{n}L \bigr)^{0}$ and in $\ker \qq{\vps}{n}$. Hence, $\eta$ is a degree zero element in $\ker( \Gr_n(\vps))$. The equality 
$\bcurv{n}(\qq{\al}{n}) = \bcurv{n}(\qq{\ti{\al}}{n})=0$, along with Eq.\ \ref{eq:bnd-curv} imply that 
$\eta$ is a cocycle. 

Reversing this argument shows that adding any such cocycle $\eta$ to a Maurer-Cartan element $\qq{\al}{n}$ satisfying  $\qpp{n-1}_\ast(\qq{\al}{n}) = \qq{\al}{n-1}$ and 
${\qq{\vps}{n}}_\ast(\qq{\al}{n})=\qq{\al'}{n}$, produces a new  
Maurer-Cartan element $\qq{\ti{\al}}{n}:= \qq{\al}{n} + \eta$ for which the same equalities hold.
}
\end{enumerate}
\itemqed
\end{proof}

\begin{remark}\label{rmk:lpf}
For each $n > 1$, the cohomology class $W^\vps_{n+1}(\al')$ is only well-defined if
$W^\vps_{n}(\al') =0$. The first obstruction class $W^\vps_2(\al')$ for the $n=2$ base case  
is always well-defined since $\qf{L}{1} = \qf{L'}{1}=0$.

\end{remark}

\subsection{Lifting problem: the general case}
As in the statement of Lifting Problem \eqref{lp1p}, let $\Ph \maps \LdQ{} \to \LdQ{\prime}$ be a morphism with $\vph:= \ctan(\Ph)= \Ph^1_1$ and $\al' \in \MC(L')$ a Maurer-Cartan element. Fix, once and for all, a path object\footnote{From here on, we omit the differentials and codifferentials in order to keep the notation manageable.} ${L'}^{I} \in \cLie$ for $L'$. Let 
\[
L \xrightarrow[{}^{\sim}]{i} L \times_{L'} L^{\prime I} \xfib{P_{\Ph}} L'
\]
denote the corresponding factorization in $\cLie$ of $\Ph$,  where $i$ is a weak equivalence and $P_{\Ph}$ is a fibration. Let $p_{\Phi}:= (P_\Ph)^1_1$ denote the linear term of $P_{\Ph}$. 
Since the tangent map $\ctan(i)$ is a weak equivalence in $\cCh$, it induces a quasi-isomorphism between the associated graded mapping cones
\begin{equation} \label{eq:qiso-cone}
i_{\sharp} \maps \Cone \bigl (\Gr (\vph) \bigr) \xto{\sim} \Cone \bigl (\Gr (p_{\Phi}) \bigr).
\end{equation}  
By Prop.\ \ref{prop:strict_fib}, there is a commuting diagram in $\cLie$ 
\[
\begin{tikzdiag}{2}{1}
{
\wti{L \times_{L'} L^{\prime I}} \&  \& L \times_{L'} L^{\prime I} \\
\& L'  \& \\
};
\path[->,font=\scriptsize]
(m-1-1) edge node[auto] {$\cong$} node[auto,swap] {$$} (m-1-3)
;
\path[->>,font=\scriptsize]
(m-1-1) edge node[auto,swap] {$\wti{P_{\Ph}}$} (m-2-2)
(m-1-3) edge node[auto] {${P_{\Ph}}$} (m-2-2)
 ;
\end{tikzdiag}
\]
where $\wti{P_{\Ph}}$ is a strict fibration whose linear term $\wti{p_\Ph}$ is equal to the linear term 
${p_\Ph}$ of $P_\Ph$. Moreover, as mentioned in Rmk.\ \ref{rmk:strict_fib}, we have an identification of tangent complexes 
\[
\ctan \Bigl( \wti{L \times_{L'} L^{\prime I}} \xto{\wti{P_{\Ph}}} L'  \Bigr)  = 
\ctan \Bigl( {L \times_{L'} L^{\prime I}} \xto{{P_{\Ph}}} L'  \Bigr).  
\]  
Therefore, we may rewrite the quasi-isomorphism \eqref{eq:qiso-cone} as
\begin{equation} \label{eq:qiso-cone2}
i_{\sharp} \maps \Cone \bigl (\Gr (\vph) \bigr) \xto{\sim} \Cone \bigl (\Gr (\wti{p_{\Phi}}) \bigr).
\end{equation}  
Since $\wti{P_{\Ph}} \maps \wti{L \times_{L'} L^{\prime I}} \to L'$ is a strict fibration, Prop.\ \ref{prop:lpf} implies that the problem of lifting $\al' \in \MC(L')$ through $\wti{P_{\Ph}}$ 
is characterized by the sequence of cohomology classes
\[
W^{p_{\Phi}}_n(\al') \in H^{-n,n}\bigl(\Cone \bigl (\Gr \wti{p_{\Phi}} \bigr) \bigr) \quad  \forall n > 1.
\]  
whenever they are well-defined (Rmk.\ \ref{rmk:lpf}). 

Returning to our original morphism $\Ph \maps \LdQ{} \to \LdQ{\prime}$, for each $n > 1$,
we define the \df{{\boldmath $n$}th order obstruction class} to be the unique cohomology class
in the associated graded mapping cone
\[
\omega_n^{\Ph}(\al') \in H^{-n,n}\bigl(\Cone \bigl (\Gr (\vph) \bigr) \bigr)
\]   
whose image $i_{\sharp}(\omega_n^{\Ph}(\al'))$ under the quasi-isomorphism \eqref{eq:qiso-cone2} 
is the class $W^{p_{\Phi}}_n(\al')$. Thus, in light of Prop.\ \ref{prop:lpf} and the equivalence of Lifting Problems
\eqref{lp1} and \eqref{lp1p}, we have proved the following result: 
\begin{theorem}\label{thm:lp}
Let $\Ph \maps \LdQ{} \to \LdQ{\prime}$ be a morphism in $\cLie$, and let $\al' \in \MC(L')$ be a Maurer-Cartan element in $L'$. Then there exists a Maurer-Cartan element $\al \in \MC(L)$ such that $\Ph_\ast(\al)$ is equivalent to $\al'$ if and only if the obstruction classes
$\omega_n^{\Ph}(\al') \in H^{-n,n}\bigl(\Cone \bigl (\Gr (\vph) \bigr) \bigr)$ vanish for all $n > 1$.

\end{theorem}

\end{document}